\documentclass{amsart}
\pdfoutput=1
\usepackage{hyperref}
\usepackage{epsfig,color}
\usepackage{thm-restate}
\usepackage{enumitem}

\newtheorem{lemma}{Lemma}[section]
\newtheorem{theorem}[lemma]{Theorem}
\newtheorem{prop}[lemma]{Proposition}
\newtheorem{cor}[lemma]{Corollary}
\newtheorem{conj}[lemma]{Conjecture}

\theoremstyle{definition}
\newtheorem{defn}[lemma]{Definition}

\newtheorem{example}[lemma]{Example}
\newtheorem{rem}[lemma]{Remark}

\newtheorem{claim}{Claim}[lemma]
\newtheorem{assume}[lemma]{Assumption}
\newtheorem{case}{Case}

\newtheorem*{claim*}{Claim}
\newtheorem*{vtc}{Very translating condition}
\newtheorem*{theorem*}{Theorem}

\newcommand{\mc}[1]{\mathcal{#1}}
\newcommand{\llangle}{\langle\negthinspace\langle}
\newcommand{\rrangle}{\rangle\negthinspace\rangle}

\newcommand{\calC} {\ensuremath {\mathcal{C}}}
\newcommand{\calH} {\ensuremath {\mathcal{H}}}
\newcommand{\sepH} {\ensuremath {\widehat{\mathcal{H}}}}

\newcommand{\bR}{\mathbb{R}}
\newcommand{\bZ}{\mathbb{Z}}
\newcommand{\bH}{\mathbb{H}}
\newcommand{\bN}{\mathbb{N}}
\newcommand{\co}{\colon}
\newcommand{\diam}{\mathrm{diam}}
\newcommand{\depth}{\mathrm{depth}}

\begin{document}

\title{Boundaries of Dehn fillings} 

\author{Daniel Groves}
\address{Department of Mathematics, Statistics and Computer Science, University of Illinois at Chicago, 322 Science and Engineering Offices (M/C 249), 851 S. Morgan St., Chicago IL 60607, USA}
\email{groves@math.uic.edu}

\author{Jason Fox Manning}
\address{Department of Mathematics, 310 Malott Hall, Cornell University, Ithaca, NY 14853}
\email{jfmanning@math.cornell.edu}

\author{Alessandro Sisto}
\address{Department of Mathematics, ETH Zurich, 8092 Zurich, Switzerland}
\email{sisto@math.ethz.ch}

\maketitle

\begin{abstract}
  We begin an investigation into the behavior of Bowditch and Gromov boundaries under the operation of Dehn filling.  In particular we show many Dehn fillings of a toral relatively hyperbolic group with $2$--sphere boundary are hyperbolic with $2$--sphere boundary.  As an application, we show that the Cannon conjecture implies a relatively hyperbolic version of the Cannon conjecture.
\end{abstract}

\setcounter{tocdepth}{1}
\tableofcontents
\section{Introduction}

One of the central problems in geometric group theory and low-dimensional topology is the Cannon Conjecture (see \cite[Conjecture 11.34]{Cannon91}, \cite[Conjecture 5.1]{CannonSwenson}), which states that a hyperbolic group whose (Gromov) boundary is a $2$-sphere is virtually a Kleinian group.  By a result of Bowditch \cite{Bowditch98} hyperbolic groups can be characterized in terms of topological properties of their action on the boundary.  The Cannon Conjecture is that (in case the boundary is $S^2$) this topological action is in fact conjugate to an action by M\"obius transformations.  Relatively hyperbolic groups are a natural generalization of hyperbolic groups which are intended (among other things) to generalize the situation of the fundamental group of a finite-volume hyperbolic $n$-manifold acting on $\bH^n$.  

A relatively hyperbolic group pair $(G,\mc{P})$ has associated with it a natural compact space $\partial(G,\mc{P})$ called the \emph{Bowditch boundary} \cite[$\S$9]{bowditch12} on which it acts as a geometrically finite convergence group, so that every parabolic fixed point has stabilizer conjugate to a unique element of $\mc{P}$.  The motivating example is when $G<SO(n,1)$ is a geometrically finite Kleinian group and $\mc{P}$ a collection of conjugacy representatives of maximal parabolic subgroups.  In this case the Bowditch boundary coincides with the limit set.   A result of Yaman \cite{Yaman} characterizes relatively hyperbolic groups in terms of their action on the Bowditch boundary.

It is natural to wonder whether a relatively hyperbolic group whose Bowditch boundary is a $2$-sphere is virtually Kleinian.  In fact, both the Cannon Conjecture and this relative version are special cases of a much more general conjecture of Martin and Skora \cite[Conjecture 6.1]{MartinSkora}.  One of the main results of this paper (see Corollary \ref{cor:relative Cannon} below) is to prove that the relative version of the Cannon Conjecture follows from the absolute version.

If the peripheral subgroups $\mc{P}$ of a relatively hyperbolic group pair $(G,\mc{P})$ are themselves hyperbolic, then so is $G$, and it therefore acts as a uniform convergence group on its \emph{Gromov boundary} $\partial G$.  (The relationship between these boundaries is explained in \cite{Tran:comparison_boundaries}, see also \cite{Ger:Floyd,GerPot:Floyd,MOY:blowingupanddown,ManningBoundary}.)  In \cite{osin:peripheral,rhds} (cf. \cite{DGO}), the operation of \emph{group theoretic Dehn filling} is developed, and is shown to satisfy a coarse analog of Thurston's Hyperbolic Dehn Surgery Theorem \cite[Section 5.8]{thurston:notes}.  This is to say that many ``Dehn fillings'' of a relatively hyperbolic group pair are themselves relatively hyperbolic.

 Relatively hyperbolic Dehn filling has found many important applications, including in the proof of the virtual Haken conjecture \cite{VH} and the solution of the isomorphism problem in a large class of relatively hyperbolic groups \cite{DG15}.

In the classical setting one begins with a relatively hyperbolic group pair $(G,\mc{P})$ whose Bowditch boundary is a $2$--sphere and ends with a hyperbolic group $\overline{G}$ whose Gromov boundary is again a $2$--sphere. 
On the other hand, examples of CAT$(-1)$ fillings of high-dimensional manifolds \cite{moshersageev,FujMan10} suggest that group theoretic Dehn filling often produces a group $\overline{G}$ whose boundary is much more complicated than that of $(G,\mc{P})$, but which nonetheless admits a fairly explicit description.
 One purpose of this paper is to begin an investigation of whether these results are special to fillings of manifolds, or are reflective of more general phenomena.  To this end, we obtain a description (Theorems \ref{thm:trunc}, \ref{thm:existsvisual} and \ref{thm:linconn}) of the boundary of a Dehn filling as a certain kind of limit of quotients of subsets of the original boundary by discrete groups.  The following result is contained in Theorems \ref{thm:trunc} and \ref{thm:existsvisual} (see Definition \ref{wghlimit} for the definition of weak Gromov-Hausdorff convergence). For simplicity, we state it in the case of one peripheral subgroup, that is $\mc{P}=\{P\}$.
 \begin{theorem}
   Let $\overline{G} = G/\llangle N\rrangle$ be a sufficiently long hyperbolic filling of the relatively hyperbolic pair $(G,\{P\})$, with $N\triangleleft P$ infinite.  Then there is a sequence of Gromov hyperbolic spaces $X_i$ whose boundaries $\partial X_i$ weakly Gromov-Hausdorff converge to $\partial\overline{G}$, if we endow all these boundaries with suitable metrics.  Moreover there is an exhaustion $K_1<K_2<\cdots$ of $\ker(G\to \overline{G})$ so that each $\partial X_i$ can be identified with 
\[ \left(\left(\partial (G,\{P\})\setminus \Lambda(K_i) \right)/ K_i\right)  \cup \mc{F},\]
 where $\mc{F}$ is a union of finitely many copies of $\partial (P/N)$.
 \end{theorem}
  This gives a new way to prove statements about boundaries of Dehn fillings, by proving a statement about the approximating $\partial X_i$, and showing it persists in the limit.  Theorem \ref{thm:linconn} states that under some additional assumptions, we may assume all these metrics are uniformly linearly connected, which helps control the limit.

Our main application is the following statement, which says roughly that sufficiently long Dehn fillings of relatively hyperbolic groups with $2$--sphere boundary must have $2$--sphere boundary.
\begin{restatable}{theorem}{boundarysphere}\label{thm:boundary sphere}
Let $G$ be a group, and $\mc{P}=\{P_1,\ldots,P_n\}$ a collection of free abelian subgroups.  Suppose that $(G,\mc{P})$ is relatively hyperbolic, and that $\partial(G,\mc{P})$ is a $2$--sphere.
  
Then for all sufficiently long fillings $G\to \overline{G} = G(N_1,\ldots,N_n)$ with $P_i/N_i$ virtually infinite cyclic for each $i$, we have that $\overline{G}$ is hyperbolic with $\partial{\overline{G}}$ homeomorphic to $S^2$.
\end{restatable}

Note that if $\partial(G,\mc{P})$ is a $2$--sphere, then any parabolic acts properly cocompactly on $\bR^2$.  If we are assuming it is free abelian, it must therefore be $\bZ^2$.  If we didn't assume abelian, we might have to worry about higher genus surface groups as peripheral groups.  These higher genus surface groups being hyperbolic groups, we should exclude them from the peripheral structure to get a boundary which is a Sierpinski carpet.  Conjecturally, a hyperbolic group with Sierpinski carpet boundary is virtually Kleinian.  Kapovich and Kleiner \cite{KK} prove that this would follow from the Cannon Conjecture.

One can make a relative version of the Cannon Conjecture as follows (cf. \cite[Problem 57]{kapovichproblems}):
\begin{conj}(Relative Cannon Conjecture)
  Let $(G,\mc{P})$ relatively hyperbolic group with $\partial(G,\mc{P})\cong S^2$ and all elements of $\mc{P}$ free abelian.  Then $G$ is Kleinian.
\end{conj}
We remark that the usual Cannon conjecture says `virtually Kleinian' because a non-elementary hyperbolic group may not act faithfully on its boundary; there may be a finite kernel.  However, under the assumption that the parabolic subgroups of a nonelementary relatively hyperbolic group (with nontrivial peripheral structure) are free abelian, there are no nontrivial finite normal subgroups, and so `Kleinian' is the expected conclusion.

We have the following corollary of Theorem \ref{thm:boundary sphere}; see Section \ref{sec:corollary} for the proof.
\begin{cor}\label{cor:relative Cannon}
  The Cannon Conjecture implies the Relative Cannon Conjecture.
\end{cor}
This resolves \cite[Problem 60]{kapovichproblems}, though we do not proceed via Kapovich's suggested method of proof.

\subsection{Sketch proof of Theorem \ref{thm:boundary sphere}}
We must somehow reconstruct $\partial \overline{G}$ from information about $\partial(G,\mc{P})$.  It is a result of Dahmani--Guirardel--Osin that $K = \ker(G\to \overline{G})$ is (for a sufficiently long filling) freely generated by parabolic subgroups \cite{DGO}.  Associated to $(G,\mc{P})$ is a proper, Gromov hyperbolic space (the \emph{combinatorial cusped space}) $X = X(G,\mc{P})$ on which $G$ acts geometrically finitely (cocompactly away from horoballs); the Gromov boundary of this space is equivariantly homeomorphic to $\partial(G,\mc{P})$.

In Section \ref{sec:spiderweb} we develop an analog in the cusped space of the ``windmills'' technology of \cite{DGO} to obtain an exhaustion of $K$ by free products of \emph{finitely many} parabolic subgroups.  Our replacements for windmills are called \emph{spiderwebs} -- these form an exhaustion $W_1\subset W_2\cdots $ of the cusped space by quasiconvex subsets, each of which is acted on geometrically finitely by a finitely generated subgroup $K_n$ of $K$.  The ``partial quotients'' $X/K_n$ approximate a cusped space $\overline{X} = X/K$ for the relatively hyperbolic pair $(\overline{G},\overline{\mc{P}})$.  But in the situation of interest $\overline{G}$ is itself hyperbolic, so we need approximations to $\partial\overline{G}$, not to $\partial(\overline{G},\overline{\mc{P}})$.  Such approximations are obtained from $X/K_n$ by removing finitely many images of (deep) horoballs of $X$.  We must take some care to ensure that these truncated partial quotients are uniformly hyperbolic over all $n$.  With even more care, we are able to show these boundaries are uniformly linearly connected over all $n$ (Theorem \ref{thm:linconn}), and that they have nice descriptions in terms of $\partial(G,\mc{P})$ (Theorem \ref{thm:trunc}).

Once this is ensured, we have a sequence of spaces which converge in the pointed Gromov--Hausdorff topology to a $\overline{G}$--cocompact space.   Their boundaries therefore converge (in a sense described in Section \ref{sec:wgh}) to the boundary of $\overline{G}$ (Theorem \ref{thm:existsvisual}).

The above results apply more generally when $\overline{G}$ is hyperbolic, and the result of a long filling of a relatively hyperbolic pair $(G,\mc{P})$, and in fact we state versions in the setting of the Bowditch boundary of $\partial(\overline{G},\overline{P})$ as Theorems \ref{relversion:trunc}--\ref{rel:linconn}.  The proofs of these relative versions are strictly easier than those of Theorem \ref{thm:trunc}--\ref{thm:linconn}, though we do not provide the relative proofs in this paper.

In Section \ref{sec:spheres} we specialize to $\partial(G,\mc{P})\cong S^2$, and $P_i/K_i$ virtually cyclic.  In this case we can show the approximating boundaries are spheres by a homological argument.  

We now sketch the argument that the boundary is planar.  Results from \cite{GM-splittings} show that the boundary is a Peano continuum\footnote{meaning a connected, locally connected, compact metrizable space} without local cut points.
We then invoke a characterization of Claytor \cite{Claytor34}, which says that a Peano continuum without cut points is planar if and only if it contains no non-planar graph.  
An adaptation of a lemma of Ivanov (Lemma \ref{ivanovlemma}) shows that if $\partial\overline{G}$ contained such a graph, then so would all but finitely many of the approximating boundaries.  Since they are spheres, they do not.

Since $\partial\overline{G}$ is planar, connected, and has no local cut points, a result of Kapovich and Kleiner \cite[Theorem 4]{KK} implies that it is either $S^2$ or a Sierpinski Carpet.  In Subsection \ref{ss:not Sierpinski} we rule out the Sierpinski Carpet.

\subsection{Outline}
Section \ref{s:Prelim} contains background, notation and preliminary results; the reader can skim it and refer back to it when needed.

In Section \ref{sec:wgh} we introduce the notion of weak Gromov--Hausdorff convergence, which plays an important role in our description of the boundary of a Dehn filled group.

In Section \ref{sec:spiderweb} we introduce spiderwebs, a variation of the windmills from \cite{DGO}. We cannot use windmills directly for our purposes, but our construction is very similar to that in \cite{DGO}.

Finally, the main contributions of this paper start with Sections \ref{s:technical} and \ref{sec:proofs}, where we state and prove our main results about general Dehn filling. As discussed above, we describe the boundary of a Dehn filled group as a certain weak Gromov--Hausdorff limit of spaces, each the boundary of a certain hyperbolic space, whose topology we have control on.

Starting with Section \ref{sec:spheres}, we focus on the setup of Theorem \ref{thm:boundary sphere}, that is to say we consider fillings of a relatively hyperbolic pair whose Bowditch boundary is a $2$--sphere. First of all, we exploit the general description of the approximating boundaries to show that in that situation they are all spheres.

Section \ref{ss:not Sierpinski} contains the last missing piece of the proof of Theorem \ref{thm:boundary sphere}: We prove that a weak Gromov--Hausdorff limit of simply connected spaces is $\epsilon$-simply-connected for every $\epsilon>0$, therefore proving that the Sierpinski carpet cannot be a limit of spheres.

In Section \ref{s:proof of boundary sphere} we prove Theorem \ref{thm:boundary sphere}, which at that point only requires putting together various pieces.

In Section \ref{sec:corollary}, we prove Corollary \ref{cor:relative Cannon}, which requires arguments about limits of representations in $\mathrm{Isom}(\mathbb H^3)$.

Finally, in Appendix \ref{app:technical} we record some technical results which are surely well known to experts but for which we do not know of a reference in the literature.

\subsection{Acknowledgments}  
The authors would like to thank Peter Ha\"{\i}ssinsky and Genevieve Walsh for useful conversations, and an anonymous referee for several helpful comments.

This material is based upon work supported by the National Science Foundation under grant No. DMS-1440140 while the second and third authors were in residence at the Mathematical Sciences Research Institute in Berkeley, California, during the Fall 2016 semester.  The first author is partially supported by a grant from the Simons Foundation (\#342049 to Daniel Groves) and by NSF grant DMS-1507076. The second author is partially supported by NSF grant DMS-1462263.

\section{Preliminaries} \label{s:Prelim}
For a point $p$ of a metric space $(M,d)$, write $S_R(p)$ for $\{x\in M\mid d(x,p)= R\}$, and $B_R(p)$ for $\{x\in M\mid d(x,p)\leq R\}$.  If $M$ is a geodesic space we write $[x,y]$ for a choice of geodesic from $x$ to $y$ in $M$.

In a geodesic space $(Z,d)$, every geodesic triangle $\Delta$ comes with a surjective map to a possibly degenerate \emph{comparison tripod} $T_\Delta$, which is isometric on each side of the triangle, and so the vertices map to feet of the tripod.  If the vertices of the triangle are $x$, $y$, and $z$, the leg corresponding to $x$ has length $$(y\mid z)_x \co = \frac{1}{2}(d(y,x)+d(z,x)-d(y,z)),$$ also known as the \emph{Gromov product of $y$ and $z$ with respect to $x$}.

For $\delta>0$, the geodesic space $Z$ is a \emph{$\delta$--hyperbolic space} if all geodesic triangles in $Z$ are {\em $\delta$--thin}, in the sense that the map to the comparison tripod has fibers of diameter at most $\delta$.  A space is \emph{Gromov hyperbolic} if it is $\delta$--hyperbolic for some $\delta$.  See \cite[III.H]{BH} for more details, and the relationship with other definitions.  

A Gromov hyperbolic space $Z$ has a \emph{boundary at infinity} or \emph{Gromov boundary} $\partial Z$, which can be defined in terms of sequences of points.  Namely, a sequence $\{x_i\}$ \emph{converges to infinity} if $\lim\limits_{i,j\to \infty}(x_i\mid x_j)_{p}=\infty$ for some (or equivalently every) basepoint $p$.  Two sequences $\{x_i\}$ and $\{y_i\}$ are \emph{equivalent} if $\lim\limits_{i,j\to\infty}(x_i\mid y_j)_p = \infty$, and $\partial Z$ is defined to be the set of equivalence classes of sequences which converge to infinity.  If the equivalence class of $\{x_i\}$ is $\xi$, we write $\{x_i\}\to \xi$.
In a proper Gromov hyperbolic space, $\partial Z$ can also be defined as equivalence classes of geodesic rays, where two rays are counted as equivalent if they have images which are finite Hausdorff distance apart.  All the spaces we consider are proper.

For $p\in Z$ and $\xi, \upsilon\in Z\cup \partial Z$, the Gromov product is extended as follows:
\[ (\xi\mid \upsilon)_p = \sup\left.\left\{ \liminf_{i,j\to \infty} (x_i\mid y_j)_p\right| \{x_i\}\to\xi, \{y_i\}\to \upsilon\right\}.\]
  
It is a standard fact, see e.g. \cite[Lemma 5.6]{Vai:hyperbolic} or \cite[III.H.3.17.(5)]{BH}, that, up to a small error, one can compute the Gromov product using any given representative sequences, meaning that if $\{x_i\}\to\xi, \{y_i\}\to \upsilon$ then $\liminf (x_i\mid y_j)_p$ is within $2\delta$ of $(\xi\mid \upsilon)_p$.

The following observation is \cite[III.H.3.17.(3)]{BH}:
\begin{lemma}\label{lem:avoidlimsup}
  Let $Z$ be Gromov hyperbolic, and let $p\in Z$.  For any $\xi,\upsilon$ in $Z\cup\partial Z$, there are sequences $\{x_i\}\to \xi$, and $\{y_i\}\to \upsilon$ so that $\lim\limits_{n\to\infty}(x_n\mid y_n)_p = (\xi\mid \upsilon)_p$.
\end{lemma}

We also consider the Gromov product of geodesic rays $(\alpha|\beta)_p$ with respect to their common starting point $p$, which we define to be
$$(\alpha|\beta)_p= \liminf_{s,t\to \infty} (\alpha(s)\mid \beta(t))_p.$$

\subsection{Visual metrics on the boundary of a Gromov hyperbolic space} \label{ss:visual metrics}
For any given parameter $\epsilon >0$ and basepoint $p\in X$, the function $(\eta,\xi)\mapsto e^{-\epsilon(\eta\mid \xi)_p}$ behaves somewhat like a metric on $\partial X$, though it may not satisfy the triangle inequality.  It does makes sense to ask whether $e^{-\epsilon(\cdot\mid\cdot)_p}$ is bilipschitz or quasi-isometric to some metric on $\partial X$. 

We recall the definition:
\begin{defn}\label{def:visualmetric}
  Let $Z$ be a Gromov hyperbolic space, with basepoint $w$.  A \emph{visual metric on $\partial Z$, based at $w$, with parameters $\epsilon,\kappa$} is a metric $\rho(\cdot,\cdot)$ which is $\kappa$--bilipschitz to $e^{-\epsilon(\cdot|\cdot)_w}$.
\end{defn}

From \cite[III.H.3.21]{BH} one can fairly readily deduce the following:
\begin{prop}\cite[III.H.3.21]{BH}\label{findvisualmetric}
  Let $\delta>0$.  Then for all positive $\epsilon \leq \frac{1}{6\delta}$ there is a $\kappa=\kappa(\epsilon,\delta)\geq 1$ with $\lim\limits_{\epsilon\to 0}\kappa(\epsilon,\delta) = 1$ so that:

  If $Z$ is a $\delta$--hyperbolic space and $p\in Z$, then $\partial Z$ has a visual metric  based at $p$ with parameters $\epsilon,\kappa$.
\end{prop}

Visual metrics are hardly ever length metrics, and in fact hardly ever admit rectifiable paths. However, the notion of linear connectedness is a useful ``replacement'' for the notion of length metric.

\begin{defn}
 Let $L\geq 1$.  A metric space $M$ is \emph{$L$--linearly connected} if every pair of points $x,y\in M$ is contained in a connected subset $J$ of diameter at most $L\cdot d(x,y)$.  We say $M$ is \emph{linearly connected} if it is $L$--linearly connected for some $L$.
\end{defn}

\begin{rem}
As observed, for example, in the introduction of \cite{MackayQuasiArcs}, if $M$ is compact then up to increasing $L$ by an arbitrarily small amount we can assume that $J$ is an arc.  We frequently make this assumption in the rest of the paper.
\end{rem}

A homeomorphism $f\co X\to Y$ of metric spaces is a \emph{quasi-symmetry} if there is a homeomorphism $\eta\co [0,\infty)\to [0,\infty)$ so that 
\[ \frac{d(f(x),f(y))}{d(f(x),f(z))}\leq \eta\left(\frac{d(x,y)}{d(x,z)}\right) \]
for all triples of distinct points $x,y,z\in X$.  The spaces $X$ and $Y$ are then said to be \emph{quasi-symmetric}. All visual metrics on the boundary of a given hyperbolic space are quasi-symmetric to each other.   Observe:
\begin{lemma}\label{lem:LCQS}
  If $X$ is linearly connected, then so is any space quasi-symmetric to $X$.
\end{lemma}

\subsection{The cusped space associated to a relatively hyperbolic pair}
In this section we associate a metric graph (the \emph{(combinatorial) cusped space}) to a relatively hyperbolic pair, and fix notation for various subsets of it.

\begin{defn} \label{d:comb horo}
Let $\Gamma$ be a graph, endowed with the metric that gives each edge length $1$.  The \emph{combinatorial horoball based on $\Gamma$} is the metric graph $\calH(\Gamma)$ whose vertex set is $\Gamma^{(0)}\times \bZ_{\geq 0}$, and with two types of edges:
\begin{enumerate}
\item A \emph{vertical} edge of length $1$ from $(v,n)$ to $(v,n+1)$ for any $v\in \Gamma^{(0)}$ and any $n\geq 0$;
\item\label{eq:horiz edges} For $k > 0$, if $v$ and $w$ are vertices of $\Gamma$ so that $0 < d_\Gamma(v,w) \le 2^k$ then there is a single {\em horizontal} edge of length $1$ joining $(v,k)$ to $(w,k)$.
\end{enumerate}
  Define the \emph{depth} of a vertex $D(v,n)=n$  and extend the depth function affinely over edges.

  The inverse image $D^{-1}(n)$ for $n$ an integer is called the \emph{horosphere at depth $n$}.  This is a graph whose vertices are in bijection with those of $\Gamma$.  The distance in $D^{-1}(n)$ between two vertices $(v,n)$ and $(w,n)$ is $\lceil 2^{-n} d_\Gamma (v,w) \rceil$.  

  If $\calH = \calH(\Gamma)$ for some $\Gamma$, and $I$ is a nondegenerate interval in $\bR$, we define $\calH^I = D^{-1}(I)$.  
\end{defn}

Let $(G,\mc{P})$ be a group pair (so $G$ is a group and $\mc{P}$ is a collection of subgroups), and suppose that $G$ and the elements of $\mc{P}$ are all finitely generated.  Choose a generating set $S$ for $G$ which contains a generating set for each $P\in \mc{P}$.  (This is called a \emph{compatible} generating set.)  Let $\Gamma$ be the Cayley graph for $G$ with respect to $S$, metrized so each edge has length $1$.
 Each left coset $gP$ of $P\in \mc{P}$ spans a connected $gPg^{-1}$--invariant subgraph $\Gamma(gP)\subset \Gamma$.  
 \begin{defn}\label{def:cc}
   The \emph{cusped space} $X(G,\mc{P})$ is obtained from $\Gamma$ by attaching, for each $P\in \mc{P}$, and each coset $gP$, a copy of $\calH(gP)$, by identifying $\Gamma(gP)$ to the horosphere at depth $0$ of $\calH(gP)$.
 \end{defn}
 The cusped space is not quite determined by the pair $(G,\mc{P})$, since we had to choose a generating set, but any two choices give quasi-isometric spaces, by \cite[Corollary 6.7]{Groff}.

 \begin{defn}
   $(G,\mc{P})$ is relatively hyperbolic if and only if the cusped space $X(G,\mc{P})$ is Gromov hyperbolic.
 \end{defn}
In \cite[Theorem 3.25]{rhds} it is proved that this definition is equivalent to other definitions of relative hyperbolicity, in the finitely generated case.  See \cite{Hru-relqconv} for an extension of this definition to the non-finitely generated case.  Throughout this paper, we are only concerned with the case that $G$ and all elements of $\mc{P}$ are finitely generated.  We recall the following useful property of horoballs in the cusped space of a relatively hyperbolic group.
\begin{lemma}\label{horoballconvexity} \cite[Lemma 3.26]{rhds}
  Suppose $X=X(G,\mc{P})$ is $\delta$--hyperbolic, and that $\calH\subset X$ is a combinatorial horoball.
  For any integer $R\geq \delta$, the set $\calH^{[R,\infty)}$ is convex in $X$.
\end{lemma}

\begin{defn}
  Suppose that $(G,\mc{P})$ is relatively hyperbolic, and suppose that each element of $\mc{P}$ is infinite.  Let $X(G,\mc{P})$ be the associated cusped space.  The Gromov boundary $\partial X(G,\mc{P})$ is called the \emph{Bowditch boundary} of $(G,\mc{P})$.  
\end{defn}
In case some elements of $\mc{P}$ are finite, then $\partial X(G,\mc{P})$ contains isolated points.  If $\mc{P}^\infty$ is the collection of infinite elements of $\mc{P}$, then $(G,\mc{P}^\infty)$ is also relatively hyperbolic, and its Bowditch boundary can be obtained from $\partial X(G,\mc{P})$ by removing the isolated points.

In case $G$ itself is hyperbolic, Bowditch characterized which $(G,\mc{P})$ are relatively hyperbolic.  Recall a family of subgroups $\mc{P}$ is \emph{almost malnormal} if whenever $P_1\cap g P_2 g^{-1}$ is infinite, for $P_1,P_2\in \mc{P}$ and $g\in G$, we have $P_1=P_2$ and $g\in P_1$.
\begin{theorem}\cite[Theorem 7.11]{bowditch12}
  Let $G$ be hyperbolic, and suppose $\mc{P}$ is a family of distinct subgroups of $G$.  The pair $(G,\mc{P})$ is relatively hyperbolic if and only if $\mc{P}$ is an almost malnormal family of quasi-isometrically embedded subgroups.
\end{theorem}

\subsection{Dehn fillings}

\begin{defn}
 Let $G$ be a group and let $\mc{P}=\{P_1,\dots,P_n\}$ be a finite collection of subgroups of $G$. Given a collection of normal subgroups $N_i\trianglelefteq P_i$, called \emph{filling kernels}, the quotient $G\to G(N_1,\dots,N_n)=G/K$, where $K=\left\llangle \bigcup_i N_i \right\rrangle$, is called a \emph{(Dehn) filling} of $(G,\mc{P})$. We say that a property holds for all \emph{sufficiently long fillings} of $(G,\mc{P})$ if there is a finite set $\mathcal B\subseteq G\setminus\{1\}$ so that whenever $N_i\cap \mathcal B=\emptyset$ for all $i$, the group $G/K$ has the property.
\end{defn}

\begin{theorem}\cite{osin:peripheral,rhds}\label{thm:dehnfilling}
 Let $(G,\mc{P}=\{P_1,\dots,P_n\})$ be relatively hyperbolic. Then for any finite subset $F\subseteq G\setminus\{1\}$ the following holds. For any sufficiently long filling $\phi\co G\to G/K$ we have
 \begin{enumerate}
  \item for each $i$, $\phi$ induces an embedding of $P_i/N_i$ in $G/K$ whose image we identify with $P_i/N_i$,\label{item:P_imodN_i_embed}
  \item $(G/K,\{P_1/N_1,\dots,P_n/N_n\})$ is relatively hyperbolic,
  \item $\phi$ restricted to $F$ is injective.
 \end{enumerate}
\end{theorem}

For any relatively hyperbolic pair $(G,\mc{P})$, the peripheral groups $\mc{P}$ always consist of an almost malnormal family of quasi-isometrically embedded subgroups \cite[Proposition 2.36 and Lemma 5.4]{osin:relhypbook}.  Hence we have the following corollary of Theorem \ref{thm:dehnfilling}.
\begin{cor}
  Let $(G,\mc{P})$ be relatively hyperbolic.  For all sufficiently long fillings $G\to G/K$, the filling $G/K$ is hyperbolic if and only if every $P_i/N_i$ is hyperbolic.
\end{cor}

The following is an easy consequence of the third part of Theorem \ref{thm:dehnfilling}.

\begin{lemma}\label{Greendlinger:BallsEmbed}
 Let $(G,\mc{P})$ be relatively hyperbolic, with associated cusped space $X$. Then for any $R\geq 0$ the following holds. For any sufficiently long filling $G\to G/K$ the restriction of the map $X\to X/K$ to any ball of radius $R$ centered at an element of the Cayley graph is an isometry onto its image. Moreover, the same holds true for the map $X\to X/K_0$ where $K_0<K$ is any subgroup.
\end{lemma}

The next result is proved in \cite{agm} assuming that $G$ is torsion-free, but this assumption is not necessary, as explained in the proof of \cite[Theorem A.43]{VH}.  Alternatively, it follows from Lemma \ref{Greendlinger:BallsEmbed} and the Coarse Cartan--Hadamard Theorem (Theorem \ref{CartanHadamard} below).

\begin{prop} \cite[Proposition 2.3]{agm} \label{p:uniform delta}
 Suppose that $(G,\mc{P})$ is relatively hyperbolic, and fix a generating set for $G$ as in Definition \ref{def:cc}.  There exists a $\delta$ so that (i) the cusped space for $(G,\mc{P})$ is $\delta$--hyperbolic; and (ii) For all sufficiently long fillings $(G,\mc{P}) \to (\overline{G},\overline{\mc{P}})$ the cusped space for $(\overline{G},\overline{\mc{P}})$ (with respect to the image of the fixed generating set for $G$) is $\delta$--hyperbolic.
\end{prop}

\subsection{Geometry of truncated horoballs} \label{ss:geom trunc}
In the classical $2\pi$ Theorem of Gromov and Thurston, a cusped hyperbolic $3$--manifold is modified to a closed negatively curved one by replacing each cusp neighborhood by a thick ``Margulis tube'' around a short geodesic \cite{BleilerHodgson96}.  In the universal cover this tube lifts to a large neighborhood of a geodesic line.

In our setting we model our Dehn filled group $\overline{G} = G/K$ by a space which can be either thought of as
\begin{enumerate}
\item The quotient cusped space $X/K$, with certain deep horoballs removed, or
\item The Cayley Graph of $\overline{G}$, with certain \emph{truncated horoballs} added.
\end{enumerate}
The truncated horoballs are analogous to the neighborhoods of geodesic lines discussed above.  The same space with truncated horoballs omitted would still be Gromov hyperbolic, but we would lose control of various constants and be unable to make uniform statements over all long fillings.

Let $\theta > 0$ and suppose that $\Gamma$ is a $\theta$-hyperbolic Cayley graph.  It follows (see \cite[$\S$ III.H.1.22]{BH}) that $\Gamma$ satisfies Gromov's $4$-point condition $Q(\theta)$: for all $x,y,z,w\in \Gamma$,
\[ d(x,w)+d(y,z)\leq \max\left\{ d(x,y)+d(z,w),d(x,z)+d(y,w)\right\}+2\theta.\]

\begin{defn}\label{defn:tgamma}
 Let $t(\Gamma)$ be the smallest integer so that the graph $\calH^{t(\Gamma)}$ satisfies the Gromov $4$-point condition $Q(5)$; we argue below that this is well-defined.
\end{defn}

Let $\calH(\Gamma)$ be the combinatorial horoball based on $\Gamma$.  As noted in Section \ref{s:Prelim}, the metric on vertices in $\calH^k = D^{-1}(k)$ is defined by
\[	d_{\calH^k}(v,w) = \lceil 2^{-k} d_\Gamma(v,w) \rceil	.	\]
This formula and the defining equation for Gromov products shows that $t(\Gamma)$ exists and $t(\Gamma)\approx \log_2(\theta)$.  By the proofs of \cite[Propositions III.H.1.17, III.H.1.22]{BH} this implies that triangles in $\calH^{t(\Gamma)}$ are $30$-thin.  The graph $\calH^{t(\Gamma)}$ is a $30$--hyperbolic Cayley graph.  The loops of length at most $481$ based at a vertex give the relations in a Dehn presentation (see the proof of \cite[III.$\Gamma$.2.6]{BH}).
Attaching disks to all loops of length at most $481$ in $\calH^{t(\Gamma)}$, we therefore obtain a simply connected complex with linear combinatorial isoperimetric function with constant $1$.

In \cite[$\S3$]{rhds} a simply-connected $2$-complex is built from $\calH(\Gamma)$ by attaching vertical squares and pentagons and horizontal triangles, and the depth function $D$ is extended across these $2$-cells.  Let $\widetilde{\calH}(\Gamma)$ be this simply-connected $2$-complex, and denote the extended depth function by $\widetilde{D}$.  For $I$ an interval in $\bR$, define $\widetilde{\calH}^{I}(\Gamma) = \widetilde{D}^{-1}\left( I \right)$.  Then $\calH^{I}(\Gamma)$ (defined in above) is the $1$-skeleton of $\widetilde{\calH}^I(\Gamma)$.

The space $\widetilde{\calH}(\Gamma)$ satisfies a linear combinatorial isoperimetric function with constant $3$, by \cite[Proposition 3.7]{rhds}.
The proof
of this can be easily adapted by filling at depth $t(\Gamma)$ using the disks from the Dehn presentation, to prove the following result.

\begin{prop}
Suppose that $\Gamma$ is a $\theta$-hyperbolic Cayley graph, and that $t(\Gamma)$ is chosen as in Definition \ref{defn:tgamma}.  The $2$--complex 
$\widetilde{\calH}^{[0,t(\Gamma)]}(\Gamma)$ satisfies a linear combinatorial isoperimetric inequality with constant $3$.
\end{prop}

Since we also have a universal bound on the length of boundaries of disks, \cite[Proposition 2.23]{rhds} gives the following.
\begin{cor} \label{c:universal hyperbolic}
Let $\theta > 0$ and suppose that $\Gamma$ is a $\theta$-hyperbolic Cayley graph. The graph $\calH^{[0,t(\Gamma)]}(\Gamma)$ is $\theta_0$--hyperbolic, for a universal constant $\theta_0$.
\end{cor}

It is straightforward to see that if $0 \le a \le t(\Gamma)$ then $\calH^{[a,t(\Gamma)]}(\Gamma)$ is convex in $\calH^{[0,t(\Gamma)]}(\Gamma)$.  Therefore we have the following result.
\begin{cor} \label{TruncatedHoroballsAreHyperbolic}
Let $\theta > 0$ and suppose that $\Gamma$ is a $\theta$-hyperbolic Cayley graph. For any $a \in [0,t(\Gamma)]$, the graph $\calH^{[a,t(\Gamma)]}(\Gamma)$ is $\theta_0$--hyperbolic, for the same constant $\theta_0$ from Corollary \ref{c:universal hyperbolic}.
\end{cor}

It is possible to understand geodesics in $\calH^{[a,t(\Gamma)]}(\Gamma)$ in a very similar way to geodesics in $\calH(\Gamma)$.  The following result can be be proved using almost exactly the same proof as \cite[Lemma 3.10]{rhds}.

\begin{lemma} \label{l:geod in trunc}
Let $\theta > 0$ and suppose that $\Gamma$ is a $\theta$-hyperbolic Cayley graph. Let $a \in [0,t(\Gamma)]$ and suppose that $p,q \in \calH^{[a,t(\Gamma)]}(\Gamma)$ are distinct vertices.  There is a geodesic $\gamma$ in $\calH^{[a,t(\Gamma)]}(\Gamma)$ between $p$ and $q$ which consists of at most two vertical segments and a single horizontal segment.  Moreover, if this horizontal segment is not at depth $t(\Gamma)$ then it has length at most $3$.
\end{lemma}

The next lemma tells us that truncated horoballs are ``locally visual''.
\begin{lemma} \label{TruncatedHoroballsAreVisual}
Let $\theta > 0$ and suppose that $\Gamma$ is a $\theta$-hyperbolic Cayley graph. Suppose that $t(\Gamma) > \lambda > 10\theta_0$.

  For any $a \in [\lambda,t(\Gamma)]$ and any $p,q \in \calH^{[a,t(\Gamma)]}(\Gamma)$ so that $d(p,q) = \lambda$ there exists a geodesic $[p,q']$ of length $2\lambda$ in $\calH^{[a-\lambda,t(\Gamma)]}(\Gamma)$ so that $d(q,[p,q']) \le 2\theta_0$.
\end{lemma}
\begin{proof}
We assume that $p,q$ are vertices.

Choose a geodesic $[p,q]$ of the form as in the conclusion of Lemma \ref{l:geod in trunc}.
There are a number of possibilities.

If either $[p,q]$ has two vertical segments, or $p$ has greater depth than $q$, then $[p,q]$ ends at $q$ with a vertical segment heading towards the horosphere $\calH^0\cong \Gamma$.  In this case, append a vertical segment of length $\lambda + 10\theta_0$ to $[p,q]$ to form a new path $\sigma$ of length $2\lambda + 10\theta_0$.  It is straightforward to see that $\sigma$ is a $10\theta_0$--local geodesic, and so $\sigma$ lies within $2\theta_0$ of any geodesic between the endpoints of $\sigma$ (see \cite[III.H.1.13]{BH}).  Taking an initial subpath of length $2\lambda$ of any such geodesic gives a path $[p,q']$ as in the conclusion of the lemma.

  Suppose next that $q$ has greater depth than $p$ and that there is a horizontal segment of length at least $3$ at the end of $[p,q]$.  In this case, appending a vertical path of length $\lambda+10\theta_0$ from $q$ to the end of $[p,q]$ creates a $10\theta_0$--local geodesic, and we proceed as in the first case.

  The only remaining case is that $q$ has greater depth than $p$ and that $[p,q]$ is entirely vertical or terminates with a horizontal segment of length at most $2$.  Let $d$ be the depth of $p$. We claim that there exists $y\in\calH^{d}$ at distance at least $2^{\lambda-4}$ from $p$. In fact, if this was not the case then $\calH^{d+\lambda-4}$ would be contained in a $1$--ball around some point (that lies vertically below $p$), and hence $\calH^{d+\lambda-3}$ would have diameter $1$, implying that $t(\Gamma)$ is at most $d+\lambda-3$. However, the depth of $q$ is at least $d+\lambda-2$, a contradiction. There is a $\calH^{[a-\lambda,t(\Gamma)]}(\Gamma)$--geodesic from $p$ to $y$ that goes straight down from $p$ distance $\geq\lambda-3$, along a short horizontal segment and then straight up to $y$.  The path $[p,q]$ lies in the $5$--neighborhood of this geodesic, which can be prolonged to a geodesic of length $2\lambda$, still contained in $\calH^{[a-\lambda,t(\Gamma)]}(\Gamma)$, if needed.
\end{proof}

Finally, we prove that for sufficiently long fillings, the value of $t(\Gamma)$ (the partial truncation depth above) can also be made large.  This follows quickly from the following straightforward result.

\begin{lemma} \label{l:not very hyperbolic}
Suppose that $H$ is a group with finite generating set $S$.  For any $A > 0$ there exists $B$ so that for any nontrivial normal subgroup $J$ of $H$ so that (i) $H/J$ is a hyperbolic group; and (ii) $J$ contains no nontrivial elements of length less than $B$ (with respect to the word metric $d_S$), the Cayley graph of $H/J$ with respect to the image of $S$ is not $A$--hyperbolic.
\end{lemma}
\begin{proof}
Let $h$ be the shortest nontrivial element of $J$. Consider a geodesic $\gamma$ in the Cayley graph of $H$ from $1$ to $h$, which gives a loop $p$ in the Cayley graph of $H/J$. It is easily seen that any subgeodesic of $\gamma$ of length at most half the word length $|h|_S$ gives a geodesic in the Cayley graph of $H/J$. In particular, the loop $p$ can be subdivided into a geodesic bigon where the midpoint of one side is at distance $|h|_S/4$ from the other side. The lemma now follows easily.
\end{proof}

Suppose that $(G,\mc{P})$ is relatively hyperbolic.  Fix a compatible finite generating set $S$.  Suppose that $G(N_1,\ldots,N_m)$ is a Dehn filling with each
$P_i / N_i$ being hyperbolic.  According to Corollary \ref{c:universal hyperbolic} there are constants $t(i)$ so that the partially truncated horoballs of the Cayley graphs of $P_i/N_i$ to depth $t(i)$ is $\theta_0$--hyperbolic, for a universal constant $\theta_0$.  Moreover, $t(i)$ depends only on the hyperbolicity constant of the Cayley graph of $P_i/N_i$ (with respect to the obvious generating set in the image of $S$).  Moreover, from the construction and Lemma \ref{l:not very hyperbolic}, it is clear that this hyperbolicity constant goes to infinity and so does $t(i)$.  Therefore, the following is an immediate consequence of Lemma \ref{l:not very hyperbolic}.
\begin{cor} \label{c:t_c big}
Let $C > 0$ be any constant.  For all sufficiently long fillings $G(N_1, \ldots , N_m)$ of $(G,\mc{P})$, the constants $t(i)$ are all larger than $C$.
\end{cor}

\subsection{A Greendlinger Lemma} \label{s:greendlinger}

Roughly speaking, the next lemma says that, for $G\to G/K$ a sufficiently long filling, if we have some $x\in G$ and $g\in  K\setminus\{1\}$, then any geodesic $[x,gx]$ in the cusped space for $G$ goes deep into a horoball, and it can be shortened using an element of the conjugate of the filling kernel corresponding to the horoball. It is similar to \cite[Lemma 5.10]{DGO} and it can presumably be proven using the techniques we use in Section \ref{sec:spiderweb} to construct spiderwebs, but we give a simple proof that only relies on Proposition \ref{p:uniform delta} and Lemma \ref{Greendlinger:BallsEmbed}.

For a relatively hyperbolic pair $(G,\mc{P})$ with cusped space $X$, let $\mc{C}$ denote the collection of parabolic points for the $G$--action on $\partial X$.  Suppose that $G \to G(N_1,\ldots , N_m)$ is a Dehn filling, with $N_i \unlhd P_i$.  The points in $\mc{C}$ are limit points of horoballs of the form $\calH(gP_i)$ for $g \in G$.  If $c \in \mc{C}$ is of the form $c = \partial \calH (gP_i)$ then let $K_c = gN_ig^{-1}$.  We also write $\calH_c$ for $\calH(gP_i)$.

\begin{lemma}\label{greendlinger}
Let $(G,\mc{P})$ be relatively hyperbolic, with associated cusped space $X$. For any $D\geq 0$, for any sufficiently long filling $G\to G/K$ the following holds: For any $g\in K\setminus\{1\}$ and $x\in X$ there exists $c\in\calC$ so that
 \begin{enumerate}
  \item any geodesic $[x,gx]$ intersects $\calH_c^{D}$,
  \item for any geodesic $[x,gx]$ there exists $k \in K_c$ so that $d(x,kgx) < d(x,gx)$.
 \end{enumerate}
\end{lemma}
\begin{proof}
As in Proposition \ref{p:uniform delta}, let $\delta\geq 1$ be so that
\begin{itemize}
 \item $X$ is $\delta$--hyperbolic, and
 \item for all sufficiently long fillings $X/K$ is $\delta$--hyperbolic.
\end{itemize}

Fix $D \ge 0$.  It follows from Lemma \ref{Greendlinger:BallsEmbed} that for any sufficiently long filling $G\to G/K$, and any $x,y\in X$ in the same $K$--orbit satisfying $d(x,y)\leq 10\delta$, there exists a horoball $\calH_c$ so that $x,y\in \calH^{[D,+\infty)}_c$. Hence $y=kx$ for some $k\in K_c$ since the intersection of $K$ and the stabilizer of $\calH_c$ is $K_c$.

 Let $K$ be the kernel of such a filling, and let $g \in K \setminus \{1 \}$ and $x\in X$. If $d(x,gx)\leq 10\delta$ then we are done by the argument above, so suppose $d(x,gx)>10\delta$.  Let $\gamma=[x,gx]$ be a geodesic in $X$ from $x$ to $gx$ and let $\tilde{\gamma}$ be the projected path in $X/K$. Since $\tilde\gamma$ is a loop of length at least $10\delta$, $\tilde{\gamma}$ is not a $10\delta$--local geodesic.  Therefore, there are two
points $p$ and $q$ appearing in this order along $\gamma$ with $d(p,q)\leq 10\delta$ and some $k\in K$ so that $d(p,kq)<d(p,q)$. By the argument above, we have $p,q\in \calH^{[D,+\infty)}_c$ for some horoball $\calH_c$, and hence $\gamma\cap \calH_c^{D}\neq \emptyset$. Also, we have $k\in K_c$ and
$$d(x,kgx)\leq d(x,p)+d(p,kq)+d(kq,kgx)<d(x,p)+d(p,q)+d(q,gx)=d(x,gx),$$
as required.
\end{proof}

\section{Weak Gromov--Hausdorff convergence}\label{sec:wgh}

Our strategy to describe boundaries of Dehn fillings involves describing them as limits, in a suitable sense, of metric spaces that we have more control over.  The correct notion of limit for our purposes is similar to that of Gromov--Hausdorff limit and is described as follows.

\begin{defn}\label{wghlimit}
 Let $(M_i,d_i)_{i\in\mathbb N}$ and $(M,d)$ be metric spaces. We say that $(M,d)$ is a \emph{weak Gromov--Hausdorff limit} of the sequence $(M_i,d_i)$ if there exists $\lambda\geq 1$ and a sequence of $(\lambda,\epsilon_i)$--quasi-isometries $M\to M_i$, with $\epsilon_i\to 0$ as $i\to \infty$.
\end{defn}

\begin{example}
 If the compact metric space $(M,d)$ is a weak Gromov--Hausdorff limit of the sequence of connected metric spaces $(M_i,d_i)$, then $(M,d)$ is connected. In fact, if $M$ is not connected then we can write $M=A\sqcup B$ with $A,B$ non-empty and $d(A,B)=\epsilon>0$. It is readily seen that for $n$ large enough $M_n$ inherits a similar decomposition and hence it is not connected.
\end{example}

This section has two goals. The first one is to show that when a sequence of hyperbolic spaces converges in a suitable sense to a hyperbolic space, then their boundaries weakly Gromov--Hausdorff converge to the boundary of the limit hyperbolic space. The second goal is to give a criterion which allows us to prove (using a result of Claytor \cite{Claytor34}) that the weak Gromov--Hausdorff limit of a sequence of metric spaces homeomorphic to $S^2$ is planar.  The criterion we prove in this section, Lemma \ref{ivanovlemma}, is an adaptation of a result of Ivanov \cite{Ivanov97}.

\subsection{From convergence of spaces to convergence of their boundaries}

The definition of convergence of hyperbolic spaces that we use is the following one.

\begin{defn} \label{def:strongconverge}
  Let $(X,p)$ be a pointed metric space. Say the sequence of pointed metric spaces $\{(X_i,p_i)\}_{i\in\bN}$ \emph{strongly converges} to $(X,p)$ if the following holds:
For every $R>0$, there are isometries $\phi_i\co B_R(p)\to B_R(p_i)$ with $\phi_i(p)=p_i$ for all but finitely many $i$.
\end{defn}

In order to relate the boundary of a hyperbolic space to spheres of large radius, we need the space to be $D$--visual in the following sense -- a different concept from that of a \emph{visual metric} given in Definition \ref{def:visualmetric}.

\begin{defn}\label{defn:visualspace}
  Let $D>0$.  A geodesic metric space $X$ is \emph{$D$--visual} if, for every $a,b \in X$, there is a geodesic ray based at $a$ passing within $D$ of $b$. 
\end{defn}

The following is the main result of this subsection and it is an immediate corollary of Lemma \ref{lem:projecttosphere} below.
\begin{prop} \label{prop:strongconverge}
  Let $\delta>0$.
  Suppose $\{(X_i,p_i)\}_{i\in\bN}$ strongly converges to $(X,p)$, and that 
  the spaces $X$ and $X_i$ are all $\delta$--hyperbolic and $\delta$--visual.  Then for all positive $\epsilon\leq \frac{1}{6 \delta}$ and $\kappa$ as in Proposition \ref{findvisualmetric}, and any visual metrics   $\rho_i$ on $\partial X_i$, $\rho$ on $\partial X$ with parameters $\epsilon,\kappa$, the boundary $(\partial X,\rho)$ is a weak Gromov--Hausdorff limit of $(\partial X_i,\rho_i)$.
\end{prop}

\begin{rem}
  With a bit more work it should be possible to weaken the assumption of strong convergence in \ref{prop:strongconverge} to the assumption of pointed Gromov--Hausdorff convergence.
\end{rem}

\begin{lemma}\label{lem:tripod}
 Let $X$ be $\delta$--hyperbolic and let $w\in X$. Let $\alpha,\beta$ be rays starting at $w$ with limit points $a,b\in \partial X$.  Let $T$ be the tripod obtained by gluing two rays together along an initial subsegment of length $(a|b)_w$.  Then there is a $(1,5\delta)$--quasi-isometry from $\alpha\cup \beta$ to $T$, isometric on each of $\alpha$ and $\beta$.
\end{lemma}
\begin{proof}
  Let $s,t\in [0,\infty)$.  We must show that $d(\alpha(s),\beta(t))$ is within $5\delta$ of the distance of their images in $T$:
\[ \tau(s,t) =
\begin{cases}
|s-t| &  \min\{s,t\}\leq (a|b)_w\\
s+t-2(a|b)_w & \mbox{ otherwise}  \\
\end{cases}
\]
  By Lemma \ref{lem:avoidlimsup}, there are sequences $\{a_i\}\to a$ and $\{b_i\}\to b$ with $\lim\limits_{i\to\infty}(a_i|b_i)_w = (a|b)_w$.  Choose $n, N$ so that $(a_n|\alpha(N))_w$ and $(b_n|\beta(N))_w$ both exceed $\max\{s,t\}+10\delta$, and so that $(a_n|b_n)_w$ is within $\eta \le \frac{\delta}{2}$ of $(a|b)_w$.  Let $\alpha'$ be a geodesic from $w$ to $a_n$, and let $\beta'$ be a geodesic from $w$ to $b_n$.  We have $d(\alpha(s),\alpha'(s))$ and $d(\beta(s),\beta'(s))$ both bounded above by $\delta$.  

  Suppose first that one of $s,t$ is at most $(a|b)_w$.  Without loss of generality it is $s$.  Then $d(\alpha'(s),\beta'(s)) \leq \delta+2\eta \le 2\delta$.  It follows that $d(\alpha(s),\beta(s)) \le 4\delta$, and so $d(\alpha(s),\beta(t))$ is within $4\delta$ of $\tau(s,t)=|s-t|$.

  Finally suppose that both $s$ and $t$ are larger than $(a|b)_w$.  Consider a geodesic triangle $\Delta$ two of whose sides are $\alpha'$ and $\beta'$.  Let $\overline{a}$ and $\overline{b}$ be the images in the comparison tripod for $\Delta$ of $\alpha'(s)$ and $\beta'(t)$, respectively.  Then the distance $d(\overline{a},\overline{b})=s+t-2(a_n|b_n)_w$ is within $2\eta$ of $\tau(s,t)=s+t-2(a|b)_w$.  Thus $d(\alpha'(s),\beta'(t))$ differs from $\tau(s,t)$ by at most $2\delta+2\eta$, and $d(\alpha(s),\beta(t))$ differs from $\tau(s,t)$ by at most $4\delta+2\eta \le 5\delta$.
\end{proof}

For the next lemma, recall (as we observed in Section \ref{ss:visual metrics}) that even though $e^{-\epsilon(\cdot\mid\cdot)}$ may not be a metric, the concept of quasi-isometry still makes sense. Also, recall that given a point $p$ of a metric space $(M,d)$, we denote the sphere of radius $R$ around $p$, that is to say the set $\{x\in M\mid d(x,p)= R\}$, by $S_R(p)$.
\begin{lemma}\label{lem:projecttosphere}
 For every $\delta,\epsilon$ there exists $\lambda$ so that the following holds. Let $X$ be $\delta$--hyperbolic and $\delta$--visual.  Then for any $w\in X$ and any $R>0$, there is a $(\lambda,c)$--quasi-isometry
\[ \phi \co (\partial X,e^{-\epsilon(\cdot|\cdot)_w})\to (S_R(w), e^{-\epsilon(\cdot|\cdot)_w}), \]
where $c=c(\delta,\epsilon,R)$ tends to $0$ as $R$ tends to $+\infty$.
\end{lemma}

\begin{proof}
All rays in this proof are rays starting at $w$.  Denote $e^{-\epsilon(\cdot|\cdot)}$ by $\rho(\cdot,\cdot)$.  We'll prove the lemma for  $\lambda = e^{\frac{5}{2}\epsilon\delta}$ and $c = e^{\epsilon(\frac{5}{2}\delta-R)}$.  Note that $c$ tends to $0$ as $R$ tends to $\infty$.

Let us define a map $\phi\co \partial X\to S_R(p)$. For $a\in\partial X$, choose a ray $\gamma_a$ (parametrized by arc length) representing it. Then, set $\phi(a)=\gamma_a(R)$.
 
The fact that $X$ is $\delta$--visual combined with the fact that asymptotic rays stay within distance $\delta$ of each other implies that for any $x\in S_R(w)$ there exists $a\in\partial X$ with, say, $d(x,\phi(a))\leq 4\delta$, and hence $\rho(x,\phi(a))\leq e^{-\epsilon(R-2\delta)}<c$. Hence, the image of $\phi$ is $c$--dense in $S_R(w)$.

Now let $a,b\in\partial X$. We distinguish two cases.

If $(a|b)_w\geq R$ then $\rho(a,b)\leq e^{-\epsilon R}$.  In this case $\phi(a)$ and $\phi(b)$ lie within $5\delta$ of each other by Lemma \ref{lem:tripod}.  In particular $\rho(\phi(a),\phi(b))\leq e^{\frac{5}{2}\epsilon \delta-\epsilon R}$, so the difference $|\rho(\phi(a),\phi(b)) - \rho(a,b)|$ is at most $e^{\frac{5}{2}\epsilon\delta-\epsilon R}=c.$

Suppose on the other hand $(a|b)_w\leq R$, so that $\rho(a,b)\geq e^{-\epsilon R}$. By Lemma \ref{lem:tripod}, $\left|d(\phi(a),\phi(b))-2(R-(a|b)_w)\right|\leq 5\delta$, and hence $\left|(\phi(a)|\phi(b))_w-(a|b)_w\right|\leq \frac{5}{2}\delta$.  We deduce $$\lambda^{-1}\rho(a,b) \leq \rho(\phi(a),\phi(b))\leq \lambda\rho(a,b),$$
with no additive error in this case.
\end{proof}

\subsection{Linear connectedness and a lemma of Ivanov}

In order to show that the boundary of our filled group is planar in the proof of Theorem \ref{thm:boundary sphere} in Section \ref{s:proof of boundary sphere}, we use the following adaptation of a lemma of Ivanov \cite[Lemma 2.2]{Ivanov97}.
\begin{lemma}\label{ivanovlemma}
 Let $(M_i,d_i)$ be metric spaces, and assume that each $M_i$ is (homeomorphic to) a closed smooth manifold of dimension $\geq 2$. Suppose that there exists $L$ so that each $M_i$ is $L$-linearly connected and that $(M,d)$ is a weak Gromov--Hausdorff limit of $(M_i,d_i)$. If the finite graph $\Gamma$ can be topologically embedded in $M$ then for all large enough $i$ it can also be embedded in $M_i$.
\end{lemma}

We emphasize that we do {\em not} assume that the limit $M$ is a manifold.

\begin{proof}
By assumption there is some $K\geq 1$, and a sequence of $(K,\epsilon(i))$--quasi-isometries $\pi_i\co M\to M_i$ with $\epsilon(i)\to 0$ as $i\to +\infty$.

 Let $f\co \Gamma\to M$ be a topological embedding.  We fix some constants $C,\epsilon, \epsilon'$ satisfying:
 \begin{enumerate}
 \item $C>5K^2L$;
 \item for any disjoint subgraphs $\Gamma_1,\Gamma_2$ of $\Gamma$ we have $d(f(\Gamma_1),f(\Gamma_2))\geq C\epsilon$;
 \item\label{sharedendpoint} if the edges $e_1,e_2$ share the endpoint $v$ and $p_i\in e_i$ is so that $d(f(p_i),f(v))\geq \epsilon/C$ then $d(f(p_1),f(p_2))\geq C\epsilon'$; and
 \item $\epsilon'<\frac{\epsilon}{6KL}$.
 \end{enumerate}
 Fix $i$ so that $\epsilon(i)\leq \epsilon'$ until the end of the proof.
 For $v$ a vertex of $\Gamma$, let $\tilde{v}=\pi_i(f(v))$.
\setcounter{claim}{0}
\begin{claim}\label{vertexclaim}
We can choose, for each vertex $v$ of $\Gamma$, a path-connected neighborhood $U_{\tilde v}$ of $\tilde v$ so that $B_{\epsilon}(\tilde v)\subseteq U_{\tilde v} \subseteq B_{4L\epsilon}(\tilde v)$. Moreover, we can require $U_{\tilde v}$ to be a compact manifold (with boundary).
\end{claim}
 
\begin{proof}[Proof of Claim \ref{vertexclaim}]
 For $x,y\in M_i$, let $A_{x,y}$ be the union of all paths of length $\leq Ld_i(x,y)$ joining $x$ to $y$.   Let $A = \bigcup \{A_{x,y}\mid x,y\in B_{\epsilon}(\tilde{v})\}$.
  Notice that $A\subseteq B_{(2L+1)\epsilon}(\tilde v)$, and that $A$ is path-connected.  Fix a homeomorphism $h$ from a smooth manifold to $M_i$, and let $g\co M_i\to [0,\infty)$ be chosen so that $g\circ h$ is smooth and $g^{-1}(0)$ is the closure of $A$.

For any $R>(2L+1)\epsilon$, and any sufficiently small regular value $t$ of $g\circ h$,
we have $g^{-1}([0,t])\subseteq B_{R}(\tilde v)$.  In particular, we can fix $t$ so that $U_t=g^{-1}([0,t])\subseteq B_{4L\epsilon}(\tilde v)$.  We may take $U_{\tilde v}$ to be the connected component of $U_t$ containing $A$.
\end{proof}

\begin{claim} \label{edgeclaim}
  We can choose, for each edge $e$ of $\Gamma$, an embedded path $\gamma_e$ in $M_i$ in such a way that the following properties are satisfied. If $v$ is not an endpoint of $e$ then $\gamma_e$ does not intersect $U_{\tilde v}$, and it intersects $U_{\tilde v}$ exactly in one endpoint if $v$ is an endpoint of $e$. Moreover, the paths $\gamma_e$ are disjoint.
\end{claim}
\begin{proof}[Proof of Claim \ref{edgeclaim}]

  Let $e$ be an edge of $\Gamma$, and let $\{p_j\}_{j=1,...,n}$ be a sequence of points along $f(e)$ that subdivide $f(e)$ into subpaths of diameter $\leq \epsilon'$.  For each $j$ set $q_j=\pi_i(p_j)$. Consider paths $\gamma_{j}$ connecting $q_j$ to $q_{j+1}$ of diameter $\leq L d_i(q_j,q_{j+1})$. Let $A_e$ be the union of all such paths, and notice that $A_e\subset N_{L(K+1)\epsilon'}(\pi_i(f(e)))$.  
  
  Suppose $v$ is not an endpoint of $e$.  We claim $A_e\cap U_{\tilde{v}}$ is empty.  Indeed, $d(f(v),f(e))\geq C\epsilon$, so $$d_i(\tilde{v},A_e)\geq \frac{C}{K}\epsilon - \epsilon_i - L(K+1)\epsilon'\geq \left(5KL - \frac{K}{2}\right)\epsilon >4L\epsilon.$$  

  Suppose $e$ and $e'$ are edges of $\Gamma$ not sharing an endpoint.  We claim $A_e\cap A_{e'}$ is empty.  Indeed, 
$$d_i(A_e,A_{e'})\geq d_i(\pi_i(e),\pi_i(e'))-2L(K+1)\epsilon'
\geq \frac{C}{K}\epsilon -\epsilon_i - 2L(K+1)\epsilon' >0.$$

  Finally suppose that $e$ and $e'$ are edges which do share an endpoint $v$.  We claim $A_e\cap A_{e'}\subseteq \mathring U_{\tilde v}$.  Indeed, if $x\in A_e\cap A_{e'}$, there are $q_j\in \pi_i(f(e))$ and $q_k'\in \pi_i(f(e'))$ within $L(K+1)\epsilon'$ of $x$.  The corresponding points $p_j\in f(e)$ and $p_k'\in f(e')$ must satisfy
  $d(p_j,p_k')\leq K(L(K+1)\epsilon'+\epsilon')<C\epsilon'$.  Using the condition \eqref{sharedendpoint}, it follows that $d(p_j,f(v))$ is bounded above by $\frac{\epsilon}{C}$, and so $d(q_j,\tilde{v})\leq \frac{K}{C}\epsilon+\epsilon_i$.  Finally $d(x, \tilde{v})\leq \frac{K}{C}\epsilon + \epsilon_i + L(K+1)\epsilon' < \epsilon$.

 It is now easy to see that the path-connected set $A_e$ contains an embedded path $\gamma_e$ as required.
\end{proof}

 In order to conclude the construction we just need to observe that, since $U_{\tilde v}$ is a manifold of dimension at least 2, each $U_{\tilde v}$ contains a union of paths $P_{\tilde v}$ that pairwise only intersect at $\tilde v$, each connecting $\tilde v$ to an endpoint of some $\gamma_e$. 
 
 The union $\bigcup P_{\tilde v}\cup \bigcup \gamma_e$ is a subset of $M_i$ homeomorphic to $\Gamma$.
\end{proof}

\section{Spiderwebs}\label{sec:spiderweb}

In this section we make a construction similar to that of {\em windmills} from \cite[$\S5$]{DGO}.  We call our construction {\em spiderwebs}.  The main difference between the constructions is that we want the stabilizers of spiderwebs to be a free product of \emph{finitely many} factors, which is not the case for the windmills from \cite{DGO}.

We work in this section with a $\theta$--hyperbolic space, reserving the symbol $\delta$ for a constant that is chosen in later sections (and depends on $\theta$).  We will will fix a particular $\theta$ in Assumption \ref{rem:choice of theta} and  then fix $\delta = 1500\theta$ in Assumption \ref{rem:choice of delta}.  

\subsection{Notation}\label{sec:notation}
We fix the following notation from now until the end of the section.
Let $(G,\mc{P})$ be a relatively hyperbolic pair, and let $X$ be a cusped space for the pair as in Definition \ref{def:cc}.  Fix an arbitrary integer $\theta\geq 1$ so that $X$ is $\theta$--hyperbolic.  As in Section \ref{s:greendlinger}, let $\mc{C}$ be the collection of parabolic fixed points in $\partial X$.  We are going to choose a $G$--equivariant, $10^3\theta$--separated family of horoballs as follows:  Let $c\in \mc{C}$.  Then $c$ is the unique limit point of some $\calH_c = \calH(gP)$, for some coset $gP$ of some $P\in \mc{P}$.  Let $\sepH_c = \calH_c^{[500\theta,\infty)}$; this is convex in $X$ by Lemma \ref{horoballconvexity}.  Note that the closure of the complement of $\bigcup \sepH_c$ is $G$--cocompact.  

Suppose that $\{ N_i\lhd P_i \}$ is a collection of (long) filling kernels, with $N_i\neq\{1\}$.  As in Section \ref{s:greendlinger}, if $c = \partial \calH(gP_i)$, then let $K_c = gN_ig^{-1}$ be the conjugate of a filling kernel fixing $c$.  We suppose that the groups $K_c$ satisfy the following:
\begin{vtc}
 For each $c\in \calC$, $g\in K_c\setminus \{1\}$ and $x\in X\setminus \sepH_c$ we have $d_X(x,gx)\geq 10^4\theta$.
\end{vtc}
The following is an easy consequence of Theorem \ref{thm:dehnfilling}.
\begin{lemma} \label{lem:long is v translating}
For sufficiently long fillings the family $\{K_c\}$ satisfies the very translating condition.
\end{lemma}

\subsection{Spiderwebs}

For $Y$ a subset of $X$ we denote $\calC(Y)=\{c\in \calC\mid Y\cap \sepH_c\ne \emptyset\}$.

\begin{defn}[Spiderweb] \label{def:spiderweb}
A {\em $\theta$--spiderweb} is a subset $W$ of $X$ containing $1$ and satisfying the following axioms.
\begin{enumerate}[label=(S\arabic*)]
\item $W$ is $4\theta$--quasiconvex. \label{item:qc}
\item $\calC(W) = \calC(N_{50\theta}(W))$.\label{item:otherhoroballsfar}
\item\label{item:relcocompact} The group $K_W$ generated by 
\[	\bigcup_{c \in \calC(W)} K_c	,	\]
preserves $W$.  Moreover, for any $R>0$, $(N_R(G)\cap W)/K_W$ is compact.\label{item:cocompact}
\item There exists a finite subset $C \subset \calC(W)$ so that $K_W$ is the free product $\mathop{\ast}\limits_{c \in C} K_c$.\label{item:freeproduct}
\end{enumerate}
\end{defn}

Here is the main theorem of this section:
\begin{theorem}\label{exhaustion}
  In the notation established in Subsection \ref{sec:notation}, and for $K=\llangle \  \bigcup\limits_i N_i\ \rrangle$, there exists a family of $\theta$--spiderwebs $W_0\subset W_1\cdots$ so that $\bigcup W_i = X$ and (consequently) $K=\bigcup K_{W_i}$.
\end{theorem}

To extend a given $\theta$--spiderweb $W$ to a larger one, we need a few lemmas about how $W$ interacts with its translates under elements of $K_c$ for $\sepH_c$ near to $W$, but not intersecting $W$.  Define $\mc{C}'(W)$ to be $\mc{C}(N_{100\theta}(W))\setminus \mc{C}(W)$.
\begin{lemma}\label{smallprojections}
  Let $c\in \mc{C}'(W)$.  Then $\diam_X(\pi_{\sepH_c}(W))\leq 8\theta$ and $\diam_X(\pi_W(\sepH_c))\leq 16\theta$.
\end{lemma}
\begin{proof}
  Note that $\sepH_c$ is convex, and $W$ is $4\theta$--quasiconvex.  Moreover, $d(\sepH_c,W)\geq 50\theta$ by property \eqref{item:otherhoroballsfar} of $\theta$--spiderwebs.  The lemma then follows by applying Lemma \ref{quasiconvproj}.
\end{proof}
\begin{lemma}\label{concat23}
  Suppose $c\in\mc{C}'(W)$, and $g\in K_c\setminus\{1\}$.  Let $\gamma$ be a geodesic joining $W$ to $gW$.
  \begin{enumerate}
  \item\label{old3}  The geodesic $\gamma$ intersects $\sepH_c$ in a subsegment of length at least $100\theta$.
  \item\label{old2}  The geodesic $\gamma$ is contained in $N_{6\theta}(W)\cup N_{102\theta}(\sepH_c)\cup N_{6\theta}(gW)$.
  \end{enumerate}
\end{lemma}
\begin{proof}
  Let $\gamma$ be a geodesic joining $w\in W$ to $gw'\in gW$.

  Note that $g\pi_{\sepH_c}(W)=\pi_{\sepH_c}(gW)$.  Lemma \ref{smallprojections} says $\pi_{\sepH_c}(W)$  has diameter at most $8\theta$.  By the very translating condition $d_X(\pi_{\sepH_c}(W),\pi_{\sepH_c}(gW))$ is at least $(10^4-16)\theta$.  In particular $d_X(\pi_{\sepH_c}(w),\pi_{\sepH_c}(gw'))> 10^3\theta$.  Using the second part of Lemma \ref{quasiconvproj}, the geodesic $\gamma$ passes within $6\theta$ of both $\pi_{\sepH_c}(W)$ and $\pi_{\sepH_c}(gW)$.  In particular there are points $p,p'$ on $\gamma$ at depth at least $(500-6)\theta$ in the horoball containing $\sepH_c$, and satisfying $d_X(p,p')\geq (10^4 - 16 - 12)\theta$.  Since geodesics in combinatorial horoballs are vertical except for up to three horizontal edges (Lemma \ref{l:geod in trunc}), $\gamma$ must intersect $\sepH_c$ in a subsegment of length at least $(10^4-16-12-12)\theta-3> 100\theta$, establishing the first claim of the Lemma.

  Turning to the second claim, let $\sigma_1$ be a shortest geodesic joining $W$ to $\sepH_c$, and let $\sigma_2$ be a shortest geodesic from $\sepH_c$ to $gW$.  Note that each of $\sigma_1,\sigma_2$ has length at most $100\theta$.
   Lemma \ref{geodbetweenqconv} implies that the part of $\gamma$ between $w$ and $\sepH_c$ is contained in $N_{6\theta}(W\cup \sepH_c)\cup N_{2\theta}(\sigma_1)$.  Similarly the part of $\gamma$ between $\sepH_c$ and $w'$ is contained in $N_{6\theta}(\sepH_c\cup gW)\cup N_{2\theta}(\sigma_2)$.  Thus 
\[ \gamma \subseteq N_{6\theta}(W\cup gW)\cup N_{102\theta}(\sepH_c) ,\]
as required.
\end{proof}
\begin{lemma}\label{concat1}
  Let $c,c'\in \mc{C}'(W)$ be distinct, and let $g\in K_c\setminus\{1\}$, $g'\in K_{c'}\setminus\{1\}$.  Then $d_X(\pi_W(gW),\pi_W(g'W))\geq 500\theta$.
\end{lemma}
\begin{proof}
  By way of contradiction, suppose that $g\in K_c\setminus\{1\}$ and $g'\in K_{c'}\setminus\{1\}$ satisfy $d_X(\pi_W(gW),\pi_W(g'W))< 500\theta$.

  We claim that $\pi_W(gW)\subseteq \pi_W(\sepH_c)$ and $\pi_W(g'W)\subseteq \pi_W(\sepH_{c'})$.  Indeed, suppose $x\in \pi_W(gW)$.  Then there is some $y\in gW$ with $d(y,W)=d(y,x)$.  By Lemma \ref{concat23}.\eqref{old3}, any geodesic from $x$ to $y$ intersects $\sepH_c$.  Let $z\in \sepH_c$ be on one such geodesic.  Then $d(z,W)=d(z,x)$, so $x\in \pi_W(z)\subseteq \pi_W(\sepH_c)$.  This establishes that $\pi_W(gW)\subseteq \pi_W(\sepH_c)$; the argument that $\pi_W(g'W)\subseteq \pi_W(\sepH_{c'})$ is identical.

  Thus we also have $d_X(\pi_W(\sepH_c),\pi_W(\sepH_{c'}))<500\theta$.
    By Lemma \ref{smallprojections}, these projections have diameter at most $16\theta$.
  Since $\sepH_c$ and $\sepH_{c'}$ are each distance at most $100\theta$ from $W$, we deduce $d_X(\sepH_c,\sepH_{c'})< 500\theta+ 2(16\theta)+2(100\theta)=732\theta <10^3\theta$.  Since the horoballs are $10^3\theta$--separated, this contradicts $c\neq c'$.
\end{proof}

Since $\{1\}$ is a $\theta$--spiderweb, Theorem \ref{exhaustion} follows immediately from the following proposition.

\begin{prop} \label{p:exists bigger spiderweb}
Suppose that $W$ is a $\theta$--spiderweb.  Then there is a $\theta$--spiderweb $W'$ so that
\begin{enumerate}
 \item $W'$ contains $N_{10\theta}(W)$,
 \item $K_{W'} = K_W \ast \left( \mathop{\ast}\limits_{c \in E} K_c \right)$ for some finite subset
$E \subseteq \calC(W') \setminus \calC(W)$.
\end{enumerate} 
\end{prop}

\begin{proof}
  If $\calC(N_{60\theta}(W))=\calC(W)$ then $W'=N_{10\theta}(W)$ is a $\theta$--spiderweb, and the other condition trivially holds.
 
 Therefore, suppose that $\mc{C}'(W) = \calC(N_{100\theta}(W)) \setminus \calC(W)$ is nonempty.  Note that $\mc{C}'(W)$ has finitely many $K_W$--orbits, because of item \ref{item:cocompact} in the definition of $\theta$--spiderweb.
 
 Let $E$ be a set of representatives for the $K_W$--orbits of $\mc{C}'(W)$, let $K_W^+=\langle K_W\cup(\bigcup_{c\in E}K_c)\rangle$ and let $W'$ be the union of all geodesics connecting pairs of points in the orbit $K_W^+ N_{10\theta}W$.  By Lemma \ref{qconvhull}, $W'$ is $2\theta$--quasiconvex.  We remark that Lemma \ref{concat23}.\eqref{old3} (together with non-triviality of the $N_i$) implies that $\mc{C}'(W)\subseteq \mc{C}(W')$.

Our goal is now to prove that $W'$ is a $\theta$--spiderweb and $K_W^+=K_{W'}$.

Let 
\[ \phi\co K_W \ast  \left(\mathop{\ast}_{c\in E}K_c\right)\to K_W^+ \]
be the natural map.   Clearly $\phi$ is surjective.  We establish \eqref{item:freeproduct} in the definition of $\theta$--spiderweb by showing that $\phi$ is injective.  At the same time, we obtain information about geodesics between $K_W^+$--translates of $W$ sufficient to establish \eqref{item:otherhoroballsfar} in the definition of a $\theta$--spiderweb.

\begin{claim*}
  Let $w,w'\in W$, and let $g\in K_W\ast  \left(\mathop{\ast}\limits_{c\in E}K_c\right)$.  Let $\gamma$ be a geodesic joining $w$ to $\phi(g)w'$.
  Let $H_1 = \bigcup_{c\in E}\sepH_c$.  
  \begin{enumerate}
    \item\label{cl:gammacontrol} The geodesic $\gamma$ lies in
       $N_{34\theta}(K_W^+ W)\cup N_{130\theta}(K_W^+ H_1)$.
    \item\label{cl:gammadepthcontrol} Let $x\in \gamma\setminus N_{34\theta}(K_W^+W)$.  Then the distance from $x$ to $K_W^+W$ is at most $\depth(x)-300\theta$.
    \item\label{cl:injective} If $g\notin K_W$, then $\phi(g)w'\notin W$.
  \end{enumerate}
\end{claim*}
We complete the proof of the Proposition, assuming the claim.
Axiom \eqref{item:qc}, quasiconvexity, follows from Lemma \ref{qconvhull}, as already noted.  

We next show that Axiom \eqref{item:otherhoroballsfar} holds. In fact, we show $K_W^+(\calC(W)\cup E)=\calC(W')=\calC(N_{50\delta}(W'))$. The containments ``$\subseteq$'' are clear, so we are left to show if some horoball $\sepH_c$ satisfies $d_X(\sepH_c,W')\leq 50\theta$ then $c\in K_W^+(\calC(W)\cup E)$.  Let $x\in W'$ minimize the distance to $\sepH_c$.  The point $x$ is on some geodesic joining points in $K_W^+ . N_{10\theta}(W)$.  It therefore lies within $12\theta$ of a geodesic joining points in $K_W^+ . W$.  Translating everything by an element of $K_W^+$, we may assume that this geodesic has one endpoint in $W$, as in the claim.
Part \eqref{cl:gammacontrol} of the claim implies that $x$ lies either in a $46\theta$--neighborhood of some $K_W^+$--translate of $W$, or in a $142\theta$--neighborhood of some $K_W^+$--translate $k . \sepH_{c'}$ of $\sepH_{c'}$ for some $c'\in E$.  In the first case, we conclude that $\sepH_c$ has a $K_W^+$--translate meeting a $100\theta$--neighborhood of $W$, implying that $c\in K_W^+(\calC(W)\cup E)$.
In the second case, we have $d(\sepH_c,k. \sepH_{c'})\leq 200\theta$, implying $c=kc'$ by $10^3\theta$--separation of horoballs, and hence $c\in K_W^+( E)$.

The invariance of $W'$ under $K_{W}^+$ is immediate from the construction. Also, $K^+_{W}=K_{W'}$ because $\calC(W')=K_W^+(\calC(W)\cup E)$, so we get the first part of Axiom \eqref{item:cocompact}. The second part of Axiom \eqref{item:cocompact} follows from part \eqref{cl:gammadepthcontrol} of the Claim.

Since $\phi$ is automatically injective on $K_W$, part \eqref{cl:injective} of the claim shows $\phi$ is injective, establishing Axiom \eqref{item:freeproduct}, and showing $W'$ is a $\theta$--spiderweb.

\begin{proof}[Proof of Claim]
  We'll prove the claims by building a nice $100\theta$--local $12\theta$--tight quasigeodesic joining $w$ to $gw'$, and applying quasigeodesic stability.  
  Since $g$ lies in a free product, it can be written $g = g_1\cdots g_n$ where each $g_i$ is a nontrivial element of some free factor.  Without any loss of generality, we may assume that $g\neq 1$, and (rechoosing $w'$ if necessary) that $g_n\notin K_W$.
  We define certain prefixes $k_i = g_1\cdots g_{j_i}$ of $g$ inductively as follows:
  \begin{equation*}
    j_1= \begin{cases} 1 & g_1\notin K_W \\ 2 & g_1\in K_W\end{cases}\mbox{, and }
    j_{i+1} = \begin{cases} j_i+1 & g_{j_i+1}\notin K_W \\ j_i+2 & g_{j_i+1}\in K_W\end{cases}
  \end{equation*}
 Thus, for example, $k_1 = g_1$ if $g_1\notin K_W$ and  $k_1=g_1g_2$ otherwise.  We obtain elements $k_1,\ldots,k_s$, where $k_s = g$.  
 These choices ensure that, if we define $W_0 = W$ and $W_i= k_iW$ for $i\in \{1,\ldots,s\}$, we always have $W_i \neq W_{i-1}$.  For each $i\in \{1,\ldots,s\}$ choose a shortest geodesic $[q_{i-1},p_i]$ from $W_{i-1}$ to $W_i$.  This is a translate of a segment joining $W$ to $\kappa W$ for some $\kappa\in K_c\setminus\{1\}$, $c\in \mc{C}'(W)$.  Lemma \ref{concat23}.\eqref{old3} implies that $[q_{i-1},p_i]$ has length at least $100\theta$.

  Note that $q_{i-1}\in \pi_{W_{i-1}}(p_i)$ and $p_i\in \pi_{W_i}(q_{i-1})$.

  Set $p_0 = w$, and $q_s = \phi(g)w'$.  For each $i\in \{0,\ldots s\}$ choose a geodesic $[p_i,q_i]$; this geodesic lies in a $4\theta$--neighborhood of $W_i$ by quasiconvexity.  When $i\notin\{0,s\}$, we have $p_i\in \pi_{W_i}(W_{i-1})$ and $q_i\in \pi_{W_i}(W_{i+1})$, so by Lemma \ref{concat1}, it has length at least $500\theta$.

  Let $\alpha$ be the broken geodesic $[p_0,q_0]\cdot[q_0,p_1]\cdots[p_s,q_s]$.  We claim that $\alpha$ is a $100\theta$--local $12\theta$--tight path.  Except possibly for the first and last segments, all the geodesic subsegments of $\alpha$ have length at least $100\theta$, so tightness need only be verified on concatenations of two of the geodesic subsegments.  One of these segments always connects a point to a closest point projection in some  $W_i$ which contains both endpoints of the second segment.  Since $W_i$ is $4\theta$--quasiconvex, we can apply Lemma \ref{concatatproj} to conclude that this concatenation of two subsegments is $12\theta$--tight.

We can now apply Lemma \ref{localtight}, with $C = 12\theta$, to conclude that any geodesic $\gamma$ with the same endpoints as $\alpha$ is Hausdorff distance at most $28\theta$ from $\alpha$.  In particular, such a geodesic $\alpha$ does not lie in a $4\theta$--neighborhood of $W$, so $\phi(g)w'\notin W$ and part \eqref{cl:injective} of the claim is established.

To establish part \eqref{cl:gammacontrol}, we note that $\alpha$ lies in $N_{6\theta}(K_W^+W)\cup N_{102\theta}(K_W^+H_1)$ by applying Lemma \ref{concat23}.\eqref{old2} to the subsegments passing between the $W_i$.  It follows that any geodesic from $w$ to $\phi(g)w'$ lies in $N_{34\theta}(K_W^+W)\cup N_{130\theta}(K_W^+H_1)$.

To establish part \eqref{cl:gammadepthcontrol}, let $x$ lie on $\gamma$.  Then $x$ lies within $28\theta$ of some point $x'$ on $\alpha$.  If $x'\in [p_i,q_i]$ for some $i$, then $d_X(x,K_W^+W)\leq 30\theta$.  Otherwise, $x'\in [q_i,p_i]$, which is entirely contained in the $100\theta$--neighborhood of some $\sepH_c$.  In particular, the depth of $x'$ is at least $400\theta + d(x,K_W^+W)$
\end{proof}
This completes the proof of Proposition \ref{p:exists bigger spiderweb}.
\end{proof} 

We have already noted that Theorem \ref{exhaustion} follows immediately from Proposition \ref{p:exists bigger spiderweb}, so we have proved Theorem \ref{exhaustion} and completed the construction of $\theta$--spiderwebs.

The following result follows immediately from the construction of spiderwebs and may be useful in future applications.

  \begin{theorem}
  Suppose that $(G,\mc{P})$ is a relatively hyperbolic pair, let $X$ be the cusped space for $(G,\mc{P})$ and let $\mc{C}$ be the set of parabolic fixed points in $\partial X$.
  
    For all sufficiently long fillings, the following holds: Let $K$ be the kernel of the filling.  There is a set $T\subset \mc{C}$ meeting each $K$--orbit exactly once, so that 
\[ K = \mathop{\ast}\limits_{t\in T} \left( K\cap \mathrm{Stab}(t) \right), \]
    and each subgroup $K \cap \mathrm{Stab}(t)$ is conjugate in $G$ to a unique filling kernel $N_i\lhd P_i$.
  \end{theorem}
\begin{proof}
Fix a long enough filling
\[	G \to G(N_1,\ldots , N_m)	\]
so that the very translating condition above holds (this condition holds for sufficiently long fillings by Lemma \ref{lem:long is v translating}).

We then choose the construction of spiderwebs $\{ W_i \}_{i \in \mathbb N}$ as in Theorem \ref{exhaustion}, and specifically the family constructed via Proposition \ref{p:exists bigger spiderweb}.  By Theorem \ref{exhaustion} we have, for each $i$,

\[K = \bigcup_i K_{W_i}	,\]
and by Proposition \ref{p:exists bigger spiderweb} we know that 
\[	K_{W_{i+1}} = K_{W_i} \ast \left( \mathop{\ast}\limits_{c \in E_i} K_c \right)	,	\]
for some finite $E_i \subset \mathcal{C}$, where $K_c$ is a conjugate of some filling kernel $N_j$.  It follows that 
\begin{equation}\label{finitefreeproduct}	
  K_{W_i} = \mathop{\ast}\limits_{c\in \overline{E}_i} K_c\mbox{,\quad where \quad }\overline{E}_i = \bigsqcup_{j=1}^i E_j.
\end{equation}  
Since $K$ is an increasing union of the subgroups $K_{W_i}$, we have, for $T = \cup_i \overline{E}_i = \sqcup_i E_i$
\[ K = \mathop{\ast}\limits_{c\in T} K_c.\]

It remains to show that $T$ meets each $K$--orbit in $\calC$ exactly once.  Since $\bigcup\limits_i W_i = X$, it is clear that each $K$--orbit of element of $\mathcal{C}$ is eventually included in one of the $E_i$.
Suppose by contradiction that there is $k\in K$ and $c_1,c_2\in T$ so that $kc_1 = c_2$.  Let $j$ be chosen large enough so that $c_1, c_2\in \overline{E}_j$, and so that $k\in K_{W_j}$.  Then the subgroups $K_{c_1}$ and $K_{c_2}$ are conjugate inside $K_{W_j}$, contradicting the free product structure \eqref{finitefreeproduct}.

\end{proof}

In \cite[Theorem 7.9]{DGO} it is proved that the kernel is a free product of conjugates of the filling kernels $N_i$.  The only new part of the above result is to identify the indexing set for the free product as being in bijection with the $K$--orbits of $\mc{C}$.  We believe that this description of the indexing set also follows from the construction of windmills in \cite{DGO}, and also that this description is surely known by the authors of \cite{DGO}.

\section{Approximating the boundary of a Dehn filling} \label{s:technical}
  The statements in this section form the core of our new method for understanding the boundary of a Dehn filling.  In this section we give statements in the absolute and relative setting, but only use (or indeed prove) the absolute statements in the sequel.  The careful reader will see that the relative statements are strictly easier to establish.  

The absolute (hyperbolic) statements require some further constructions, which we give in the next subsection.

\subsection{Truncated quotients}
  Let $(G,\mc{P})$ be relatively hyperbolic.  In subsequent sections we will focus on (long) filling kernels $\{K_i\lhd P_i\}$ with  $P_i/K_i$ hyperbolic for each $i$.  We call such fillings \emph{hyperbolic fillings}. Since we do not require anything about the hyperbolicity constant of (a Cayley graph of) $P_i/K_i$, we do not get a uniform hyperbolicity constant for quotients of (a given Cayley graph of) $G$ by the filling kernel. We overcome this by taking truncated quotients as defined below. Having a uniform hyperbolicity constant regardless of the long hyperbolic filling will be crucial for us. Recall that in the case that $P_i/K_i$ is (virtually) $\mathbb Z$, the corresponding truncated horoball can be thought of as (the lift to the universal cover of) a Margulis tube, as discussed in Subsection \ref{ss:geom trunc}.
  
 Let $K$ be the normal closure in $G$ of $\bigcup_i K_i$.  For a sufficiently long filling, it is the case that the intersection $K_c$ of $K$ with a horoball stabilizer is conjugate to some $K_i$ (see Theorem \ref{thm:dehnfilling}.\eqref{item:P_imodN_i_embed}).  If we are assuming that the $P_i/K_i$ are hyperbolic, this means that $K_c$ acts on each ``horosphere'' $\calH_c^D\subset\calH_c$ with quotient a Gromov hyperbolic graph.  We saw in Subsection \ref{ss:geom trunc} that for sufficiently deep horospheres, the quotient is a hyperbolic graph with uniform constant.

Fix a cusped space $X$ for $(G,\mc{P})$.  In Section \ref{sec:spiderweb}, we constructed, for sufficiently long fillings $G\to \overline{G} = G(K_1,\ldots,K_n)$ a sequence of spiderwebs (Definition \ref{def:spiderweb}) $W_k\subseteq W_{k+1}\subseteq \cdots\subset X$, each of which is stabilized by a subgroup $K_{W_j}$ of the kernel of $G\to \overline{G}$.

Recall that in this section we are assuming all the quotients $P_i/K_i$ are hyperbolic groups.  The universal constant $\theta_0$ comes from Corollary \ref{c:universal hyperbolic}.

\begin{defn}\label{defn:trunc_quotient}
  Let $W$ be a $\theta$--spiderweb in $X$, associated to the filling kernels $\{K_i\lhd P_i\}$.  The \emph{truncated quotient} $T_W$ associated to $W$ is obtained in the following way.  As in Subsection \ref{ss:geom trunc}, for each horoball center $c$, we
we let $t_c$ be the minimal integer so that the quotient by $K_c$ of the horosphere at depth $t_c$ in $\calH_c$ satisfies Gromov's $4$--point condition $Q(5)$, and note that Corollary \ref{c:universal hyperbolic} then implies that the quotient $\calH_c^{[500\theta,t_c]}/K_c$ is $\theta_0$--hyperbolic (where $\calH_c^{[500\theta,t_c]}$ is that part of the horoball $\calH_c$ between depth $500\theta$ and $t_c$).
Let $\Sigma_c$ denote the horosphere at depth $t_c$, centered at $c$.
The group $K_W$ acts properly on $X\setminus \bigcup\{\calH_c^{(t_c,\infty)}\mid c\in \calC(W)\}$, and we let $T_W$ be the quotient by this action.  We similarly define $T_{\overline G}$ to be the the quotient of $X\setminus \bigcup\{\calH_c^{(t_c,\infty)}\mid c\in \calC\}$ by $K$. 
\end{defn}

The points coming from $W$ in $T_W$ form a compact subset.  The quotient $T_{\overline G}$ is quasi-isometric to $\overline G$.

\subsection{Statements for hyperbolic fillings}
The next results are some of the main ingredients of the proof of Theorem \ref{thm:boundary sphere}.

The following theorem says that, for $W$ a $\theta$--spiderweb associated to a long filling, $T_W$ is hyperbolic and visual (as in Definition \ref{defn:visualspace}) with uniform constants, and it describes the topology of its boundary.

\begin{restatable}{theorem}{thmtrunc}\label{thm:trunc}
  Let $(G,\mc{P})$ be relatively hyperbolic with cusped space $X$.  Then there exist $\theta,\delta$ with the following properties.  For all sufficiently long hyperbolic fillings $G \to \overline{G} = G(N_1,\ldots , N_n)$ with $N_i$ infinite for all $i$ and any $\theta$--spiderweb $W$ (see Definition \ref{def:spiderweb}) associated to the filling we have
  \begin{enumerate}
   \item\label{trunc_are_hyp} The truncated quotient $T_W$ is $\delta$--hyperbolic and $\delta$--visual, and so is $T_{\overline G}$.
   
   \item\label{trunc_cover} If $\mathcal F$ is the union of subsets of $\partial T_W$ of the form $\Lambda(\Sigma_c /K_c)$ for $c\in \calC(W)$, then there exists a regular covering map $(\partial X \setminus \Lambda(K_W))\to \partial T_W \setminus \mathcal F$ with deck group $K_W$.
\item\label{trunc_F codense} $\partial T_W \setminus \mathcal F$ is open and dense in $\partial T_W$.
\end{enumerate}  
  \end{restatable}
  
  The following theorem describes the boundary of the quotient group as a limit of boundaries of $T_{W_i}$, where the $\theta$--spiderwebs $W_i$ form an exhaustion of the cusped space.  Recall the notion of weak Gromov--Hausdorff convergence from Section \ref{sec:wgh}.

\begin{restatable}{theorem}{existsvisual} \label{thm:existsvisual}  
Let $(G,\mc{P})$ be relatively hyperbolic. Then there exist $\epsilon, \kappa$ so that for all sufficiently long hyperbolic fillings $G\to \overline{G} = G(N_1,\ldots,N_n)$ with $N_i$ infinite for all $i$ the following hold: 
\begin{enumerate}
\item\label{exists_visual} For any $\theta$--spiderweb $W$ as in Theorem \ref{thm:trunc} associated to the filling there is a visual metric $\rho_W$ on $\partial T_W$ based at the image of $1$ with parameters $\epsilon, \kappa$.
\item\label{exists_GH} If $W_1\subseteq W_2\subseteq \ldots$ is a sequence of $\theta$--spiderwebs with $\bigcup W_j=X$, then the sequence $\left\{ (\partial T_{W_j},\rho_{W_j}) \right\}$ weakly Gromov--Hausdorff converges to a visual metric on $\partial T_{\overline G}$.   
\end{enumerate}
\end{restatable}

The following theorem guarantees that the boundaries of the $T_W$ are linearly connected with uniform constant.  This is important in order to be able to apply Lemma \ref{ivanovlemma}.

\begin{restatable}{theorem}{linearconnectedness}\label{thm:linconn}
 Let $(G,\mc{P})$ be relatively hyperbolic and suppose that the Bowditch boundary $\partial X$ (when endowed with any visual metric) is linearly connected.  Then for all sufficiently long hyperbolic fillings $G\to \overline{G} = G(N_1,\ldots,N_n)$ with $\overline {G}$ one-ended and $N_i$ infinite for each $i$, the following holds: There exists $L$ so that, for any $\theta$--spiderweb $W$ with parameter $\theta$ as in Theorem \ref{thm:trunc}, $(\partial T_W,\rho_W)$ is $L$--linearly connected, where $\rho_W$ is the visual metric from Theorem \ref{thm:existsvisual}.
\end{restatable}

 Note that $L$ depends on the filling, but then is uniform over $\theta$--spiderwebs associated to that filling.  The constants $\delta$, $\epsilon$, $\kappa$ do not depend on the (long) filling.

\subsection{Statements for general fillings}
For a general (not-necessarily-hyperbolic) long Dehn filling, we can make similar statements as above for the Bowditch boundary of $(\overline{G},\overline{\mc{P}})$.  We use the following terminology.  For $W$ a spiderweb associated to a Dehn filling, let $X_W = {X}/{K_W}$.  If $K$ is the kernel of the filling map $G\to \overline{G}$, let $X_{\overline{G}} = {X}/{K}$, and note that, for long fillings, $X_{\overline{G}}$ is a cusped space for the pair $(\overline{G},\overline{\mc{P}})$.  In particular, $\partial (\overline{G},\overline{\mc{P}}) = \partial X_{\overline{G}}$.
\begin{theorem}\label{relversion:trunc}
  Let $(G,\mc{P})$ be relatively hyperbolic with cusped space $X$.  Then there exist $\theta,\delta$ with the following properties.  For all sufficiently long fillings $G \to \overline{G} = G(N_1,\ldots , N_n)$ with $N_i$ infinite for all $i$ and any $\theta$--spiderweb $W$ associated to the filling we have
  \begin{enumerate}
   \item The quotient $X_W$ is $\delta$--hyperbolic and $\delta$--visual, and so is $X_{\overline G}$.
   \item If $\mathcal F\subset \partial X_W$ consists of the points $\Lambda(H_c /K_c)$ for $c\in \calC(W)$, then there exists a regular covering map $(\partial X \setminus \Lambda(K_W))\to \partial X_W \setminus \mathcal F$ with deck group $K_W$.
    \item Let $\mc{F}^{\mathrm{iso}}\subseteq \mc{F}$ consist of those points which are isolated in $\partial X_W$ (so they come from finite index $N_i\lhd P_i$).  Then $\partial X_W  \setminus \mathcal F$ is open and dense in $\partial X_W\setminus \mathcal{F}^{\mathrm{iso}}$.
\end{enumerate}  
  \end{theorem}

\begin{theorem} \label{relativevisual}
Let $(G,\mc{P})$ be relatively hyperbolic. Then there exist $\epsilon, \kappa$ so that for all sufficiently long fillings $G\to \overline{G} = G(N_1,\ldots,N_n)$ with $N_i$ infinite for all $i$ the following hold: 
\begin{enumerate}
\item For any $\theta$--spiderweb $W$ as in Theorem \ref{relversion:trunc} associated to the filling there is a visual metric $\rho_W$ on $\partial X_W$ based at the image of $1$ with parameters $\epsilon, \kappa$.
\item If $W_1\subseteq W_2\subseteq \ldots$ is a sequence of $\theta$--spiderwebs with $\bigcup W_j=X$, then the sequence $\left\{ (\partial X_{W_j},\rho_{W_j}) \right\}$ weakly Gromov--Hausdorff converges to a visual metric on $\partial X_{\overline{G}}$.
\end{enumerate}
\end{theorem}

\begin{theorem} \label{rel:linconn}
 Let $(G,\mc{P})$ be relatively hyperbolic and suppose that the Bowditch boundary $\partial X$ (when endowed with any visual metric) is linearly connected.  Then for all sufficiently long fillings $G\to \overline{G} = G(N_1,\ldots,N_n)$ with $\partial (\overline {G},\overline{\mc{P}})$ linearly connected and $N_i$ infinite for each $i$, the following holds: There exists $L$ so that, for any $\theta$--spiderweb $W$ with parameter $\theta$ as in Theorem \ref{relversion:trunc}, $(\partial X_W,\rho_W)$ is $L$--linearly connected, where $\rho_W$ is the visual metric from Theorem \ref{relativevisual}.
\end{theorem}

\section{Proofs of approximation theorems for hyperbolic fillings}\label{sec:proofs}
In this section we give proofs of the hyperbolic versions of the theorems stated in the last section, namely Theorems \ref{thm:trunc}, \ref{thm:existsvisual} and \ref{thm:linconn}.  No other results from either this section or the last are used in the sequel.

\begin{assume}\label{rem:choice of theta}[Choice of $\theta$]
We fix once and for all a choice of $\theta$ so that:
\begin{enumerate}
\item $\theta \ge 100$;
\item $\theta$ is a hyperbolicity constant for the cusped space $X$, and $X$ is also $\theta$--visual; and
\item $\theta \ge \theta_0$, where $\theta_0$ is as in Corollary \ref{c:universal hyperbolic}.
\end{enumerate}
From now on we drop ``$\theta$'' when talking about spiderwebs.
\end{assume}

\subsection{Hyperbolicity and visibility of the truncated quotient} We now prove Theorem \ref{thm:trunc}.\eqref{trunc_are_hyp} which states that the truncated quotient $T_W$ is $\delta$--hyperbolic and $\delta$--visual, and so is $T_{\overline{G}}$, for some constant $\delta$ which is independent of the (long) filling and the spiderweb $W$.

Both $\delta$--hyperbolicity and $\delta$--visibility are proved using a kind of local-to-global principle.  In the case of hyperbolicity, this is the Coarse Cartan--Hadamard Theorem of Delzant and Gromov \cite{DelzantGromov08}.  
We use the formulation of Coulon in \cite{Coulon14}.  Say that a space is {\em $r$--simply-connected} if the fundamental group is normally generated by free homotopy classes of loops of diameter less than $r$.
\begin{theorem}[Coarse Cartan--Hadamard] \cite[A.1]{Coulon14}\label{CartanHadamard}
  Let $\nu\geq 0$, and let $R\geq 10^7\nu$.  Let $M$ be a geodesic space.  If every ball of radius $R$ in $M$ is $\nu$--hyperbolic and if $M$ is $10^{-5}R$--simply-connected, then $M$ is $300\nu$--hyperbolic.
\end{theorem}

The following is our local-to-global principle for visibility. (Recall from Definition \ref{defn:visualspace} that a space is visual if, roughly speaking, geodesics can be coarsely prolonged to geodesic rays.)
\begin{prop}[Local visibility implies global visibility]\label{lem:local_visibility}
 For every $\nu\geq 1$ the following holds. Let $M$ be a proper $\nu$--hyperbolic space and suppose that for 
 all $p,q \in M$ with $d(p,q)\leq 100\nu$ there exists a geodesic $[p,q']$ of length at least $200\nu$ with $d(q,[p,q'])\leq \nu$. Then $M$ is $5\nu$--visual.
\end{prop}

\begin{proof}
Let $[a,b]$ be a geodesic segment.  We will define a sequence of points $\{p_i\}_{i\in \bZ_{\geq 0}}$ so that (i) $p_0=a$ and $b\in\{p_1,p_2\}$; (ii) the Gromov products $(p_{i-1},p_{i+1})_{p_i}$ are small: and (iii) the distances $d_M(p_{i+1},p_i)$ are large for $i>0$.  A standard argument (eg \cite[Lemma 4.9]{AgolGrovesManning-alternateMSQT}) then shows that the concatenations of geodesics $[p_0,p_1]\cdots[p_{n-1},p_n]$ lie close to any geodesics $[p_0,p_n]$.  These geodesics (sub)converge to a geodesic ray passing near $b$.

Two cases must be distinguished.  If $d_M(a,b) < 50\nu$, then let $p_1 = b$.  Now choose a geodesic $\sigma_2$ of length $200\nu$ beginning at $a$ and passing within $\nu$ of $b$.  Choose $p_2$ to be the point on $\sigma_2$ at distance $50\nu$ from $b$.   

The second case is when $d_M(a,b) \ge 50\nu$.  In this case, choose $p_1$ to be the point on $[a,b]$ at distance $50\nu$ from $b$ and let $p_2 = b$.

We can inductively suppose that points $p_0,\ldots,p_{i-1}$ have been chosen, and that $d_M(p_{i-1},p_{i-2})=50\nu$.
We apply the hypothesis of the lemma
with $p=p_{i-2}$ and $q=p_{i-1}$ to find a geodesic $\sigma_i$ of length $200\nu$ beginning at $p_{i-2}$, passing within $\nu$ of $p_{i-1}$.  The geodesic $\sigma_i$ contains a point at distance $50\nu$ from $p_{i-1}$ and distance at least $98\nu$ from $p_{i-2}$.  We choose $p_i$ to be such a point of $\sigma_i$.  

We thus have a sequence of points $p_0,p_1,\ldots$ so that $d_M(p_i,p_{i+1})=50\nu$ and $d_M(p_i,[p_{i-1},p_{i+1}])<\nu$ for each $i>0$.  A standard argument shows the concatenation $[p_0,p_1]\cdots[p_{n-1},p_n]$ lies in a $5\nu$--neighborhood of $[p_0,p_n]$.  In particular, the point $b$ lies within $5\nu$ of any geodesic $[p_0,p_n]$.  Because $M$ is proper, a sequence of geodesics $\gamma_n = [p_0,p_n]$ must subconverge to a geodesic ray $\gamma$ which also passes within $5\nu$ of the point $b$.
\end{proof}

We want to consider a long hyperbolic filling of $(G,\mc{P})$, and the truncated partial quotient $T_W$ associated to a spiderweb for the kernel of such a filling (see Definition \ref{defn:trunc_quotient}).  In particular, we will show it is $\delta$--hyperbolic and $\delta$--visual, where 
$ \delta = 1500\theta$, thus establishing Theorem \ref{thm:trunc}.\eqref{trunc_are_hyp}.

\begin{claim} \label{claim:T_W D-sc}
$T_W$ is $100\theta$--simply-connected.
\end{claim}

\begin{claim} \label{claim:one ball or the other}
  For all sufficiently long hyperbolic fillings $G\to \overline{G}$, and any associated spiderweb $W$, the following holds.  
  Let $B$ be a ball of radius $10^7\theta$ in $T_W$ or $T_{\overline G}$. Then $B$ is isometric to a ball in either $X$ or $\calH^{[500\theta,t_c]}_c/K_c$ for some $c\in\mathcal C$. Moreover, the first case holds whenever $B$ is not entirely contained in a horoball.
\end{claim}

Before proving  Claims \ref{claim:T_W D-sc} and \ref{claim:one ball or the other}, we argue that together they imply hyperbolicity and visibility of $T_W$ (the argument for $T_{\overline G}$ is identical).  Since both $X$ and $\calH^{[500\theta,t_c]}_c/K_c$ are $\theta$--hyperbolic (see Corollary \ref{TruncatedHoroballsAreHyperbolic}), it follows immediately from the two claims and Theorem \ref{CartanHadamard} that $T_W$ is $300\theta$--hyperbolic.  

We now check that $T_W$ is $1500\theta$--visual using Proposition \ref{lem:local_visibility} with $\nu=300\theta$ (the argument for $T_{\overline G}$ is identical).  Consider a geodesic $[p,q]$ in $T_W$ with $d(p,q)\leq 100\nu = 3\cdot 10^4\theta$.  We need to find a geodesic $[p,q']$ of length at least $6\cdot 10^4\theta$ and so that $d(q,[p,q'])\leq \nu = 300\theta$.  If $p$ is contained in a ball $B$ in $T_W$ of radius $10^7\theta$ isometric to a ball in $X$, then the required geodesic of length $6\cdot 10^4\theta$ exists since $X$ is $\theta$--visual.  If not, a ball of radius $10^7\theta$ centered at $p$ lies in some horoball $\calH^{[500\theta,t_c]}_c/K_c$.  In particular $p$ lies at depth at least $10^7\theta$, and
$[p,q]$ lies in a horoball $\calH^{[10^6\theta,t_c]}_c/K_c$.   Lemma \ref{TruncatedHoroballsAreVisual} gives a geodesic $[p,q']$ passing within $2\theta_0< \nu=300\theta$ of $q$.  By Proposition \ref{lem:local_visibility}, $T_W$ is $5\nu = 1500\theta$--visual.

\begin{proof}[Proof of Claim \ref{claim:T_W D-sc}]
Note that $X$ is $\theta$--hyperbolic.  It follows immediately from \cite[III.H.2.6]{BH} that $X$ is $16\theta$--simply-connected, which is to say that $\pi_1(X)$ is normally generated by free homotopy classes of loops of length at most $16\theta$.

Now $\pi_1(X_W/K_W)$ is normally generated by free homotopy classes of:
\begin{enumerate}
\item The images of the loops which normally generate $\pi_1(X)$; and
\item Loops representing a choice of generators of $K_W$.
\end{enumerate}
We can choose loops which represent generators of $K_W$ to each lie within a single horoball at depth $0$.

The subgraph $T_W$ of $X/K_W$ agrees with $X/K_W$ at depth less than $t_c$ (for each given horoball), so the generators of $\pi_1(T_W)$ can be taken to be a collection of loops which are either:

\begin{enumerate}
\item Images of loops representing generators of $\pi_1(X)$ of length at most $16\theta$; or
\item Peripheral loops, entirely contained in a horoball.
\end{enumerate}
Any path in a horoball can be pushed across pentagons and squares to maximal depth.  Since at maximal depth the horoball is $\theta_0$--hyperbolic, another application of \cite[III.H.2.6]{BH} implies that $T_W$ is $100\theta$--simply-connected, as required.  This proves Claim \ref{claim:T_W D-sc}.
\end{proof}

 \begin{proof}[Proof of Claim \ref{claim:one ball or the other}]
  We suppose that $G\to \overline{G}$ is a long enough filling to apply Corollary \ref{c:t_c big} with $C = 10^{10}\theta$ and to apply Lemma \ref{Greendlinger:BallsEmbed} with $R = 10^9\theta$.  In particular, we have $t_c\geq 10^{10}\theta$ for all $c\in \mc{C}$.
We fix an associated spiderweb $W$, and prove the Claim for $T_W$, the proof for $T_{\overline G}$ being almost identical.   
 
  Let $B'$ be a ball with the same center as $B$ and radius $10^8\theta$.  Notice that geodesics connecting points in $B$ are contained in $B'$.  We distinguish two cases.

In the first case, $B'$ is disjoint from $G/K_W$, the image of the Cayley graph in $X/K_W$. In this case, $B$ is isometric to a ball in $\calH^{[500\theta,t_c]}_c/K_c$.

In the second case, $B'$ intersects $G/K_W$, say at the image of $p\in G$.  Since $t_c\geq 10^{10}\theta$, 
$B$ misses the truncation completely, and is isometric to a ball $\hat{B}$ in $X/K_W$ with center at depth $\leq 10^8\theta$.  This ball $\hat{B}$ is entirely contained in the image of a ball $B'$ of radius $10^9\theta$ centered on a vertex of the Cayley graph of $G$ contained in $X$.  Using Lemma \ref{Greendlinger:BallsEmbed}, this ball embeds isometrically into $X/K_W$, so $B$ is actually isometric to a subset of $B'\subset X$.  This proves Claim \ref{claim:one ball or the other}.
 \end{proof}
As we explained above, these two claims imply Theorem \ref{thm:trunc}.\eqref{trunc_are_hyp}, so the proof of this theorem is complete.
 
 \begin{assume}\label{rem:choice of delta}[Choice of $\delta$]
 For the remainder of this section, $\delta$ denotes the constant in Theorem \ref{thm:trunc}.\eqref{trunc_are_hyp}; i.e.  $\delta = 1500\theta$.
\end{assume}

\subsection{Topology of the boundary of the truncated quotient}
In this section, we prove part \eqref{trunc_cover} of Theorem \ref{thm:trunc} about the existence of a covering map $(\partial X \setminus \Lambda(K_W))\to \partial T_W \setminus \mathcal F$ with deck group $K_W$ (for appropriate set $\mathcal F$).  To this end, we fix for this subsection a hyperbolic Dehn filling of $(G,\mc{P})$ with associated $\theta$--hyperbolic cusped space $X$, and make the following assumption.
\begin{assume}\label{assume:trunc_cover}  The Dehn filling is sufficiently long so that:
\begin{enumerate}
\item If $W$ is a spiderweb associated to the filling, then the truncated quotient $T_W$ is $\delta$--hyperbolic and $\delta$--visual (Theorem \ref{thm:trunc}.\eqref{trunc_are_hyp})
\item\label{suppose t_c big} Every truncation depth $t_c$ is at least $10^{10}\theta$ (Corollary \ref{c:t_c big}).
\item\label{suppose greendlinger} For any $x\in X$ and any $k\in K\setminus\{1\}$, any geodesic $[x,kx]$ meets some horoball at depth $D\geq 10^5\theta$ (Lemma \ref{greendlinger}).
\end{enumerate}
\end{assume}
We also fix a spiderweb $W\subset X$ associated to the Dehn filling we have chosen.
\begin{defn}\label{defn:saturated_spiderweb}
  The \emph{saturated spiderweb} $S_W$ is $W \cup \left( \bigcup\limits_{c\in \calC(W)}\sepH_c \right)$.  The {\em truncated, quotiented version} $\overline{S}_W$ is the intersection of $S_W/K_W$ with $T_W\subseteq X/W$.
\end{defn}

\subsubsection{Quasiconvexity and limit sets}
\begin{lemma}\label{lem:saturation_quasiconvex}
  The saturated spiderweb $S_W$ is $6\theta$--quasiconvex in $X$.
\end{lemma}
\begin{proof}
  If $A$ is a $K$--quasiconvex set, and $\mc{W}$ is a collection of $L$--quasiconvex sets, each of which has nonempty intersection with $A$, an easy quadrangular argument shows that $A\cup \bigcup\mc{W}$ is $(\max\{K,L\}+2\theta)$--quasiconvex.

  The spiderweb $W$ is $4\theta$--quasiconvex, and the horoballs $\sepH_c$ are $0$--quasiconvex, so the result follows.
\end{proof}

Since $X\to X/K_W$ is a covering map and $S_W$ is $K_W$--equivariant, we have the following corollary.
\begin{lemma}\label{lem:untruncquasi}
  $S_W/K_W$ is $6\theta$--quasiconvex in $X/K_W$.
\end{lemma}

Next we show that the quasiconvexity persists after we truncate.

\begin{lemma}\label{lem:overlineSW_qconv}
  $\overline{S}_W$ is $3\delta$--quasiconvex in $T_W$.
\end{lemma}
\begin{proof}
Suppose that $p, q \in \overline{S}_W$.  Let $\gamma$ be a $X/K_W$--geodesic between $p$ and $q$.  According to \cite[Lemma 3.10]{rhds}, we may assume that $\gamma$ intersects any horoball in a path which consists of at most two vertical segments and a single horizontal segment.  We form a path $\overline{\gamma}$ in $T_W$ between $p$ and $q$ as follows.  Any part of $\gamma$ which is not contained in $T_W$ lies in a truncated part of a horoball $H$.  Such a segment of $\gamma$ consists of two vertical segments (of length at least $t_c$) and a single horizontal segment.  Replace any such subsegment below depth $t_c$ by a geodesic at depth $t_c$ in the truncated horoball.  

Applying Lemma \ref{lem:untruncquasi} to $\gamma$ and noting that $\overline{\gamma}\setminus \gamma$ lies entirely in $\overline{S}_W$, we see that $\overline{\gamma}$ lies in a $6\theta$--neighborhood of $\overline{S}_W$.

We claim that $\overline{\gamma}$ is a $10\delta$--local geodesic in $T_W$.  At depths less than $t_c - 10\delta$ this is clear, since at such depths the spaces $X/K_W$ and $T_W$, and the paths $\gamma$ and $\overline{\gamma}$, are locally identical.  Thus suppose that $\sigma$ is a subsegment of $\overline{\gamma}$ of length $10\delta$ that has at least one point at depth greater than $t_c - 10\delta$.  Since $t_c\gg 10\delta$ (Assumption \ref{assume:trunc_cover}.\eqref{suppose t_c big} above), this subsegment lies entirely inside a single truncated horoball.

Let $a$ and $b$ be the endpoints of $\sigma$.  If $\sigma$ were part of the original path $\gamma$ then since distances in $T_W$ are greater than those in $X/K_W$, $\sigma$ remains a geodesic in this case.

We are left with the possibility that $\sigma$ is not part of the original path $\gamma$.  Suppose that $\sigma_0$ is a $T_W$--geodesic between $a$ and $b$, chosen to satisfy the conclusion of Lemma \ref{l:geod in trunc}.  In particular, $\sigma_0$ consists of at most two vertical segments and a single horizontal segment, either at depth $t_c$ or having length at most $3$.  It is clear that the only way $\sigma_0$ could be shorter than $\sigma$ is if $\sigma_0$ does not intersect depth $t_c$, since in every other case the only possible difference between $\sigma$ and $\sigma_0$ is the choice of geodesic at depth $t_c$.  However, if $\sigma_0$ does not intersect depth $t_c$, 
it must be that neither $a$ nor $b$ is at depth $t_c$, and the $X/K_W$--geodesic between $a$ and $b$ goes beneath depth $t_c$ (in order that $p$ be truncated).  
Lemma \ref{l:geod in trunc} ensures that any horizontal segment in $\sigma_0$ has depth at most $3$.  But then it is clear that there would not have been truncation, since $\sigma_0$ is then an $X/K_W$--geodesic as well.  Thus $\sigma_0$ cannot be shorter than $\sigma$ and we have argued that $\overline{\gamma}$ is a $10\delta$--local geodesic, as required.
By \cite[III.H.1.13]{BH}, any geodesic joining $p$ to $q$ lies within $2\delta$ of such a path, and thus lies in the $(2\delta+6\theta)$--neighborhood of $\overline{S}_W$.  Since $6\theta<\delta$, the lemma is proved.
\end{proof}

\begin{lemma}\label{lem:same_limit_set}
  $\Lambda(S_W) = \Lambda(W) = \Lambda(K_W)$
\end{lemma}

\begin{proof}
The group $K_W$ stabilizes $W$, hence $\Lambda(K_W)\subseteq \Lambda(W)$, and $W\subseteq S_W$, so $\Lambda(W)\subseteq \Lambda(S_W)$. It remains to show $\Lambda(S_W)\subseteq\Lambda(K_W)$. Let $\{x_i\}$ be a sequence of points in $S_W$ converging to some $x\in\partial X$. We can assume that either (i) they are all contained in $G$; or (ii) each is contained in a horoball $\sepH_{g_ic}$ for some $g_i\in K_W$ (there are finitely many $K_W$--orbits of horoballs intersecting $W$).
 In the first case, $x\in\Lambda(K_W)$ because $K_W$ acts cocompactly on $W\cap G$. In the second case, up to passing to a subsequence one of the following holds:  Either all $g_ic$ coincide or all $g_ic$ are pairwise distinct.

First suppose that all $g_ic$ coincide. Recalling that $c$ is the point at infinity of $\sepH_c$, we have $x=g_ic$, and $g_ic\in \Lambda(K_W)$ because $K_{g_ic}<K_W$ is infinite.
  
Finally, suppose that all $g_ic$ are pairwise distinct. In this case it is easy to see that $x$ coincides with the limit of the $g_i$, hence $x\in \Lambda(K_W)$ as required.
\end{proof}

We next describe the limit set of $\overline{S}_W$.  Recall that the set $\mc{F}$ is the union of the limit sets $\Lambda(\Sigma_c/K_c)$ for $c \in \mc{C}(W)$.
\begin{lemma}
  $\Lambda(\overline{S}_W)$ is $\mc{F}$.
\end{lemma}
\begin{proof}
  Note that $\overline{S}_W$ is finite Hausdorff distance from the union $\bigcup\limits_{c\in \calC(W)} \Sigma_c/K_c$, which (choosing representatives of $K_W$--orbits of horoball centers $c$) is actually a finite union of quasiconvex sets of the form $\Sigma_c/K_c$.  The limit set $\Lambda(\overline{S})$ is thus equal to the union of the limit sets of the $\Sigma_c/K_c$ in $\partial T_W$, which is $\mc{F}$.
\end{proof}

\subsubsection{The action of $K_W$ on $S_W$ and $X$}

\begin{defn}
The {\em frontier} of $S_W$ is the set of vertices in $S_W$ which are joined by an edge to a point in $X \setminus S_W$.
\end{defn}
Observe that by construction every vertex in the frontier of $S_W$ has depth at most $500\theta$.

\begin{lemma}\label{lem:frontier_translation}
Suppose that $x \not \in S_W$ or $x$ belongs to the frontier of $S_W$.  Then for any $k \in K \setminus \{ 1 \}$ we have $d(x,kx) > 100\delta$.
\end{lemma}
\begin{proof}
  Consider $x \not \in S_W$.  Then any $y \in \pi_{S_W}(x)$ is contained in the frontier of $S_W$.  By Assumption \ref{assume:trunc_cover}.\eqref{suppose greendlinger}, any geodesic $[y,ky]$ must go at least $10^5\theta$ into some horoball.  Since the depth of $y$ is at most $500\theta$, we have $d(y,ky)\geq 2\cdot(10^5-500)\theta$.
The broken geodesic $[x,y] \cup [y,ky] \cup [ky,kx]$ is $2\theta$ close to a geodesic $[x,kx]$, from which it quickly follows that $d(x,kx) > 1.5\cdot 10^5\theta = 100\delta$, as required.
\end{proof}

\begin{cor}\label{cor:fiftydeltaisometry}
  Let $x$ be a point in $X\setminus N_{100\delta}(S_W)$.  Then the map from $X$ to $X/K_W$ restricts to an isometry from the $50\delta$--ball in $X$ around $x$ to a $50\delta$--ball in $T_W$.
\end{cor}

\subsubsection{The covering map}\label{subsec:covering_map}
Let $\phi\co X\to X/K_W$ be the quotient map.
We define a map 
\[ \Theta \co (\partial X \setminus \Lambda(K_W))\to (\partial T_W) \setminus \mathcal F \]
as follows.  Represent $\xi\in \partial X\setminus \Lambda(K_W)$ by a $50\delta$--local geodesic ray $\gamma \co [0,\infty)\to X$ starting at $1$.  (Note that there is a geodesic ray Hausdorff distance at most $3\theta$ from $\gamma$; we use local geodesic rays because they occur naturally in the proof anyway.)
Let $R_\gamma$ be the smallest number so that $d(\gamma(t), S_W)\geq 100\delta$ for all $t\geq R_\gamma$.
Since $\xi\notin \Lambda(K_W) = \Lambda(S_W)$ (see Lemma \ref{lem:same_limit_set}), and $S_W$ is quasiconvex, there is such an $R_\gamma$.  If $\gamma$ is actually a geodesic, Lemma \ref{quasiconvproj} can be used to show that $\gamma$ makes linear progress away from $S_W$ after time $R_\gamma$:
\begin{equation}\label{eq:linear}
 100\delta - 22\theta + t < d( \gamma(R_\gamma+t),S_W) \leq 100\delta + t .
\end{equation}
Similar statements can be made for a $50\delta$--local geodesic, using the fact it is quasigeodesic and close to a geodesic, and/or replacing $100\delta$ with any quantity sufficiently large with respect to $\theta$.

Define $\overline\gamma(t) = \phi(\gamma(t+R_\gamma))$.  
The image $\overline\gamma$ of the ray in $X/K_W$ lies entirely in $T_W \setminus N_{100\delta}\left(\overline{S}_W\right)$.  It follows from Corollary \ref{cor:fiftydeltaisometry} that $\overline\gamma$ is a $50\delta$--local geodesic in the $\delta$--hyperbolic space $T_W$.
In particular,
$[\overline{\gamma}]$ represents a point of $(\partial T_W) \setminus \mathcal F $, and we define $\Theta(\xi)$ to be this point.

\begin{lemma}\label{lem:Theta_well_def}
  The map $\Theta$ is well-defined and continuous.
\end{lemma}
\begin{proof}
  Suppose $\gamma$ and  $\{\gamma_i\}_{i\in \bN}$ are $50\delta$--local geodesics in $X$ starting at $1$, so that $[\gamma_i]\to [\gamma]$ in $\partial X$.  We show that $\Theta$ is well-defined and continuous by showing that $[\overline{\gamma}_i]\to [\overline{\gamma}]$ in $\partial T_W$.  (To deduce the map is well-defined, take a constant sequence.)

  If all the rays are geodesic, they stay within $\theta$ of each other on larger and larger initial subsegments.  Using the inequality \eqref{eq:linear} above, it can be shown that for all but finitely many $i$, we have $d(\gamma_i(R_{\gamma_i}),\gamma(R_\gamma))\leq 24\theta$.  
If they are only $50\delta$--local geodesics, we can use the fact that they are $3\theta$--close to geodesics to get the bound
\[ d(\gamma_i(R_{\gamma_i}),\gamma(R_\gamma))\leq 104\theta <\delta \]
  for all but finitely many $i$.
  It follows that the rays $\overline{\gamma}$ and $\{\overline{\gamma}_i\}$ all start within $\delta$ of one another.  Moreover, for any large $t$, all but finitely many $\overline{\gamma}_i$ pass within $7\theta$ of $\overline{\gamma}(t)$.  It follows that the equivalence classes $\{[\overline{\gamma}_i]\}$ converge to $[\overline{\gamma}]$, as required.
\end{proof}

\begin{prop}\label{prop:Theta_covers}
  The map $\Theta$ is a covering map. The preimage of any given point is a $K_W$--orbit in $\partial X$.
\end{prop}

\begin{proof}
Let $\overline\xi\in (\partial T_W) \setminus \mathcal F$, and let $\overline\gamma$ be a geodesic ray in $T_W$ from $\overline 1$ to $\overline\xi$.  For $t_0$ sufficiently large there exists an open neighborhood $\overline U\subseteq(\partial T_W) \setminus \mathcal F$ of $\overline\xi$ with the property that any $50\delta$--local geodesic from $N_{100\delta}(\overline S_W)$ to $\overline U$ passes within $10\delta$ of $\overline b=\overline\gamma(t_0)$.

Let $\{b_g\}_{g\in K_W}$ be the preimage of $\overline b$, with indexing so that $gb_h=b_{gh}$ for all $g,h\in K_W$. Let $R_g$ be the set of all lifts to $X$ starting at $b_g$ of $50\delta$--local geodesic rays in $T_W$ starting at $\overline b$ and limiting to $\overline U$ (notice that $T_W$ is a subset of $X/K_W$ which is covered by $X$).  Since $\overline{S}_W$ is quasiconvex, and $\overline{b}$ starts away from $N_{100\delta}(\overline{S}_W)$, we can apply Corollary \ref{cor:fiftydeltaisometry} to imply that all such lifts are $50\delta$--local geodesic rays.  Let $U_g$ be the set of all limit points in $\partial X$ of elements of $R_g$.

 Clearly, for any $g,h\in K_W$ we have $gU_h=U_{gh}$. We now have to prove that
 \begin{enumerate}
  \item for each distinct $g,h\in K_W$ we have $U_g\cap U_h=\emptyset$,\label{item:disjointU}
  \item for each $g\in K_W$, $U_g$ is open and $\Theta|_{U_g}$ is a homeomorphism onto $\overline U$,\label{item:homeoU}
  \item $\Theta^{-1}(\overline U)=\bigcup_{g\in K_W} U_g$.
 \end{enumerate}

  To show item \eqref{item:disjointU}, notice that any $\gamma \in R_g$ has the property that the diameter of $\pi_{S_W}(\gamma)$ is at most $20\delta$. For $g,h$ distinct elements of $K_W$, the distance between $\pi_{S_W}(b_g)$ and $\pi_{S_W}(b_h)=hg^{-1}\pi_{S_W}(b_g)$ is at least $100\delta$ by Lemma \ref{lem:frontier_translation}. In particular, elements of $R_g$ cannot be asymptotic to elements of $R_h$, so $U_g\cap U_h = \emptyset$.
  
  Let us prove that $U_g$ is open.  Let $[\alpha]\in U_g$.  We may suppose $\alpha\in R_g$, so $\alpha$ is the lift of $\overline{\alpha}$ starting at $b_g$ and $\overline{\alpha}$ is a $50\delta$--local geodesic starting at $\overline{b}$ and limiting to a point in $\overline{U}$.  Since $\overline{U}$ is open, some standard basic neighborhood of $[\overline{\alpha}]$ is contained in $\overline{U}$.  In particular, for $t$ chosen sufficiently large, if $\eta$ is a $50\delta$--local geodesic starting at $\overline{b}$ and passing within $10\delta$ of $\overline{\alpha}(t)$, then $[\eta]\in \overline{U}$.

  Now consider a standard neigborhood $V$ of $[\alpha]\in \partial X$, the set of points represented by $50\delta$--local geodesic rays starting at $b_g$ and passing within $10\delta$ of $\alpha(t)$.  By the previous paragraph, all such rays are elements of $R_g$, so $V\subseteq U_g$.  Since $[\alpha]$ was arbitrary, this proves $U_g$ is open.

  Let us prove that $\Theta|_{U_g}$ is injective. Take $\gamma_1,\gamma_2\in R_g$ with distinct limit points in $\partial X$. Then for some smallest $t_1$, $d(\gamma_1(t_1),\gamma_2(t_1))\geq 20\delta$, which implies that the same holds for their projections to $T_W$ (see Corollary \ref{cor:fiftydeltaisometry}), which in turn implies that such projections have distinct limit points.
  
  Surjectivity of $\Theta|_{U_g}$ onto $\overline U$ and continuity of the inverse are clear from the definition via lifts. We proved item \eqref{item:homeoU}.
  
  We are left to prove that for any ray $\alpha$ in $X$ starting at $1$ with $\Theta(\alpha)\in\overline U$ we have $\alpha\in U_g$ for some $g$. Let $\alpha'$ be the subray of $\alpha$ that intersects $N_{100\delta}(S_W)$ at its starting point only. By the defining property of $\overline b$, the projection $\overline\alpha'$ of $\alpha'$ to $T_W$ passes within $10\delta$ of $\overline b$, which implies that $\alpha'$ passes within $10\delta$ of $b_g$ for some $g\in K_W$. Any geodesic ray starting at $b_g$ and asymptotic to $\alpha$ belongs to $R_g$ because its projection to $T_W$ is asymptotic to $\overline\alpha'$ which limits to $\overline U$. Hence, the limit point of $\alpha$ is in $U_g$, as required.
This completes the proof of Proposition \ref{prop:Theta_covers}.
\end{proof}

\begin{proof}[Proof of Theorem \ref{thm:trunc}.\eqref{trunc_cover}] We defined a map $\Theta\co(\partial X)\setminus \Lambda(K_W)\to (\partial T_W)\setminus \mathcal F$ in Subsection \ref{subsec:covering_map} and proved that it is a covering map in Proposition \ref{prop:Theta_covers}. The fact that this covering is regular with deck group $K_W$ may be seen as follows: For any $k\in K_W$ and $\xi\in (\partial X\setminus \Lambda(K_W))$ we have $\Theta(k\xi)=\Theta(\xi)$ since we can represent $\xi$ and $k\xi$ by rays $\gamma$ and $k\gamma$, so that the definition of $\Theta$ and Lemma \ref{lem:Theta_well_def} clearly give $\Theta(\xi)=\Theta(\gamma)=\Theta(k\gamma)=\Theta(k\xi)$. Hence, $K_W$ acts by deck transformations, and by the description of preimages of points given by Proposition \ref{prop:Theta_covers}, it acts transitively on preimages of points. Hence, the covering is regular with deck group exactly $K_W$.
\end{proof}

\subsubsection{Connectedness of $\partial T_W$}
The following result is Theorem \ref{thm:trunc}.\eqref{trunc_F codense}.

\begin{lemma}\label{lem:opendense}
  For any spiderweb $W$ associated to a sufficiently long filling, $\partial T_W \setminus \mc{F}$ is open and dense in $\partial T_W$.
\end{lemma}
\begin{proof}
  We have already remarked that $\mc{F}$ is a finite union of closed sets, so $\partial T_W\setminus \mc{F}$ is open.

  More specifically, the set $\mc{F}$ is a finite disjoint union of closed sets $F_1,\ldots,F_k$, where each $F_i$ is the limit set of some $\Sigma_c/K_c$.  Let $c_1,\ldots,c_k$ be representatives of the $K_W$--orbits of the points $c$ which occur, and write $\mc{F}_i = \Lambda(\Sigma_{c_i}/K_{c_i})$, so $\mc{F} = \sqcup \mc{F}_i$.

  Let $\xi\in \mc{F}_i$ be represented by a geodesic ray $\gamma$.  By quasiconvexity, we may assume that some tail of $\gamma$ is contained in the truncated image of $\sepH_{c_i}$.  Moreover, we may assume this tail is entirely horizontal.  For $N\in \bZ$ very large, we form a new, $50\delta$--local geodesic $\hat{\gamma}_N$ which agrees with $\gamma$ up to $t=N$, and then changes to a vertical path until it leaves the image of $\sepH_c$.  Using $\delta$--visibility, this path is close to a geodesic ray $\gamma_N$ which fellow travels $\gamma$ for time $N$, but tends to a point not in the limit set of $\mc{F}_i$.  Some of these $\gamma_N$ may end up in $\mc{F}_j$ for $j\neq i$, but each $\mc{F}_j$ is disjoint from some open neighborhood of $\mc{F}_i$, so this only happens for finitely many $N$.  It follows that $\xi$ is a limit of points in $\partial T_W\setminus \mc{F}$.
\end{proof}

We have now proved Theorem \ref{thm:trunc}.

\begin{cor}\label{cor:approxconnected}
  Suppose $W$ is associated to a sufficiently long filling and that $\partial X\setminus \Lambda(K_W)$ is connected.  Then $\partial T_W$ is connected.
\end{cor}
\begin{proof}
  Theorem \ref{thm:trunc}.\eqref{trunc_cover} implies that $\partial T_W\setminus \mc{F}$ is covered by $\partial X\setminus \Lambda(K_W)$.  
Thus connectedness of $\partial T_W\setminus \mc{F}$  follows from connectedness of  $\partial X\setminus \Lambda(K_W)$.  By Lemma \ref{lem:opendense}, $\partial T_W\setminus \mc{F}$ is open and dense in $\partial T_W$.  It follows that $\partial T_W$ is connected.
\end{proof}

\subsection{Choosing the visual metric}
We have already proved (Theorem \ref{thm:trunc}.\eqref{trunc_are_hyp}) that, for long fillings, the $T_W$ are all $\delta$--hyperbolic, for a fixed $\delta>0$.  Fix $\epsilon = \frac{1}{10\delta}$.  Fix $\kappa = \kappa(\epsilon,\delta)$ as in Proposition \ref{findvisualmetric}.

Fix a spiderweb $W$ with parameter $\theta$ associated to a filling sufficiently long that $T_W$ is $\delta$--hyperbolic and $\delta$--visual.
We denote the image in $T_W$ of $1\in X$ by $\overline{1}$.
Proposition \ref{findvisualmetric} implies that there exists a visual metric $\rho_W(\cdot,\cdot)$ on $\partial T_W$ based at $\overline{1}$ with parameters $\epsilon,\kappa$. This proves Theorem \ref{thm:existsvisual}.\eqref{exists_visual}.

\subsection{Convergence}\label{sec:converge}
 In this subsection, we prove Theorem \ref{thm:existsvisual}.\eqref{exists_GH}, which states that the visual metrics constructed in the last subsection weakly Gromov--Hausdorff converge to a visual metric on $\partial T_{\overline{G}}$.

\begin{proof}[Proof of Theorem \ref{thm:existsvisual}.\eqref{exists_GH}]
Using Proposition \ref{prop:strongconverge}, it is enough to show that the $T_{W_j}$ strongly converge to $T_{\overline{G}}$, a space on which $\overline{G}$ acts geometrically. Fix any $R\geq 0$ and let $j$ be large enough that $W_j$ contains the ball $B$ of radius $2R$ in $X$, and moreover whenever $x,y\in B$ are in the same $K$--orbit then they are in the same $K_{W_j}$--orbit. The latter property can be arranged since there are only finitely many $k\in K$ so that there exists $x\in B$ with $kx\in B$.

We will show that there exists a locally isometric bijection $b$ that preserves lengths of paths from the ball $B_j$ of radius $2R$ around $\overline{1}$ in $T_{W_j}$ to the ball $\overline{B}$ of radius $2R$ around $\overline{1}$ in $T_{\overline G}$. Such bijection restricts to an isometry on the corresponding balls of radius $R$, proving strong convergence since $R$ was arbitrary.

Before defining $b$, notice that there are covering maps $\Phi_j\co X\setminus \bigcup\{\mathcal H^{(t_c,\infty)}_c\mid c\in \calC(W)\}\to T_{W_j}$ and $\Phi\co X\setminus \bigcup\{\mathcal H^{(t_c,\infty)}_c\mid c\in \calC\}\to T_{\overline G}$ (the domains of the two covering maps are obtained from $X$ by removing different sets of horoballs).

The map $b$ is defined by $b(x)=\Phi(\Phi_j^{-1}(x))$. In order to show that it is well defined we have to show that $\Phi_j^{-1}(x)$ is contained in the domain of $\Phi$ and that any point in $\Phi_j^{-1}(x)$ has the same image under $\Phi$. In order to show the former property notice that, since we can lift geodesics from $T_{W_j}$ to $X\setminus \bigcup\{\mathcal H^{(t_c,\infty)}_c\mid c\in \calC(W)\}$ to paths of the same length, $B_j$ is contained in $\Phi_j(B)$, so that $\Phi_j^{-1}(x)\subseteq K_{W_j}B\subseteq W$. Hence, if by contradiction we had some $y\in\Phi_j^{-1}(x)\cap \mathcal H^{(t_c,\infty)}$ for some $c\in \calC$ then we would actually have $c\in\calC(W)$, but clearly $\Phi_j^{-1}(x)\cap \mathcal H^{(t_c,\infty)}$ in that case. The latter property just follows from the fact that if two points of $X$ are in the same $K_W$--orbit then they are in the same $K$--orbit.

From the fact that $b$ is well-defined and the fact that $\Phi$ and $\Phi_j$ are covering maps it follows that $b$ is a local isometry. Injectivity of $b$ follows from the fact that if two points $p,q$ of $K_{W_j}B$ are in the same $K$--orbit (i.e. $\Phi(x)=\Phi(y)$) then they are in the same $K_{W_j}$--orbit (i.e. $\Phi_j(x)=\Phi_j(y)$). Surjectivity of $b$ follows from the following argument. If $y$ lies in $\overline{B}$, then we can lift a geodesic from $\overline{1}$ to $y$ to a path $\tilde\gamma$ in $X\setminus \bigcup\{\mathcal H^{(t_c,\infty)}_c\mid c\in \calC\}$ of length at most $2R$ starting at $1$. If $x$ is the endpoint of $\Phi_j\circ\tilde\gamma$, then it is readily checked that $b(x)=y$. The proof that $b$ is a locally isometric bijection is complete.
\end{proof}

We have now completed the proof of Theorem \ref{thm:existsvisual}.

\subsection{Linear connectedness}
In this subsection we prove Theorem \ref{thm:linconn} about uniform linear connectedness.
\linearconnectedness*

We must show that our approximating spaces $T_W$ have Gromov boundaries which are uniformly linearly connected.  We first reformulate the linear connectedness condition in terms of joining points by ``discrete paths'' of points which are at least a bit closer. The following is similar to the last part of the proof of \cite[Proposition 4]{BonkKleiner05}.
\begin{lemma}\label{chain}
 Let $M$ be a compact metric space. Suppose that there exists $L\geq 1$ so that each $p,q\in M$ can be joined by a chain of points $p=p_1,\dots,p_n=q$ so that $\diam(\{p_1\ldots,p_n\})\leq L d(p,q)$ and $d(p_i,p_{i+1})\leq d(p,q)/2$. Then $M$ is $5L$--linearly connected. 
\end{lemma}

\begin{proof}
Let $p,q\in M$. We can construct a chain of points interpolating between $p,q$, and then a ``finer'' one by interpolating between consecutive points of the first chain, and so on. Formally, we can construct by induction on $i$ sequences of points $\mathcal Q_i=\{q^i_j\}_{j=0,\dots,n(i)}$ with
\begin{itemize}
 \item $\mathcal Q_0=\{p,q\}$,
 \item $\mathcal Q_i\subseteq \mathcal Q_{i+1}$,
 \item $q^i_0=p,q^i_{n(i)}=q$
 \item $d(q^{i+1}_j,\mathcal Q_{i})\leq L d(p,q)/2^i$,
 \item $d(q^i_j,q^i_{j+1})\leq d(p,q)/2^i$.
\end{itemize}

Define $\mathcal Q$ to be the closure of $\bigcup \mathcal Q_i$, and notice $p,q\in \mathcal Q$. Also, it is easily seen that for each $q^i_j$ we have $d(\{p,q\},q^i_j)\leq \sum_{m=0}^{i-1} (L d(p,q)/2^m)\leq 2Ld(p,q)$, so that $diam(\mathcal Q)\leq 5L d(p,q)$. Finally, $\mathcal Q$ is connected because if not one could write $\mathcal Q$ as a union of disjoint non-empty clopen sets $A,B$. By compactness we have $d(A,B)=\epsilon >0$. Also, since $\bigcup \mathcal Q_i$ is dense in $\mathcal Q$, both $A$ and $B$ intersect $\mathcal Q_i$ for each sufficiently large $i$. However, for sufficiently large $i$ and for any $q^i_{j_1},q^i_{j_2}$ there exists a chain of points in $\mathcal Q_i\subseteq \mathcal Q$ connecting $q^i_{j_1},q^i_{j_2}$ where consecutive points are within distance $\epsilon/2$ of each other, in contradiction with the decomposition $\mathcal Q=A\sqcup B$.
\end{proof}

In the specific case that $M = \partial Z$ for a Gromov hyperbolic space $Z$, and $\partial Z$ is equipped with a visual metric, we can translate this criterion into one about geodesic rays. Since there are many constants involved, we briefly explain their roles. First of all, $\delta,\kappa,\epsilon$ are just the usual constants associated to a hyperbolic space. Secondly, $\lambda$ needs to be large enough to ensure that, in a sequence of rays interpolating between two given ones $\gamma_1,\gamma_2$, the distance between the limit points of consecutive rays is at most one half of the distance between the limit points of $\gamma_1,\gamma_2$. Finally, the constant $S$ will be the one determining the eventual linear connectedness constant, which is $L$.

\begin{figure}[htbp]
  \centering
  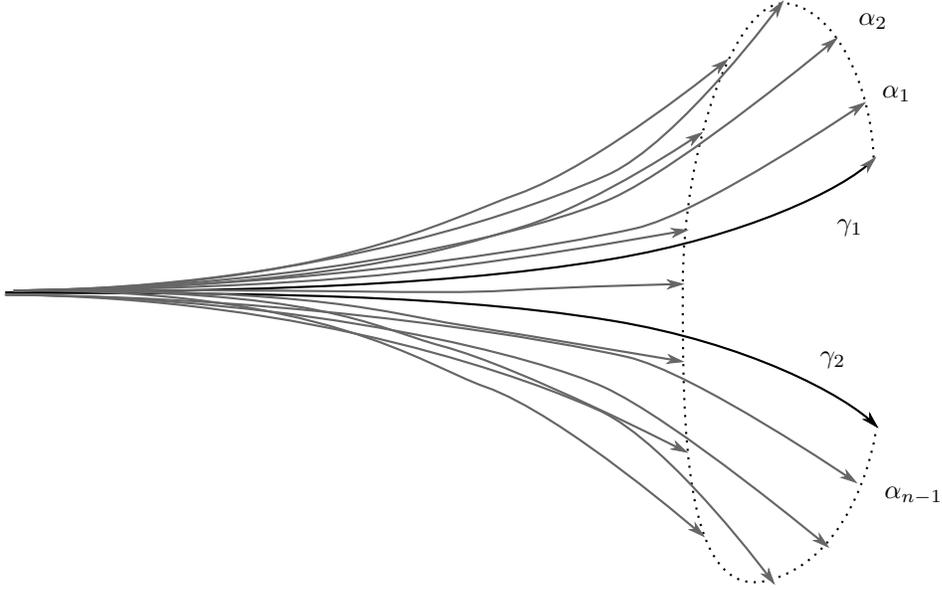
  \caption{The criterion of Lemma \ref{chain} translates into a statement (Lemma \ref{chain2}) about rays with certain Gromov products.  Large Gromov product corresponds to small distance in the boundary.}
  \label{fig:interpolate}
\end{figure}

\begin{lemma}\label{chain2}
Let $\delta,\kappa,\epsilon>0$ and let $\lambda>\ln(2\kappa^2)/\epsilon+10\delta$. For every $S$ there exists $L$ with the following property. Let $Z$ be  $\delta$--hyperbolic, with a basepoint $w$ and a visual metric $\rho$ on $\partial Z$ based at $w$ with parameters $\epsilon,\kappa$.
 Also, suppose that for each pair of rays $\gamma_1,\gamma_2$ starting at $w$ there exists a chain $\gamma_1=\alpha_1,\dots,\alpha_n=\gamma_2$ of rays starting at $w$ with $(\alpha_i|\alpha_{i+1})_w\geq (\gamma_1|\gamma_2)_w+\lambda$ and $(\alpha_i|\gamma_1)_w\geq (\gamma_1|\gamma_2)_w-S$. Then $(\partial Z,\rho)$ is $L$--linearly connected.
\end{lemma}
\begin{proof}
  Notice that $\lambda$ satisfies
\begin{equation}\label{eq:lambda1}
   \kappa^2 e^{-\epsilon\lambda + 10\epsilon\delta} < \frac{1}{2} .
\end{equation}
  Now fix $S$, and let $L = 10 \kappa^2 e^{\epsilon S+10\epsilon\delta}$.

  We check the criterion in Lemma \ref{chain}.  Let $Z$ be $\delta$--hyperbolic, let $w$ be a basepoint, and let $\rho(\cdot,\cdot)$ be a visual metric as in the statement of the lemma.
  Fix $p$, $q$ in $\partial Z$, which we represent by rays $\gamma_p$, $\gamma_q$ respectively.  Let $\{\alpha_i\}_{i=1\ldots n}$ be a chain of rays with $\alpha_1=\gamma_p$, $\alpha_n = \gamma_q$, and satisfying $(\alpha_i|\alpha_{i+1})_w \geq (\gamma_1|\gamma_2)_w+\lambda$ and $(\alpha_i|\gamma_1)_w\geq (\gamma_1|\gamma_2)_w-S$.  Let $p_i\in \partial Z$ be the equivalence class of $\alpha_i$.  

  We observed in Section \ref{s:Prelim} that Gromov products at infinity can be computed, up to a small error, using representative rays.  In particular, since $(\alpha_i|\alpha_{i+1})_w \geq (\gamma_p|\gamma_q)_w+\lambda$, we have
  $(p_i|p_{i+1})_w \geq (p|q)_w+\lambda-10\delta$, so
  \begin{equation}
    \rho(p_i,p_{i+1})\leq \kappa e^{-\epsilon(p|q)_w-\epsilon\lambda + 10\epsilon\delta}\leq \kappa^2e^{-\epsilon\lambda+10\epsilon\delta}\rho(p,q)<\frac{1}{2}\rho(p,q).
  \end{equation}
  Similarly, for any $p_i$, we have 
  \begin{equation}
    \rho(p_i,p)\leq \kappa^2e^{\epsilon S+ 10\epsilon\delta}\rho(p,q)= \frac{L}{10}\rho(p,q).
  \end{equation}
  Thus the diameter of the set $\{p_1,\ldots p_n\}$ is at most $\frac{L}{5}\rho(p,q)$.  Since $p$, $q$ were arbitrary, Lemma \ref{chain} implies that $(\partial Z,\rho)$ is $L$--linearly connected.
\end{proof}

The following lemma provides a converse to Lemma \ref{chain2} by allowing us to construct a sequence of rays starting from an arc in the boundary.

\begin{lemma}\label{lem:arc_to_rays}
Let $Z$ be hyperbolic and suppose that $\partial Z$, when endowed with a visual metric based at $w\in Z$, is linearly connected. Then there exists $R>0$ so that for every $C>0$ and every pair of rays $\gamma_1,\gamma_2$ in $Z$ starting at $w$ there exists a sequence of rays $\gamma_1=\alpha_1,\dots,\alpha_n=\gamma_2$ starting at $w$ with $(\alpha_i|\alpha_{i+1})_w\geq (\gamma_1|\gamma_2)_w+C$ and $(\alpha_i|\gamma_1)_w\geq (\gamma_1|\gamma_2)_w-R$.
\end{lemma}

\begin{proof}
 Denote by $\delta$ a hyperbolicity constant for $Z$ and fix $w\in Z$. Then there exist $\epsilon,\kappa,L$ and a visual metric $\rho$ based at $w$ with parameters $\epsilon,\kappa$ so that $(\partial Z,\rho)$ is $L/2$--linearly connected. Set $R=\log(\kappa^2L)/\epsilon+20\delta$. Fix now any $C,\gamma_1,\gamma_2$ as in the statement. Denoting $p_1,p_2\in\partial Z$ the limit points of $\gamma_1,\gamma_2$, there exists an arc $I$ connecting $p_1$ to $p_2$ and with diameter $\leq L\rho(p_1,p_2)$.  Let $\alpha$ be a ray from $w$ to a point $p\in I$.  Approximating Gromov products of points at infinity by Gromov products of rays, we have
 $$e^{-\epsilon(\gamma_1|\alpha)_w}\leq \kappa e^{10\epsilon\delta} \rho(p_1,p)\leq \kappa e^{10\epsilon\delta}L\rho(p_1,p_2)\leq \kappa^2 e^{20\epsilon\delta}Le^{-\epsilon(\gamma_1|\gamma_2)_w},$$
 from which we deduce $(\gamma_1|\alpha)_w\geq (\gamma_1|\gamma_2)_w-\log(\kappa^2L)/\epsilon-20\delta=(\gamma_1|\gamma_2)_w-R$. A similar computation shows that whenever $p,q\in\partial Z$ are close enough, any rays $\gamma_p,\gamma_q$ from $w$ to $p,q$ satisfy $(\gamma_p|\gamma_q)_w\geq (\gamma_1|\gamma_2)_w+C$. Hence, by a simple compactness argument, we can find a sequence of points $p_1=a_1,\dots,a_n=p_2$ contained in $I$ so that, for any choice of rays $\alpha_i$ from $w$ to $a_i$, $\{\alpha_i\}$ provides the required sequence of rays.
\end{proof}

In the current work, we only need the following proposition for a particular value of $R$.  However we believe the more general form given will be useful in future work.
Recall that $\overline S_W$ denotes the truncated quotient of the saturated spiderweb by $K_W$ (see Definition \ref{defn:saturated_spiderweb}), while $T_{\overline G}$ denotes the quotient of the cusped space minus certain horoballs by $K$ (see Definition \ref{defn:trunc_quotient}). Roughly speaking, we show that a large neighborhood of $\overline S_W$ in $T_W$ isometrically embeds in $T_{\overline G}$.  This is a stronger version of the strong convergence property we used in \ref{thm:existsvisual}.\eqref{exists_GH}.

\begin{prop}\label{lem:neigh_embed}
Let $(G,\mc{P})$ be relatively hyperbolic, and let $X$ be the corresponding $\theta$--hyperbolic cusped space. Then for every $R$ the following holds. For all sufficiently long hyperbolic fillings $G\to \overline{G}$ and every $\theta$--spiderweb $W$ associated to the filling, $N_R(\overline S_W)$ isometrically embeds into $T_{\overline G}$. 
More precisely:  Let $\Phi_W\co X\setminus \bigcup\{\mathcal H^{(t_c,\infty)}_c\mid c\in \calC(W)\}\to T_W$ and $\Phi\co X\setminus \bigcup\{\mathcal H^{(t_c,\infty)}_c\mid c\in \calC\}\to T_{\overline{G}}$ be the natural covering maps. Then there exists an isometric embedding $\iota\co N_R(\overline S_W)\to T_{\overline G}$ so that $\iota\circ\Phi_W=\Phi$ where both sides are defined (in particular, $\iota(1)=1$). Moreover, the image of $\iota$ is $N_R(\Phi(S_W\setminus \bigcup \{\mathcal H^{(t_c,\infty)}_c\mid c\in \calC\}))$.
\end{prop}

\begin{proof}
Our proof rests on the following claim.
\begin{claim*}
 For every $R_0$ and every sufficiently long filling the following holds. For every spiderweb $W$, whenever $g\in K$ and $x\in X$ are so that $x,gx\in X$ lie in the $R_0$--neighborhood of $S_W$, we have $g\in K_W$.
\end{claim*}
 
 Let us assume the claim and fix some $R\geq 10\delta$. We set $R_0=R+22\delta$ and assume 
that we are considering a filling sufficiently long that the conclusion of the claim holds and so that $t_c\geq R_0+500\theta+1$ for every $c\in \mc{C}$ (see Corollary \ref{c:t_c big}).

 Let $b\co  N_{R_0}(\overline S_W)\to T_{\overline G}$ be the map defined by $b(x)=\Phi(\Phi_W^{-1}(x))$. First of all, let us check that $\Phi_W^{-1}(x)$ is contained in the domain of $\Phi$, and that $\Phi(\Phi_W^{-1}(x))$ consists of a single point, so that $b$ is well-defined. The first property follows from the fact that we can lift any geodesic from $x$ to $\overline S_W$ to a path of the same length in $X$, showing that $\Phi_W^{-1}(x)$ is contained in $N_{R_0}(S_W)$, where the neighborhood is taken in $X$. If by contradiction we had $y\in \Phi_W^{-1}(x)\cap \calH_c^{(t_c,\infty)}$ for some $c$ we would then have $\calH_c^{[500\theta,\infty)}\cap S_W\neq\emptyset$, since $t_c\geq R_0+500\theta+1$, and hence $c\in\calC(W)$. But then clearly $\Phi_W^{-1}(x)\cap \calH_c^{(t_c,\infty)}=\emptyset$, a contradiction. The fact that $\Phi(\Phi_W^{-1}(x))$ consists of a single point just follows from the fact that if two points are in the same $K_W$--orbit then they are in the same $K$--orbit.

We will now show that $b$ is a locally isometric bijection onto its image. The fact that it is locally isometric easily follows from the fact that $\Phi_W$ and $\Phi$ are covering maps (and the fact that it is well-defined). Injectivity follows from the Claim, and the fact that $\Phi_W^{-1}(x)$ is contained in $N_{R_0}(S_W)$ for each $x\in N_{R_0}(\overline S_W)$, as we argued above. 

What is more, we claim that for any ball $B=B_{20\delta}(x)$ in $T_W$ centered at some $x\in N_{R+2\delta}(\overline S_W)$, $b|_B$ is a surjection onto the ball $B_{20\delta}(b(x))$. In particular, $b$ restricts to an isometry between balls of radius $10\delta$ with the same centers. The reason for surjectivity is simply that we can define an inverse by lifting to $X$ geodesics from $b(x)$ to other points in $B_{20\delta}(b(x))$ and push them to $T_W$ using $\Phi_W$, obtaining paths of length at most $20\delta$ which therefore have endpoints in $B$.

Let $\hat{S}=b(N_{R}(S_W))$.
From what we proved so far, it follows that any pair of points in $\hat{S}_W$ is connected by a $10\delta$--local geodesic contained in $b(N_{R+2\delta}(\overline S_W))$. Since any $10\delta$--local geodesic stays within $2\delta$ of any geodesic with the same endpoints (see \cite[III.H.1.13]{BH}), we get that $\hat{S}_W$ is $4\delta$--quasiconvex.  (We implicitly used $b(N_{R+2\delta}(\overline S_W))\subseteq N_{2\delta}(\hat{S}_W)$, which follows from the fact that $b$ is $1$--Lipschitz since it is locally isometric.)

Let us now that prove that $\iota=b|_{N_R(\overline S_W)}\co N_R(\overline S_W)\to \hat{S}_W$ is an isometry. Since it is $1$--Lipschitz, we are left to show that $d(x,y)\leq d(b(x),b(y))$ for each $x,y\in N_R(\overline S_W)$. This holds because any geodesic $\gamma$ from $b(x)$ to $b(y)$ is contained in $N_{4\delta}(\hat S_W)$, which in turn is contained in $b(N_{R_0}(\overline S_W))$ (this follows from the statement about $10\delta$--balls above). In particular, $x$ and $y$ are connected by a path of length at most $d(x,y)$, namely $b^{-1}(\gamma)$, as required.

Finally, to prove the ``moreover'' part one just needs to once again consider lifts of geodesics to $\Phi(S_W\setminus \bigcup \{\mathcal H^{(t_c,\infty)}_c\mid c\in \calC\})$.

 We now prove the claim.
\begin{proof}[Proof of Claim]
 Choose a filling sufficiently long that Lemma \ref{greendlinger} applies with $D=R_0+10^3\theta$.  We argue by contradiction, assuming that $x,gx$ provide a counterexample.
Since the $K$--orbit of $x$ is discrete, there exists $g'\in K_Wg$ so that $d(x,g'x)$ is minimal (notice that we still have $g'x\in N_{R_0}(S_W)$). Also, $g'\neq 1$ because we are assuming $g\notin K_W$. By Lemma \ref{greendlinger}, any geodesic $[x,g'x]$ intersects some horosphere $\calH^D_c$. Since $S_W$ is $6\theta$--quasiconvex (Lemma \ref{lem:saturation_quasiconvex}), such geodesic is contained in $N_{R_0+10\theta}(S_W)$, implying that $S_W$ intersects $\calH^{500\theta}_c$. In turn, this implies that we have $K_c<K_W$. But then, for $k\in K_c$ as in Lemma \ref{greendlinger}, we have $d(x,kg'x)<d(x,g'x)$, contradicting the minimality of $d(x,g'x)$.
\end{proof}
Having proved the claim, the proof of Proposition \ref{lem:neigh_embed} is complete.
\end{proof}

The following elementary lemma is useful in the proof of Theorem \ref{thm:linconn}.  The notation $A\sim_CB$ for quantities $A$ and $B$ indicates $A\in [B-C,B+C]$.
\begin{lemma}\label{lem:estimategp}
  Let $p_0, p_1$ be points in a $\delta$--hyperbolic space.  For $i\in \{0,1\}$, let $\alpha_i,\beta_i$ be geodesic rays based at $p_i$, so that $\alpha_0$ is asymptotic to $\alpha_1$ and $\beta_0$ is asymptotic to $\beta_1$.  Suppose further that the Gromov products $(p_0|\beta_1(t))_{p_1}$ and $(p_0|\alpha_1(t))_{p_1}$ are bounded above by a constant $C$ for every large enough $t$.  Then $(\alpha_0|\beta_0)_{p_0}\sim_{2C+8\delta} (\alpha_1|\beta_1)_{p_1}+d(p_0,p_1)$.  
\end{lemma}
\begin{proof}
  See Figure \ref{fig:gromovprod}.
  \begin{figure}[htbp]
    \centering
    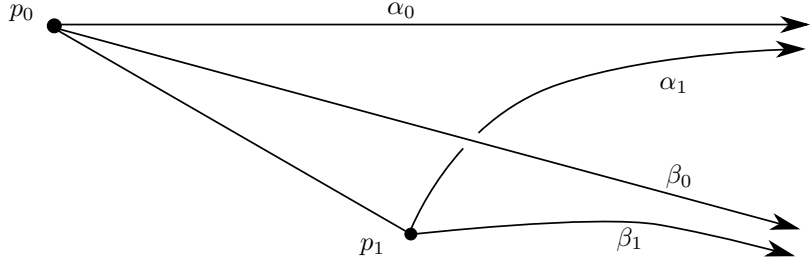
    \caption{Estimating the Gromov product at $p_0$ in terms of the one at $p_1$.}
    \label{fig:gromovprod}
  \end{figure}
  Choose points $a_i\in \alpha_i$ and $b_i\in \beta_i$ far away from $p_0$ and $p_1$ so that $d(a_0,a_1)\leq 2\delta$, $d(b_0,b_1)\leq 2\delta$, and so that $(a_i|b_i)_{p_i}\sim_{2\delta}(\alpha_i|\beta_i)_{p_i}$.  

  Now notice that $d(p_1,\alpha_0)$ and $d(p_1,\beta_0)$ are at most $C+2\delta$.  It follows that $d(a_0,p_0)\sim_{2C+6\delta} d(a_1,p_1)+d(p_0,p_1)$, and similarly $d(b_0,p_0)\sim_{2C+6\delta} d(b_1,p_1)+d(p_0,p_1)$.  Combining this with the fact that $d(a_0,b_0)\sim_{4\delta}d(a_1,b_1)$, we get the desired estimate.
\end{proof}

  Recall that Theorem \ref{thm:linconn} says that, for sufficiently long one-ended hyperbolic fillings and any spiderweb $W$ associated to such filling, the $\partial T_W$ have visual metrics $\rho_W$ (of uniform parameters) which are uniformly linearly connected.  We only expect uniformity over spiderwebs associated to a fixed filling, not uniformity over fillings.
\begin{proof}[Proof of Theorem \ref{thm:linconn}]
The idea here will be to build ``discrete paths'' joining any two points at infinity.  This means building, between any two rays to infinity, a sequence of interpolating rays satisfying the hypothesis of Lemma \ref{chain2}.  Given a pair of rays in $T_W$, there will be two cases, depending on whether the rays begin to diverge far from $\overline{S}_W$ or not.  In the first case, we will exploit the linear connectedness of $\partial X$; in the second the linear connectedness of $\partial \overline{G}$.

We must fix some constants before choosing a filling.  As before $\delta = 1500\theta$, $\epsilon = \frac{1}{10\delta}$ and  $\kappa$ are the constants (which depend only on $\delta$) from Theorem \ref{thm:existsvisual}.  Fix $\lambda > \ln (2\kappa^2)/\epsilon + 10\delta$ as in Lemma \ref{chain2}.

By hypothesis $\partial X$ is linearly connected.  Recall that Lemma \ref{lem:arc_to_rays} provides, for a hyperbolic space $Z$ and a basepoint $w\in Z$, a constant $R$ which governs the behavior of ``discrete paths'' of geodesic rays based at $w$, interpolating between two given rays.
\begin{claim*}
  There is a number $R_X$ so that the conclusion of Lemma \ref{lem:arc_to_rays} applies with $R=R_X$ and $w$ any vertex of $X$ at depth less than $201\delta$.  
\end{claim*}
\begin{proof}
  There are finitely many $G$--orbits of vertices in $X$ of bounded depth.
\end{proof}
We fix such an $R_X$.

Now fix a filling $G \to \overline{G}$ so that all the following hold, for every spiderweb $W$ associated to the filling:
\begin{enumerate}
\item The truncated quotient $T_W$ is $\delta$--hyperbolic and $\delta$--visual (Theorem \ref{thm:trunc}).
\item The boundary $\partial T_W$ carries a visual metric $\rho_W$ based at $\overline{1}$ with parameters $\epsilon,\kappa$ (Theorem \ref{thm:existsvisual}).
\item The neighborhood $N_{R_X+\lambda+10^6\delta}(\overline S_W)$ isometrically embeds in $T_{\overline G}$ (Proposition \ref{lem:neigh_embed}).
\item The Assumptions \ref{assume:trunc_cover} hold.  In particular Lemma \ref{lem:frontier_translation} and Corollary \ref{cor:fiftydeltaisometry} hold.
\end{enumerate}
By assumption $\overline{G}$ is one-ended, so $\partial T_{\overline G} \cong \partial\overline G$ is linearly connected \cite[Proposition 4]{BonkKleiner05}.  We let $R_{\overline G}$ be the constant $R$ from Lemma \ref{lem:arc_to_rays} applied to a visual metric on $\partial T_{\overline G}$ based at $\overline 1$.

Finally we fix a spiderweb $W$ associated to this filling. Recall that we have a natural covering map $\Phi_W\co X\setminus \bigcup\{\mathcal H^{(t_c,\infty)}_c\mid c\in \calC(W)\}\to T_W$. The following lemma will allow us to move back and forth more easily between geodesic rays in $T_W$ and geodesic rays in $X$.
\begin{lemma}\label{lem:liftproject}
  Let $\gamma$ be a path in $X$ which avoids the $102\delta$--neighborhood of $S_W$, and let $\bar\gamma=\Phi_W\circ\gamma$ be the projection to $T_W$.  Then $\gamma$ is geodesic if and only if $\bar\gamma$ is geodesic.
\end{lemma}
\begin{proof}
  Our argument is based on the following claim.
  \begin{claim*}
    If $\bar\sigma$ is a $T_W$--geodesic lying outside the $100\delta$--neighborhood of $\overline S_W$, then any lift $\sigma$ of $\bar\sigma$ to $X$ is a geodesic.
  \end{claim*}
  \begin{proof}[Proof of Claim]
    By Corollary \ref{cor:fiftydeltaisometry}, $\sigma$ is a $50\delta$--local geodesic.  Let $\sigma'$ be a geodesic with the same endpoints.  The space $X$ is $\theta$--hyperbolic, so $\sigma'$ lies in a $2\theta$--neighborhood of $\sigma$, by \cite[III.H.1.13]{BH}.  In particular, $\sigma'$ lies in the domain of $\Phi_W$.  If $\sigma$ were not geodesic, $\sigma'$ would have strictly smaller length, and would project to a path $\bar\sigma'$ with the same endpoints as $\bar\sigma$, contradicting the assumption that $\bar\sigma$ was geodesic.
  \end{proof}
  Now let $\gamma$ be a path in $X$ avoiding the $102\delta$--neighborhood of $S_W$, and let $\bar\gamma$ be the projection of $\gamma$ to $T_W$.  It follows that $\bar\gamma$ avoids the $102\delta$--neighborhood of $\overline S_W$.

  One direction of the Lemma is immediate from the Claim; if $\bar\gamma$ is geodesic, then so is $\gamma$.

  In the other direction, suppose that $\gamma$ is geodesic.  Since the points of $\gamma$ lie outside $N_{100\delta}S_W$, we can apply Corollary \ref{cor:fiftydeltaisometry} to deduce that $\bar\gamma$ is a $50\delta$--local geodesic in $T_W$.  The endpoints of $\bar\gamma$ are therefore joined by a geodesic $\bar\sigma$ which lies in a $2\delta$--neighborhood of $\bar\gamma$, again using \cite[III.H.1.13]{BH}.  Thus $\bar\sigma$ lies outside $N_{100\delta} \overline{S}_W$.  Let $\sigma$ be a lift of $\bar\sigma$ with the same initial point as $\gamma$.  The Claim implies that $\sigma$ is a geodesic.  

We now claim that $\sigma$ has the same terminal point as $\gamma$.  Indeed, let $p$ be the terminal point of $\gamma$ and let $q$ be the terminal point of $\sigma$, and suppose $p\neq q$.  Since $p$ and $q$ project to the same point in $T_W$, there must be some $k\in K_W\setminus \{1\}$ so that $q = kp$.
Let $p'\in \pi_{S_W}(p)$, and let $q' = kp'\in \pi_{S_W}(q)$.  Lemma \ref{lem:frontier_translation} implies that $d(p',q')>100\delta$.  Let $\eta$ be a geodesic joining $p$ to $q$.  Then $\eta \subset N_\delta(\gamma\cup\sigma)$ lies outside the $99\delta$--neighborhood of $S_W$.  The set $S_W$ is $6\theta$--quasiconvex by Lemma \ref{lem:saturation_quasiconvex}, so we can apply Lemma \ref{quasiconvproj} to deduce that the diameter of $\pi_{S_W}(\eta)$ is at most $9\delta$, contradicting $d(p',q')>100\delta$.

Since $\sigma$ and $\gamma$ are geodesics with the same endpoints, they have the same length.  It follows that $\bar\gamma$ has the same length as the geodesic $\bar\sigma$, and is therefore geodesic in $T_W$.
\end{proof}

We now begin the main argument, which is a verification of the hypothesis of Lemma \ref{chain2} for the space $T_W$ with $S= \max\{ R_X + 100\delta , R_{\overline{G}} + 10^3\delta \}$.  Accordingly, we fix $\overline\gamma_1,\overline\gamma_2$ a pair of rays based at $\overline 1\in T_W$, and look for a sequence of interpolating rays $\overline\alpha_i$ as in Lemma \ref{chain2}.  Let $t_1 = (\overline\gamma_1|\overline\gamma_2)_{\overline 1}$.

\setcounter{case}{0}
\begin{case}
  $d(\overline\gamma_1(t_1),\overline S_W)\geq R_X+10^5\delta$.
\end{case}
Let $t_0 = \sup\{t\mid d(\overline\gamma_1(t),\overline S_W)\leq 200\delta\}$, and let $\overline x = \overline\gamma_1(t_0)$.  We note that the depth of $\overline x$ is bounded by $500\theta + 200\delta<201\delta$.

We let $\overline\gamma_1'$ be the restriction of $\overline\gamma_1$ to $[t_0,\infty)$.  Let $D = t_1-t_0$, and note that $D\geq R_X+(10^5-200)\delta$.  Let $\overline\gamma_2'$ be a broken geodesic following $\overline\gamma_1'$ for distance $D$, and then following a geodesic ray asymptotic to $\overline\gamma_2$.  Let $\overline T$ be the tripod $\overline\gamma_1'\cup\overline\gamma_2'$, and note that

\begin{claim}\label{claim:T_takes_off}
 All points of $\overline T$ are distance at least $200\delta$ from $\overline S_W$.
\end{claim}

\begin{proof}[Proof of Claim]
 This is because otherwise there would be points $x_0,x_1,x_2$ on $\overline\gamma_2$, appearing in the given order, so that $x_0$ lies at distance at most $201\delta$ from $\overline S_W$ (just pick $x_0$ within $\delta$ of $\overline{x}$), $x_1$ lies at distance at least $R_X+10^4\delta$ from $\overline S_W$ (pick $x_1$ within $10\delta$ of $\overline\gamma_1(t_1)$) and $x_2$ lies at distance at most $201\delta$ from $\overline S_W$ (pick $x_2$ $\delta$--close to a point on $\overline\gamma_2'-\{\overline x\}$ contained in the $200\delta$--neighborhood of $\overline S_W$). The existence of such a triple is easily seen to contradict the fact that $N_{201\delta}(\overline S_W)$ is $2\delta$--quasiconvex, since $\overline S_W$ is $3\delta$--quasiconvex (Lemma \ref{lem:overlineSW_qconv}).
\end{proof}

Since $\overline T$ is simply connected and $\Phi_W\co X\setminus\bigcup\{\calH_c^{(t_c,\infty)}\mid c\in \calC(W)\} \to T_W$ is a covering map, we can lift $\overline T$ to a tripod $T\subset X$.  By Lemma \ref{lem:liftproject}, the legs of this tripod are geodesic.
Let $\gamma_1$ be the lift of $\overline\gamma_1'$, and let $\gamma_2$ be a geodesic ray starting at the same point $x$, asymptotic to the lift of $\overline\gamma_2'$.

  We claim that the Gromov product $(\gamma_1|\gamma_2)_x$ is within $10\delta$ of $D$.  Indeed, this Gromov product can be estimated to within $2\delta$ using points on the tripod $T$.  The tripod $T$ (respectively its image $\overline T$) is $\delta$--quasiconvex, and lies outside a $200\delta$--neighborhood of $S_W$ (respectively $\overline S_W$), so Lemma \ref{lem:liftproject} implies the projection is isometric on $T$.  It's not hard to see that for $s$, $t$ sufficiently large, we have $(\overline\gamma_1(t)|\overline\gamma_2(s))_{\overline x} \sim_{3\delta} (\overline\gamma_1| \overline\gamma_2)_{\overline 1}-t_0 = D$.

The depth of $\overline x$ was at most $201\delta$, and so the depth of $x$ is at most $201\delta$.  It follows that there is a discrete path $\{\alpha_1,\ldots,\alpha_n\}$ of rays based at $x$ interpolating between $\gamma_1$ and $\gamma_2$, and satisfying:
\begin{enumerate}
\item $\alpha_1=\gamma_1$ and $\alpha_n = \gamma_2$;
\item  $(\alpha_i|\alpha_{i+1})_x\geq (\gamma_1|\gamma_2)_x + \lambda+100\delta$ for all $i$; and
\item $(\gamma_1|\alpha_i)_x\geq (\gamma_1|\gamma_2)_x - R_X>10^4\delta$ for all $i$.
\end{enumerate}

A similar argument to the one that proves that any point on $\overline T$ is at distance at least $200\delta$ from $\overline S_W$ proves the following claim.
\begin{claim}
  No $\alpha_i$ meets a $102\delta$--neighborhood of $S_W$.
\end{claim}

\begin{proof}[Proof of Claim]
First of all, it follows from $3\delta$--quasiconvexity of $\overline S_W$ that for each $t\geq t_0$ we have $d(\overline\gamma_1(t),\overline S_W)\geq t-t_0+198\delta$. In fact, this is easily deduced from the fact that $\overline x$ lies within $\delta$ of any geodesic from $\overline\gamma_1(t)$ to $\overline S_W$.

Notice that if $p\in X-S_W$ then $d(p,S_W)=d(\Phi_W(p),\overline S_W)$, because we can project to $T_W$ a shortest geodesic from $p$ to $S_W$ and, vice versa, lift a shortest geodesic from $\Phi_W(p)$ to $\overline S_W$. In particular, for each $t\geq t_0$ we have $d(\gamma_1(t),S_W)\geq t-t_0+198\delta$.

In order to prove that $\alpha_i$ does not intersect the $102\delta$--neighborhood of $S_W$, we can now proceed similarly to Claim \ref{claim:T_takes_off} and argue that if that was not the case we could find 3 points along $\alpha_i$ so that the middle one is far away from $S_W$ but the other ones are close, contradicting quasiconvexity of $S_W$.
\end{proof}

It follows (using Lemma \ref{lem:liftproject} again) that the $\alpha_i$ project to geodesic rays $\overline\alpha_i'$ starting at $\overline x$.  We may prepend each such ray with the initial segment of $\overline\gamma_1$ terminating at $\overline x$, to obtain a broken geodesic $\overline\alpha_i''$ with Gromov product at $\overline x$ bounded above by $\delta$.  Let $\overline \alpha_1 = \overline\gamma_1$, and $\overline\alpha_n=\overline\gamma_2$.  For $i\notin\{1,n\}$, let $\overline\alpha_i$ be a geodesic ray beginning at $\overline 1$ and asymptotic to $\overline \alpha_i'$.  

Using Lemma \ref{lem:estimategp} for the second and last estimates we obtain
\begin{align*}
  (\overline\alpha_i|\overline\alpha_{i+1})_{\overline 1} & \sim_{2\delta}(\overline\alpha_i''|\overline\alpha_{i+1}'')_{\overline 1}\\
 & \sim_{10\delta}(\overline\alpha_i'|\overline\alpha_{i+1}')_{\overline x}+t_0\\
 &  = (\alpha_i|\alpha_{i+1})_x+t_0\\
 & \geq (\gamma_1|\gamma_2)_x+ \lambda + 100\delta + t_0 \\
 & \sim_{10\delta} (\overline\gamma_1|\overline\gamma_2)_{\overline x} + \lambda + 100\delta,
\end{align*}
  The total errors add up to less than $100\delta$, so we obtain 
\[ (\overline\alpha_i|\overline\alpha_{i+1})_{\overline 1} \geq (\overline\gamma_1|\overline\gamma_2)_{\overline x} + \lambda. \]
A similar computation yields, for each $i$,
\[ (\overline\gamma_1|\overline\alpha_i)_{\overline 1}\geq (\overline\gamma_1|\overline\gamma_2)_{\overline 1} - (R_X + 100\delta).\]

We have thus verified the hypothesis of Lemma \ref{chain2} in this case, with $S = S_1= R_X + 100\delta$.

\begin{case}
  $d(\overline\gamma_1(t_1),\overline S_W)< R_X+10^5\delta$.
\end{case}
Let $\overline\gamma'_i$ be the maximal initial subgeodesic of $\overline\gamma_i$ entirely contained in $N=N_{R+\lambda+10^6\delta}(\overline S_W)$.  Recall that by assumption $N$ is isometric to a subspace of $T_{\overline G}$, so let us now regard $N$ as a subspace of $T_{\overline G}$. Since $T_{\overline G}$ is $\delta$--visual, $\overline\gamma'_i$ is contained in the $2\delta$--neighborhood of some ray $\overline\gamma''_i$. There exists a sequence of rays $\overline\gamma''_1=\alpha''_1,\dots,\alpha''_n=\overline\gamma''_2$, all starting at $1\in \overline{G}$ so that $(\alpha''_i|\alpha''_{i+1})_1\geq (\overline\gamma''_1|\overline\gamma''_2)_1+\lambda+10^3\delta$ and $(\alpha''_i|\overline\gamma''_1)_1\geq (\overline\gamma''_1|\overline\gamma''_2)_1 -R_{\overline G}$. Let $\alpha'_i$ be the maximal initial subgeodesic of $\alpha''_i$ contained in $N$. We now switch back to thinking of $N$ as a subspace of $T_W$. Since $T_W$ is $\delta$--visual, there exist rays $\alpha_i$, starting at $1$, so that $\alpha'_i$ is contained in the $10\delta$--neighborhood of $\alpha_i$. We can take $\alpha_1=\overline\gamma_1,\alpha_n=\overline\gamma_2$.  It is now straightforward to check that $(\alpha_i|\alpha_{i+1})_1\geq (\overline\gamma_1|\overline\gamma_2)_1+\lambda$ and $(\alpha_i|\overline\gamma_1)_1\geq (\overline\gamma_1|\overline\gamma_2)_1 -R_{\overline G}-10^3\delta$.  

We have verified the hypothesis of Lemma \ref{chain2} in this case, with $S =S_2= R_{\overline G} + 10^3\delta$.

Taking $S$ to be the maximum of $S_1$ and $S_2$, we have verified the hypothesis of Lemma \ref{chain2} in both cases, and conclude using this lemma that $(\partial T_W,\rho_W)$ is linearly connected with constant independent of the spiderweb chosen.
\end{proof}

\section{Approximating boundaries are spheres}\label{sec:spheres}

\subsection{Statement and notation}\label{ss:setup}
In this section we fix $(G,\mc{P})$ relatively hyperbolic with $\mc{P}=\{P_1,\ldots,P_n\}$ where each $P_i$ is virtually $\mathbb{Z}\oplus\mathbb{Z}$. We let $X$ be a cusped space for the pair
and assume that $\partial(G,\mc{P})=\partial X$ is a $2$--sphere.  We also fix a Dehn filling $\pi\co G\to \overline{G}=G(N_1,\ldots,N_n)$ so that each $N_i$ is isomorphic to $\mathbb{Z}$, and suppose the filling is long enough to apply Theorem \ref{thm:trunc}.  For $\theta, \delta$ the constants in Theorem \ref{thm:trunc}, we consider a $\theta$--spiderweb $W$ associated to this filling (Definition \ref{def:spiderweb}).  The spiderweb is preserved by a finitely generated free group $K_W<\ker\pi$.  We denote the rank of $K_W$ by $k$.
The associated truncated quotient $T_W$ (Definition \ref{defn:trunc_quotient}) is $\delta$--hyperbolic by Theorem \ref{thm:trunc}.

In this section we describe the Gromov boundary of the truncated quotient:
\begin{prop}\label{Xiarespheres}
  With the above assumptions, $\partial T_W$ is homeomorphic to $S^2$.
\end{prop}
\subsection{Reduction to a homology computation}
Thanks to the following lemma, the proof of Proposition \ref{Xiarespheres} is reduced to a homology computation.

\begin{lemma}\label{lem:reduction_to_homology}
 Let $M$ be a compact Hausdorff space and let $S$ be a dense subset of $M$ homeomorphic to a surface with empty boundary.  Suppose that $m=\#(M\setminus S)$ is finite and that the dimension of $H_1(S,\mathbb Z/2)$ is at most $\max\{m-1,0\}$. Then $M$ is homeomorphic to $S^2$.
\end{lemma}

\begin{proof}
When we refer to `homology' in this proof we always mean homology with $\mathbb Z/2$ coefficients.

 First of all, we claim that $S$ is a surface of finite type.  Indeed, this follows from the fact that surfaces of infinite type have infinite dimensional first homology, as one can deduce from the classification of non-compact surfaces given in \cite[Theorem 3]{RichardsClassification}. 
 
 Let $p$ be the number of punctures of $S$.  If $p=0$, then $M=S$ is a compact surface with $H_1(M;\bZ/2)=0$, so $M\cong S^2$.
 
 Now suppose $p>0$, and let $\overline{S}$ be the closed surface obtained filling in the punctures of $S$.  Note that $\overline{S}$ is equal to the end-compactification of $S$.
 
  Since $M\setminus S$ is finite, $S$ is open and $M\setminus S$ is totally disconnected. Also, by assumption $M$ is compact and Hausdorff and $S$ is dense in $M$, and hence the universal property of end-compactifications \cite[Satz 6]{FreudenthalCompactification} gives us a map $h\co \overline S\to M$ restricting to the identity on $S$.  Since $S$ is dense in $M$ the map $h$ is surjective.  In particular, $m \leq p$. Moreover, if $d = \dim_{\bZ/2}H_1(\overline{S};\bZ/2)$ and $r = \dim_{\bZ/2}H_1(S;\bZ/2)$, then $r = d+p-1$.  By assumption $r\leq m-1$, so we must have $d=0$ and thus $\overline S\cong S^2$.  Finally $p-m = r+1-m\leq m-1+1-m=0$, again by assumption.  This shows that $h$ also restricts to a bijection between $\overline{S}\setminus S$ and $M\setminus S$, and so $h$ is a homeomorphism.
\end{proof}

\subsection{Loops and Cantor sets in disks}
By a \emph{Cantor set} we mean a totally disconnected compact metrizable space with no isolated points.  This subsection is about Cantor sets in the plane or in $S^2$, and doesn't refer directly to our group-theoretic setup.  
We will see later that $\Lambda(K_W)$ is a Cantor set, and use the following lemmas to control how $\Lambda(K_W)$ sits in $\partial X$. 

\begin{lemma}\label{lem:disks_covering_Cantor}
Let ${\bf C}$ be a Cantor set contained in an open disk $D$.  Suppose that
$\{ U_i \}_{i\in I}$ is a finite collection of disjoint clopen subsets of ${\bf C}$ whose union is ${\bf C}$.

Then there exists a finite collection of closed subdisks $\{ D_i\}_{i \in I}$, so that for all distinct $j,k\in I$ we have $D_j \cap D_k = \emptyset$ and for all $j$ we have ${\bf C} \cap \mathring{D}_j = U_j$.
\end{lemma}

\begin{proof}
 Let ${\bf C}_{std}$ be the standard middle-third Cantor set in the plane. It is known that any homeomorphism $f\co {\bf C}\to {\bf C}_{std}$ extends to a homeorphism $\overline f\co D\to \mathbb R^2$ (see \cite[Chapter 13]{Moise77}). It is then easy to construct a homeomorphism $f$ so that the collection of clopen sets $\{f(U_i)\}$ admits a family of disks in $\mathbb R^2$ as in the statement, which can be then pulled back to $D$ using $\overline f$.
 \end{proof}
 
 \begin{lemma}\label{lem:loop_unique}
  Let ${\bf C}$ be a Cantor set contained in $S^2$.  Suppose that
$U$ is a clopen subset of ${\bf C}$. If $D_1, D_2$ are closed disks in $D$ with $\mathring D_i\cap {\bf C}= U$, then $\partial D_1$ is homologous to $\partial D_2$ in $H_1(S^2\setminus {\bf C})$.
 \end{lemma}

 \begin{proof}
   Let $h\co S^2\to [0,1]$ be a smooth function which is zero exactly on $U$.  Then for a sufficiently small regular value $\epsilon$, the set $h^{-1}[0,\epsilon]$ is contained in $D_1\cap D_2$.  The $1$--manifold $h^{-1}(\epsilon)$ is clearly homologous to both $\partial D_1$ and $\partial D_2$.
\end{proof}

\subsection{The particular Cantor set}
In this subsection we return to the situation set up in Subsection \ref{ss:setup} and verify that the limit set $\Lambda(K_W)$ in $\partial X$ is a Cantor set when the rank $k\geq 2$.  We also describe a nice basis for the topology on $\Lambda(K_W)$.  

Recall that the group $K_W$ is freely generated by parabolic elements $a_1,\ldots,a_k$.  In particular it is a free group whose Gromov boundary $\partial K_W$ can be identified with the set of all infinite freely reduced words in $a^{\pm 1}_1, \ldots , a^{\pm 1}_k$.  The collection of quasiconvex subgroups $\mc{A} = \{\langle a_1 \rangle,\ldots \langle a_k\rangle\}$ is malnormal in the free group $K_W$, so the pair $(K_W,\mc{A})$ is relatively hyperbolic.  Its Bowditch boundary $\partial(K_W,\mc{A})$ is the quotient of $\partial K_W$ obtained by identifying the pairs $\{ wa_i^{\infty}, wa_i^{-\infty} \}$ for each $i$ and each freely reduced $w$.  We can choose $w$ not to end with $a_i$ or $a_i^{-1}$ in such a description.  (See \cite[Theorem 1.1]{Tran:comparison_boundaries} for the description of the boundary of a relatively hyperbolic pair $(H,\mc{Q})$ where $H$ is hyperbolic, cf. \cite{Ger:Floyd,GerPot:Floyd,MOY:blowingupanddown,ManningBoundary}.)

\begin{lemma}\label{lem:LambdaKW=partialPFk}
 There is an equivariant homeomorphism $\partial(K_W,\mc{A})\to\Lambda(K_W)$.
\end{lemma}

\begin{proof}
 The spiderweb axioms \eqref{item:qc} and \eqref{item:relcocompact} from Definition \ref{def:spiderweb} imply that $K_W$ is relatively quasiconvex in $(G,\mc{P})$, using \cite[Definition 6.5 (QC-3)]{Hru-relqconv}.
 In particular, the limit set $\Lambda(K_W)$ is equivariantly homeomorphic to the relative boundary of $K_W$ endowed with the peripheral structure induced by $G$ (see e.g. the alternative definition of relative quasiconvexity \cite[Definition 6.2 (QC-1)]{Hru-relqconv}), which corresponds to the peripheral structure on $K_W$ used to define $\partial(K_W,\mc{A})$.
\end{proof}

\begin{defn}\label{wij}
Given a natural number $k$, an index $i \in \{ 1, \ldots , k\}$, a natural number $j$ and a word $w$ which does not end with $a_i$ or $a_i^{-1}$, let $\mc{B}(w,a_i,j)$ be the image in $\partial(K_W,\mc{A})$ of 
the set of all infinite freely reduced words beginning with $wa_i^j$ or $wa_i^{-j}$.
\end{defn}

\begin{lemma}\label{lem:wij_clopen}
The sets $\mc{B}(w,a_i,j)$ are clopen in $\partial(K_W,\mc{A})$.
\end{lemma}

\begin{proof}
 The subset $A$ of $\partial K_W$ of all infinite (freely reduced) words which start with $wa_i^j$ or $wa_i^{-j}$ is closed, whence compact, and hence so is its image $\mc{B}(w,a_i,j)$ in $\partial(K_W,\mc{A})$. Moreover, $A^c$ is also closed and hence so is its image $B$ in $\partial(K_W,\mc{A})$. It is easily seen that if the infinite freely reduced word $w'$ does not start with either $wa_i^j$ or $wa_i^{-j}$ then no word identified to $w'$ in $\partial(K_W,\mc{A})$ starts with either $wa_i^j$ or $wa_i^{-j}$, hence $B=\mc{B}(w,a_i,j)^c$, and $\mc{B}(w,a_i,j)$ is clopen.
\end{proof}

\begin{cor}\label{cor:isCantor}
 If $k\geq 2$ then $\partial(K_W,\mc{A})$ is a Cantor set.
\end{cor}

\begin{proof}
The fact that $\partial(K_W,\mc{A})$ is totally disconnected follows from the fact that the $\mc{B}(w,a_i,j)$ are clopen. It is also easy to see that it does not have isolated points. Finally, $\partial(K_W,\mc{A})$ is compact and metrizable since it is (homeomorphic to) the boundary of a proper hyperbolic space.
\end{proof}

The following two lemmas follow directly from the definitions.

\begin{lemma}\label{lem:k_orbits}
For any $i$ and any $w$ that does not end with $a_i$ or $a_i^{-1}$ we have
\[	w . \mc{B}(1,a_i,1) = \mc{B}(w,a_i,1)	.	\]
\end{lemma}

\begin{lemma}\label{lem:difference_bwij}
For any $w$ which does not end with $a_i$ or $a_i^{-1}$ and any $j > 1$, the set $\mc{B}(w,a_i,j)$ is equal to $\mc{B}(w,a_i,1)$ minus the union
\[\left( \bigcup_{j' < j, i' \ne i} \mc{B}	(wa_i^{j-j'},a_{i'},1) \right) \bigcup \left( \bigcup_{j' < j, i' \ne i} \mc{B}(wa_i^{-(j-j')}, a_{i'},1) \right).\]
\end{lemma}

\begin{lemma}\label{lem:clopen_decomposition}
  Let $U\subseteq \partial(K_W,\mc{A})$ be clopen.  Then $U$ is a finite disjoint union of sets $\{\mc{B}(w_s,a_{i_s},j_s)\}_{s\in J}$.
\end{lemma}
\begin{proof}
  Let $U'$ be the preimage of $U$ in $\partial K_W$, and note that $U'$ is clopen.  
  \begin{claim*}
    There is an $n>0$ satisfying:  Whenever $v$ is an infinite freely reduced word which coincides with some $w\in U'$ on an initial subword of length $n$, then $v\in U'$
  \end{claim*}
  \begin{proof}
    Suppose not.  Then there is a sequence of pairs $\{(v_i,w_i)\}_{i\in \bN}$ so that $v_i$ coincides with $w_i$ on an initial subword of length $i$, $w_i\in U'$, but $v_i\notin U'$.  The common prefixes $u_i$ subconverge to an infinite word which is in the closure both of $U'$ and of its complement, contradicting the fact that $U'$ is clopen.
  \end{proof}
  Now let $J$ be the set of prefixes of words in $U'$ of length $n$ which end with a positive power of one of the generators $a_i$. 
Any $s\in J$ can be written uniquely as a freely reduced word $w_s a_{i_s}^{j_s}$.  Then $\left\{ \mc{B}(w_{s},a_{i_s},j_s) \right\}_{s \in J}$ satisfies the required properties. 
\end{proof}

\subsection{Proof of Proposition \ref{Xiarespheres}}
We use the notation set up at the beginning of the section.  Note that $\partial X\setminus\Lambda(K_W)$ has a $K_W$--action, which makes its homology into a $K_W$--module.  Recall that $K_W$ is free of rank $k$.
\begin{lemma}\label{lem:KW-module}
 As a $K_W$--module, $H_1(\partial X\setminus \Lambda(K_W);\bZ/2)$ has rank at most $\max\{k-1,0\}$.
\end{lemma}

\begin{proof}
When we refer to `homology' in this proof we always mean homology with $\mathbb Z/2$ coefficients.

Notice that the cases $k=0,1$ are easy, so we can assume $k\geq 2$. Let $Z=\partial X\setminus \Lambda(K_W)$. We identify $\partial(K_W,\mc{A})$ with $\Lambda(K_W)$ (which we can do in view of Lemma \ref{lem:LambdaKW=partialPFk}). Recall that $\Lambda(K_W)$ is a Cantor set by Corollary \ref{cor:isCantor}.
 For each set $\mc{B}(w,a_i,j)$ as in Definition \ref{wij}, let $l_{w,a_i,j}$ be a loop in $\partial X$ bounding a disk that intersects $\Lambda(K_W)$ in $\mc{B}(w,a_i,j)$ (which exists by Lemma \ref{lem:disks_covering_Cantor} in view of the fact that $\mc{B}(w,a_i,j)$ is clopen, see Lemma \ref{lem:wij_clopen}). The element of $H_1(Z)$ represented by $l_{w,a_i,j}$ depends only on $w,a_i$, and $j$, by Lemma \ref{lem:loop_unique}. 
As a first step in the proof, let us show that such loops generate $H_1(Z)$. It suffices to prove that any simple loop $l$ in $Z$ is, homologically, a sum of loops $l_{w,a_i,j}$. This is because it suffices to consider smooth self-transverse loops, and each of those is homologically a sum of simple loops. Let $D$ be one of the disks in $\partial X$ bounded by $l$.  It follows from Lemma \ref{lem:clopen_decomposition} that $D\cap \Lambda(K_W)$ is a disjoint union of sets of the form $\mc{B}(w,a_i,j)$, which in turn implies, in view of Lemma \ref{lem:disks_covering_Cantor}, that homologically $l$ is a sum of loops $l_{w,a_i,j}$.
 
 We now prove that each loop $l_{w,a_i,j}$ is homologically a sum of loops of the form $l_{w',a_{i'},1}$. Consider some $l_{w,a_i,j}$. By Lemma \ref{lem:difference_bwij}, one of the disks $D$ bounded by $l_{w,a_i,1}$ has the property that $D\cap \Lambda(K_W)$ is a disjoint union of $\mc{B}(w,a_i,j)$ and sets $\mc{B}(w',a_{i'},1)$. By Lemma \ref{lem:disks_covering_Cantor}, $l_{w,a_i,j}$ is homologically a sum of loops $l_{w',a_{i'},1}$, as required.
 
 Homologically, each of these loops $l_{w',a_{i'},1}$ is in the $K_W$--orbit of the loop $l_{1,a_{i'},1}$ by Lemma \ref{lem:k_orbits}.  In particular, the $k$ loops $\{l_{1,a_i,1}\}$ generate $H_1(Z)$ as a $K_W$--module.
If $k\geq 2$, these loops can be chosen to encircle disjoint discs whose union contains $\Lambda(K_W)$, so we have $\sum l(1,a_i,1) = 0$ in $H_1(S^2\setminus \Lambda(K_W))$.  Any one of the generators can be written in terms of the others, so the rank is at most $k-1$.
\end{proof}

\begin{proof}[Proof of Proposition \ref{Xiarespheres}]
For sufficiently long fillings and any spiderweb $W$ with suitable parameter, by Theorem \ref{thm:trunc} there is a normal covering map $\partial X\setminus \Lambda(K_W)\to \partial T_W\setminus \mathcal F$ with deck group $K_W$, where $\mathcal F$ is the union of all limit sets of horosphere quotients $\Lambda(\Sigma_c /K_c)$ for $c\in \calC(W)$, which in our case is a finite set with $2k$ elements if $k$ is the rank of the free group $K_W$ (we assume $k\geq 1$). In particular, $S=\partial T_W\setminus \mathcal F$ is a $2$--manifold, and since $\partial T_W \setminus \mathcal F$ is open and dense in $\partial T_W$ by Theorem \ref{thm:trunc}.\eqref{trunc_F codense} and $\partial T_W$ is compact, in view of Lemma \ref{lem:reduction_to_homology} we are left to show that the dimension of $H_1(S,\mathbb Z/2)$ is at most $2k-1$. From the short exact sequence $$1\to \pi_1(\partial X\setminus \Lambda(K_W)) \to \pi_1(S) \to K_W\to 1,$$ we see that this dimension is the sum of $k$ and the rank of $H_1(\partial X\setminus \Lambda(K_W))$ as a $K_W$--module, so we are done by Lemma \ref{lem:KW-module}.
\end{proof}

\section{Ruling out the Sierpinski carpet} \label{ss:not Sierpinski}
In the next section, we will prove Theorem \ref{thm:boundary sphere} by first showing that the boundary $\partial \overline{G}$ is planar, using Lemma \ref{ivanovlemma} and a criterion of Claytor \cite{Claytor34}.  A result of Kapovich and Kleiner \cite[Theorem 4]{KK} (along with \cite[Theorem 1.2]{GM-splittings}) then implies that $\partial \overline{G}$ is either $S^2$ or a Sierpinski carpet.  In this section, we rule out the possibility that it is a Sierpinski carpet.

\begin{defn}
  Let $M$ be a metric space, let $f\co S^1\to M$ be continuous, and let $\epsilon>0$.  We say that $f$ \emph{has an $\epsilon$--filling} if there is a triangulation of the unit disk $D^2$ and a (not necessarily continuous) extension of $f$ to $\bar{f}\co D^2\to M$, so that each simplex of the triangulation is mapped by $\bar{f}$ to a set of diameter at most $\epsilon$.  
\end{defn}

\begin{defn}
  Say a metric space $M$ is \emph{weakly simply connected} if, for every continuous $f\co S^1\to M$ and every $\epsilon>0$, $f$ has an $\epsilon$--filling.
\end{defn}

The following two lemmas are easy.
\begin{lemma}\label{lem:weakly_1_conn_topological}
  For a compact metrizable space $M$, being weakly simply connected is independent of the metric.
\end{lemma}

\begin{lemma}
  Any simply connected metric space is weakly simply connected.
\end{lemma}

\begin{lemma}\label{lem:Sierpinski_not_1_connected}
  The Sierpinski carpet (with any metric) is not weakly simply connected.
\end{lemma}

\begin{proof}
 By Lemma \ref{lem:weakly_1_conn_topological}, we just need to check this for a Sierpinski carpet $S$ embedded in the $2$--sphere and endowed with the induced metric. Suppose by contradiction that $S$ is weakly simply connected.  Then it is easily seen that, given any $\epsilon>0$ and any loop $\ell$ contained in $S$, we can find a continuous map $f\co D^2\to S^2$ whose image is contained in the $\epsilon$--neighborhood of $S$ and so that $f(\partial D^2)$ is $\ell$. However, the image of a continuous map $f\co D^2\to S^2$ so that $f(\partial D^2)$ is a simple loop $\ell$ contains one of the two connected components of $S^2\setminus \ell$. For $\ell$ a peripheral circle of $S$ and $\epsilon>0$ small enough, neither connected component of $S^2\setminus \ell$ is contained in the $\epsilon$--neighborhood of $S$, a contradiction.
\end{proof}

\begin{theorem}\label{thm:weakly_1_connected_limit}
  Suppose the compact metric space $Z$ is a weak Gromov--Hausdorff limit of $\{Z_n\}_{n\in \bN}$, and suppose that there is some $L\geq 1$ so that all the spaces $Z_n$ and $Z$ are $L$--linearly connected.  If the spaces $Z_n$ are weakly simply connected, then so is $Z$.
\end{theorem}
\begin{proof}
  Assume that $\{Z_n\}$ and $Z$ are as in the hypothesis of the theorem.  Then there are a $K\geq 1$ and some $(K,\epsilon_n)$--quasi-isometries $\psi_n\co Z_n\to Z$ and $\phi_n\co Z\to Z_n$ which are $\epsilon_n$--quasi-inverses of one another and for which $\lim\limits_{n\to\infty}\epsilon_n=0$.

  Let $f\co S^1\to Z$, and let $\epsilon>0$.  Choose some $\epsilon'<\frac{\epsilon}{10LK^2}$, and fix some $n$ so that $\epsilon_n<\epsilon'$.

We want to build an $\epsilon$--filling (in $Z$) from some $\epsilon'$--filling in $Z_n$.  We'll first approximate $f$ by a discrete map, push it to $Z_n$, fill, and then push the filling back to $Z$.

  Let $\Theta\subset S^1$ be a discrete set with at least three points.  We say that $\theta,\theta'\in\Theta$ are \emph{consecutive} if they bound a (necessarily unique) interval $I =: [\theta,\theta']$ in $S^1$ whose interior is disjoint from $\Theta$.  By refining $\Theta$ we can ensure the following:
  \begin{itemize}
  \item If $\theta, \theta'$ are consecutive (on $S^1$), and $x\in [\theta,\theta']$, then $d_Z(f(\theta),f(x))< \epsilon'$.
  \end{itemize}
  In particular, $f(S^1)$ lies in an $\epsilon'$--neighborhood of $f(\Theta)$, and $\phi_n f(S^1)$ lies in a $(K+1)\epsilon'$--neighborhood of $\phi_n f(\Theta)$.  More to the point, if $\theta, \theta'$ are consecutive elements of $\Theta$, then $d_{Z_n}(\phi_nf(\theta),\phi_nf(\theta'))<(K+1)\epsilon'$, and so there is an arc in $Z_n$ of diameter at most $L(K+1)\epsilon'$ joining $\phi_nf(\theta)$ to $\phi_nf(\theta')$.  
Concatenating these arcs, we obtain a continuous $f'\co S^1\to Z_n$.  

  \begin{claim*}
    For any $x\in S^1$, we have $d_Z(\psi_n f'(x),f(x))<\frac{\epsilon}{2}$.
  \end{claim*}

  Assuming the claim, we argue as follows.
  The map $f'$ has an $\epsilon'$--filling $F'\co D^2\to Z_n$.  Define a filling $F\co D^2\to Z$ of $f$ by:
\[ F(x) =
\begin{cases}
  f(x) & x\in S^1\\
  \psi_n F'(x) & x\in D^2\setminus S^1.
\end{cases}
\]
  Since $\epsilon_n<\epsilon'$, and since the difference between $\psi_n f'$ and $f$ is at most $\frac{\epsilon}{2}$ on $S^1$,  $F$ is a $\left((K+1)\epsilon' + \frac{\epsilon}{2}\right)$--filling of $f$.  But $\epsilon' < \frac{\epsilon}{2(K+1)}$, so $F$ is an $\epsilon$--filling.  Modulo the claim, the theorem is proved.

\begin{proof}[Proof of Claim]
  Note first that if $x\in \Theta$, then $f'(x) = \phi_n f(x)$, so $$d_Z(\psi_nf'(x),f(x))\leq \epsilon_n < \epsilon'<\frac{\epsilon}{2}.$$  

  Suppose now that $x\notin \Theta$.  There are consecutive $\theta, \theta'\in \Theta$ so that $x$ lies in $[\theta,\theta']$, and so that $d_Z(f(x),f(\theta))<\epsilon'$.  It follows that $d_{Z_n}(\phi_nf(x),\phi_nf(\theta))<(K+1)\epsilon'$.  

  From the construction of $f'$ we have $d_{Z_n}(f'(x), f'(\theta))\leq L(K+1)\epsilon'$.  Pushing back to $Z$ we get
  \begin{align*}
    d_Z(\psi_nf'(x), f(x)) & \leq d_Z(\psi_n f'(x),\psi_nf'(\theta))+d_Z(\psi_nf'(\theta),f(\theta))+d_Z(f(\theta),f(x))\\
& < K(L(K+1)\epsilon'+\epsilon') + \epsilon' + \epsilon'\\
& \leq 5LK^2\epsilon' < \frac{\epsilon}{2},
  \end{align*}
 and the claim is proved.
\end{proof}
With the claim proved, the proof of Theorem \ref{thm:weakly_1_connected_limit} is complete.  
\end{proof}

\section{Proof of Theorem \ref{thm:boundary sphere}} \label{s:proof of boundary sphere}
In this section we will prove Theorem \ref{thm:boundary sphere}, which we restate for the convenience of the reader.
\boundarysphere*

\begin{proof}
 First of all, notice that the $P_i$ have rank $2$, because they act properly discontinuously and cocompactly on the complement of the corresponding parabolic point in $\partial(G,\mc{P})$, which is homeomorphic to $\mathbb R^2$. Suppose that the filling is long enough that Theorem \ref{thm:trunc}, Theorem \ref{thm:existsvisual}, Theorem \ref{thm:linconn} and Proposition \ref{Xiarespheres}, as well as \cite[Theorem 1.2]{GM-splittings}, all apply. In particular, $\overline G$ is hyperbolic.  By Theorem \ref{exhaustion} there exists a sequence of spiderwebs $W_i$ and (visual) metrics $\rho_{W_i}$ on $\partial T_{W_i}$ so that
 \begin{itemize}
  \item each $\partial T_{W_i}$ is homeomorphic to a $2$--sphere (see Proposition \ref{Xiarespheres}),
  \item there exists $L$ so that each $\rho_{W_i}$ is $L$--linearly connected (see Theorem \ref{thm:linconn}),
  \item $(\partial T_{W_i}, \rho_{W_i})$ weakly Gromov--Hausdorff converges to a (visual) metric on $\partial T_{\overline{G}}$ (see Theorem \ref{thm:existsvisual}), which is homeomorphic to $\partial \overline{G}$.
 \end{itemize}

 It follows from \cite[Theorem 1.2]{GM-splittings} that $\partial \overline G$ is a Peano continuum without local cut points. Moreover, by Lemma \ref{ivanovlemma}, $\partial \overline{G}$ does not contain an embedded topological copy of any non-planar graph and hence, by \cite{Claytor34}, $\partial \overline G$ is planar. Since, as mentioned, $\partial \overline G$ does not contain local cut points, it must be a sphere or a Sierpinski carpet, by \cite[Theorem 4]{KK}. The latter is ruled out by Theorem \ref{thm:weakly_1_connected_limit} and Lemma \ref{lem:Sierpinski_not_1_connected}.
 \end{proof}
 
 \begin{rem}
  There is another possible variation of the argument above that does not use \cite[Theorem 1.2]{GM-splittings}. First of all, $\partial{\overline{G}}$ is connected because it is a weak Gromov--Hausdorff limit of connected spaces. Moreover, it does not have global cut points because it is the connected boundary of a hyperbolic group \cite{Sw:no_cut_points}. Hence, \cite{Claytor34} applies and $\partial \overline G$ is planar. At this point we would have to adapt the arguments in Section \ref{ss:not Sierpinski} to deal with a planar continuum properly contained in $S^2$.
 \end{rem}

\section{Proof of Corollary \ref{cor:relative Cannon}}\label{sec:corollary}

In this section we prove Corollary \ref{cor:relative Cannon}, that the Cannon Conjecture implies the Relative Cannon Conjecture.  To that end, suppose that the Cannon Conjecture is true and suppose that $(G,\mathcal{P})$ is a relatively hyperbolic pair, where $\mathcal{P} = \{ P_1 , \ldots , P_n \}$ is a finite collection of free abelian subgroups of rank $2$ and suppose further that the Bowditch boundary of $(G,\mathcal{P})$ is homeomorphic to $S^2$.

\begin{defn}
 A sequence $\{ \eta_i \co  G \to G_i \}$ of homomorphisms is {\em stably faithful} if $\eta_i$ is faithful on the ball of radius $i$ about $1$ in $G$.
\end{defn}

By choosing fillings kernels $K_{i,j} \unlhd P_j$ where $P_j / K_{i,j}$ is infinite cyclic, but the slope is growing, we obtain a stably faithful sequence of fillings $G \to G_i$.  The groups $G_i$ are all hyperbolic relative to a finite collection of infinite cyclic groups, and so are in fact hyperbolic.
By Theorem \ref{thm:boundary sphere}, we may assume that each $G_i$ is a hyperbolic group with $2$--sphere boundary.
By the Cannon Conjecture, each $G_i$ admits a discrete faithful representation into $\mathrm{Isom}(\mathbb H^3)$.  Pre-composing with the quotient maps $G \to G_i$, we get a stably faithful sequence of representations of $\rho_i \co  G \to \mathrm{Isom}(\mathbb H^3)$. 

There are two cases (after passing to a subsequence of $\{ \rho_i \}$): 
\begin{enumerate}
\item Up to conjugacy in $\mathrm{Isom}(\bH^3)$, the representations $\rho_i$ converge to a representation $\rho_\infty \co  G \to \mathrm{Isom}(\bH^3)$. 
\item The representations $\rho_i$ diverge in the $\mathrm{Isom}(\mathbb H^3)$--character variety of $G$.
\end{enumerate}
Suppose that the first case holds. The $\rho_i$ are stably faithful with discrete image.    

\begin{claim}\label{claim:find g,h}
Suppose that the limiting representation $\rho_\infty$ is not discrete and faithful. Then for every $\epsilon > 0$ there exists a pair $\{ g,h \}$ of noncommuting elements of $G$ and a point  $x \in \bH^3$ so that $\rho_\infty(g)$ and $\rho_\infty(h)$ move $x$ distance less than $\epsilon$.  
\end{claim}
\begin{proof}[Proof of Claim \ref{claim:find g,h}]
If the limiting representation is not faithful then there are certainly non-commuting elements $g$ and $h$ in the kernel, which will suffice.

On the other hand, suppose that the limiting representation is faithful, but indiscrete. Then there are elements $g_j$ of $G$ which are not in the kernel of $\rho_\infty$ but so that $\rho_\infty(g_j)$ tends towards the identity.  Unless the $g_j$ eventually commute with each other, taking two elements far enough along the sequence will suffice for $g$ and $h$.  Thus we may suppose that all of the $g_j$ commute with each other, which implies that they all preserve some point at infinity, some geodesic or some point in the interior of $\mathbb H^3$.  Since $\rho_\infty$ is faithful, it is not elementary, so there is some $\gamma \in G$ so that $\rho_\infty(\gamma)$ does not preserve this set.  We can take $g_j$ and $g_j^\gamma$ for our $g$ and $h$ (for large enough $j$).
\end{proof}

The condition from Claim \ref{claim:find g,h} is an open condition, so for all but finitely many $i$ the elements $\rho_i(g)$ and $\rho_i(h)$ move some point in $\bH^3$ a distance smaller than $\epsilon$.  Since for large $i$ the elements $\rho_i(g)$ and $\rho_i(h)$ are nontrivial, for small enough $\epsilon$ this violates Margulis' Lemma and shows that the image $\rho_i$ is not discrete, which is a contradiction.  This implies that in the first case the representation $\rho_\infty$ is discrete and faithful,
so $G$ is Kleinian, as required.

It remains to rule out the second case, that the sequence $\rho_i$ diverges.
If it does, then by choosing basepoints appropriately and rescaling these representations limit to an action of $G$ on an $\mathbb R$--tree $T$ with no global fixed point.

A standard argument (see, for example, the proof of \cite[Theorem 6.1]{GM-splittings}) shows that arc stabilizers for the $G$--action on $T$ are metabelian and hence small.  Since small subgroups of $G$ are finitely generated, this means that arc stabilizers satisfy the ascending chain condition.  Therefore, by \cite[Proposition 3.2.(2)]{BF:stable} the action of $G$ on $T$ is stable.  It now follows from \cite[Theorem 9.5]{BF:stable} that $G$ splits over a small-by-abelian (and hence small) subgroup. 

However, all small subgroups of $G$ are elementary, but we know that $G$ admits no elementary splittings (by work of Bowditch \cite{Bowditch:boundary_geom_finite,Bowditch:connected_boundary,Bowditch:peripheral_splittings,bowditch12}, see \cite[Corollary 7.9]{GM-splittings}).  This implies that $G$ is Kleinian, as required.

\appendix

\section{$\delta$--hyperbolic technicalities} \label{app:technical}

In this appendix, we collect some technical results which are needed for the proofs in this paper, but which are proved using standard arguments and are probably well known to experts.  In Subsection \ref{ss:QI}, we prove that the various different approaches to building ``cusped'' spaces all result in quasi-isometric spaces.  The key application of this in our paper is Corollary \ref{cor:boundary_lin_conn}.  In Subsection \ref{ss:delta} we collect some results about $\delta$--hyperbolic geometry, which are all well known.  The sole innovation
is to record constants.

\subsection{Quasi-isometric horoballs and cusped spaces} \label{ss:QI}
  The combinatorial cusped space in Definition \ref{def:cc} is one of several ways to build a ``cusped space'' whose hyperbolicity detects the relative hyperbolicity of the pair $(G,\mc{P})$.  Each method begins with a Cayley graph for $G$, and attaches some kind of ``horoball'' to each left coset of an element of $\mc{P}$.  In \cite{bowditch12}, Bowditch glues `hyperbolic spikes' (the subset $[0,1] \times [1,\infty)$ in the upper half-space model of $\bH^2$) to each edge in the included cosets of Cayley graphs of the peripheral subgroups, with the lines $\{ 0 \} \times [1,\infty)$ and $\{ 1 \} \times [1,\infty)$ glued according to when edges share vertices.  The resulting cusped space is also Gromov hyperbolic if and only if $(G,\mc{P})$ is relatively hyperbolic, by \cite{bowditch12}.
 Another way of building horoballs on graphs is provided by Cannon and Cooper \cite{CannonCooper}. In this case, the horospheres are copies of $\Gamma$, but scaled at depth $d$ by $\lambda^d$ for some $\lambda \in (0,1)$ (Cannon and Cooper chose $\lambda = e^{-1}$, which is the most natural choice when comparing to the metric in $\bH^n$).

In order to be able to translate results proved with different cusped spaces to the other settings, it is convenient to notice that it is not only the case that all of these constructions provide characterizations of relative hyperbolicity, but that they provide $G$--equivariantly quasi-isometric cusped spaces, a fact well-known to experts.  This is what we explain in this subsection. Throughout the paper, the fact that we can use results from the literature proved using different models is justified by the results in this section.

We consider three types of horoballs, each depending on some scaling factor $\lambda$.

\begin{defn} [Combinatorial Horoball]
Let $\Gamma$ be a graph and $\lambda > 1$ a constant.  The {\em combinatorial horoball based on $\Gamma$ with scaling factor $\lambda$} is the graph as defined in Definition \ref{d:comb horo} except that horizontal edges are added between $(v,k)$ and $(w,k)$ when $0 < d_\Gamma(v,w) \le \lambda^k$.
We denote this space by $CH(\Gamma,\lambda)$.
\end{defn}
Note that Groves and Manning used $\lambda = 2$, as in Definition \ref{d:comb horo}.

\begin{defn} [Cannon--Cooper Horoball]
Let $\Gamma$ be a metric graph and let $\lambda > 1$.  We form the \emph{CC-horoball based on $\Gamma$ with scaling factor $\lambda$} to be a metric graph $\calH(\Gamma)$ whose vertex set is $\Gamma^{(0)}\times \bZ_{\geq 0}$, and with two types of edges:
\begin{enumerate}
\item A \emph{vertical} edge of length $1$ from $(v,n)$ to $(v,n+1)$ for any $v\in \Gamma^{(0)}$ and any $n\geq 0$;
\item If $\epsilon$ is an edge of length $l$ in $\Gamma$ joining $v$ to $w$, and $n\geq 0$, there is a \emph{horizontal} edge of length $\lambda^{-n}l$ joining $(v,n)$ to $(w,n)$.  
\end{enumerate}
We denote this space by $CCH(\Gamma,\lambda)$.
\end{defn}
Note that Cannon and Cooper used $\lambda = e$.

To define the Bowditch horoball, it is convenient to use the notion of a warped product of length spaces.
\begin{defn}[Warped product of length spaces] \cite{Chen99}
  Let $(B,d_B)$ (the \emph{base}) and $(F,d_F)$ (the \emph{fiber}) be two length spaces, and let $f\co B\to [0,\infty)$ be a continuous function (the \emph{warping function}).  Let $t\mapsto (\beta(t),\gamma(t))$ define a path $\sigma\co [0,1]\to B\times F$.
Define a length by first considering, for each partition $\tau = \{0=t_0< t_1<\cdots <t_{n(\tau)} = 1\}$, the $\tau$--length:
\[ l_\tau(\sigma) = \sum_{i=1}^{n(\tau)} \left(d_B^2(\beta(t_i),\beta(t_{i-1})) + f^2(\beta(t_i))d_F^2(\gamma(t_i),\gamma(t_{i-1})) \right)^{\frac{1}{2}}.\]
Any two partitions have an upper bound (their union), and there is a well-defined limit over partitions $l(\sigma)\in [0,\infty]$.  In fact it is not hard to see there is always a finite length path, and we get a length pseudometric $d_f$ on $B\times F$.  If $f$ has no zeros, this is a metric, and $B\times F$ with this metric is written $B\times_f F$.
\end{defn}

\begin{defn} [Bowditch Horoball] \label{d:Bow HB}
Suppose that $\Gamma$ is a metric graph and that $\lambda > 1$ is a constant.  The {\em Bowditch Horoball based on $\Gamma$ with scaling factor $\lambda$} is the warped product
\[	[0,\infty) \times_{\lambda^{-t}} \Gamma	.	\]
We denote this space by $BH(\Gamma,\lambda)$.
\end{defn}
Note that Bowditch used $\lambda = e$.

The following result is elementary and probably well known to many experts.  The proof is very similar to part of the proof of \cite[Theorem, $\S4.2$]{CannonCooper} (see also \cite[Proposition 3.2]{Durham-thesis} for more details).  Neither Cannon and Cooper nor Durham deal with a general graph, but this is irrelevant for the proofs.  We leave the details as an exercise for the reader.
\begin{prop} \label{p:Horoball QI}
Suppose that $\lambda_1, \lambda_2 > 1$ are constants and that $\Gamma$ is a graph.  The map
$\Gamma^{(0)} \times \bZ_{\geq 0} \to [0,\infty) \times_{\lambda_2^{-t}} \Gamma$ defined by
\[	(v,n) \mapsto \left(\frac{\ln(\lambda_1)}{\ln(\lambda_2)}n,v\right)	,\]
extends naturally to quasi-isometries
\[	CH(\Gamma,\lambda_1) \to BH(\Gamma,\lambda_2) 	,	\]
and
\[	 CCH(\Gamma,\lambda_1) \to BH(\Gamma,\lambda_2)	,	\]
by mapping edges in the left-hand spaces to geodesics in $BH(\Gamma,\lambda_2)$.
\end{prop}

\begin{rem}
 If some care is not taken then different kinds of horoballs may not be quasi-isometric.  For example, if the warping function $f$ for $[0,\infty) \times_f \Gamma$ is taken to be doubly-exponential, the resulting horoball will still be Gromov hyperbolic, but will not be quasi-isometric to $BH(\Gamma,\lambda)$.
\end{rem}

From each kind of horoball, there is then an associated {\em cusped space}, obtained from the Cayley graph of $G$ by gluing the horoballs based on the Cayley graphs of $P \in \mc{P}$ onto the cosets in the same way is described in Definition \ref{def:cc} above.

The following is a straightforward application of Proposition \ref{p:Horoball QI}.

\begin{cor}
Suppose that $G$ is a group, and that $\mc{P}$ is finite collection of finitely generated subgroups of $G$.  The three kinds of cusped spaces obtained by gluing either combinatorial, Cannon--Cooper, or Bowditch horoballs to the Cayley graph of $G$ are all quasi-isometric via maps extending the identity map on the Cayley graph of $G$.

In particular, any one of them is Gromov hyperbolic if and only if any of the others is.
\end{cor}

 Quasi-isometric proper Gromov hyperbolic spaces have quasi-symmetric boundaries (see for example \cite[Theorem 5.2.17]{BuyaloSchroeder}).

\begin{cor}
Suppose that $(G,\mc{P})$ is relatively hyperbolic.  Let $X_{CH}, X_{CCH}$ and $X_{BH}$ be the three cusped spaces for $(G,\mc{P})$ associated to the three kinds of horoballs.  Equip the (Bowditch) boundaries $\partial X_{CH}, \partial X_{CCH}$ and $\partial X_{BH}$ with visual metrics based at $1$.  These boundaries are quasi-symmetric.
\end{cor}
 By Lemma \ref{lem:LCQS}, either all these boundaries are linearly connected, or none of them are.
The following is now an immediate consequence of \cite[Proposition 4.10]{MS1111.2499} and results of Bowditch (\cite{Bowditch:boundary_geom_finite,Bowditch:connected_boundary,Bowditch:peripheral_splittings}, see \cite[Theorem 7.3]{GM-splittings}, for example).
\begin{cor}\label{cor:boundary_lin_conn}
Suppose that $(G,\mc{P})$ is relatively hyperbolic, that $\mc{P}$ consists of finitely presented groups with no infinite torsion subgroups, and that $(G,\mc{P})$ has no nontrivial peripheral splittings.  Then the boundary of the cusped space of $(G,\mc{P})$ (with respect to any type of horoball) is linearly connected.
\end{cor}

\subsection{$\delta$--hyperbolic geometry} \label{ss:delta}

All lemmas in this section are well-known facts about $\delta$--hyperbolic spaces. We include proofs for completeness and to explicitly keep track of how the constants appearing in the construction of spiderwebs depend on $\delta$.

\begin{lemma}\label{qconvhull}
 Let $Y$ be a $\delta$--hyperbolic space and let $A\subseteq Y$ be any set. Then the union $Z$ of all geodesics connecting pairs of points in $A$ is $2\delta$--quasiconvex. 
\end{lemma}

\begin{proof}
 For $i=1,2$, let $z_i\in [x_i,y_i]$ for some $x_i,y_i\in A$, and pick any $z\in[z_1,z_2]$. To prove $2\delta$--quasiconvexity we can just notice that $z$ is $2\delta$--close to either $[x_1,z_1]\subseteq [x_1,y_1]$, $[x_1,x_2]$ or $[x_2,z_2]\subseteq [x_2,y_2]$, and all such geodesics are contained in $Z$.
\end{proof}

\begin{defn}
For a set $W$, denote the power set of $W$ by $2^W$.  Suppose that $X$ is a metric space and $W \subset X$.  Let $\pi_W\co  X\to 2^W$ be closest point projection, so $\pi_W(x)$ is the set of all $x'\in W$ satisfying $d_X(x,x')=d_X(x,W)$.
\end{defn}

\begin{lemma}\label{quasiconvproj}
Let $Y$ be a $\delta$--hyperbolic space and let $W\subseteq Y$ be $Q$--quasiconvex.
Let $x,y\in X$, and let $x' \in \pi_W(x)$ and $y' \in \pi_W(y)$.
If $d_Y(x',y')>8\delta+2Q$ then $d_Y([x,y], W)\leq 2\delta+Q$ and $\max\{d_Y([x,y], x'),d_Y([x,y], y')\}\leq 6\delta+2Q$.
\end{lemma}

\begin{proof}
 Pick any point $p$ on a geodesic $[x',y']$ satisfying $d_Y(p,x'), d_Y(p,y')> 4\delta+Q$.  The quasi-convexity of $W$ ensures that $d(p,W)\le Q$. Slimness of geodesic quadrilaterals implies that $p$ is $2\delta$--close to some point $q$ lying on a geodesic $[x,x']$, $[y,y']$ or $[x,y]$.  We will rule out the first two possibilities, and deduce that $d([x,y],p)\leq 2\delta$ and so $d([x,y],W)\leq 2\delta + Q$.
 
 By symmetry, we can assume $q\in [x,x']$.  Since $d_Y(q,x')\geq d_Y(p,x')-2\delta>2\delta+Q$ and $d_Y(x,q)=d_Y(x,x') - d_Y(q,x')$, we have
 $$d_Y(x,W)\leq d_Y(x,q) + d_Y(q,W)< \Big( d_Y(x,x') - 2\delta-Q \Big) +2\delta+Q =d_Y(x,x'),$$
 contradicting the fact that $x' \in \pi_W(x)$.

  To obtain the second assertion, note that we could have chosen $p$ at distance $4\delta+Q+\epsilon$ from $x'$ or $y'$ for any sufficiently small $\epsilon>0$.  
\end{proof}

\begin{lemma}\label{geodbetweenqconv}
 Let $Y$ be a $\delta$--hyperbolic space and let $W_1,W_2$  be $Q$--quasiconvex subsets of $Y$. Also, let $\gamma$ be any geodesic from some point $p_1\in W_1$ to some point $p_2\in W_2$. Then any geodesic $\alpha$ from $W_1$ to $W_2$ is contained in $N_{Q+2\delta}(W_1)\cup N_{Q+2\delta}(W_2)\cup N_{2\delta}(\gamma)$.
\end{lemma}

\begin{proof}
 Let $q_i$ be the endpoints of $\alpha$, with $q_i\in W_i$. Then any point in $\alpha$ is $2\delta$--close to either $\gamma$ or to a geodesic $[p_i,q_i]$ for some $i$, and each such geodesic is contained in $N_Q(W_i)$, so we are done.
\end{proof}

\begin{defn}
We say that a path $\alpha$ in a geodesic space is {\em $C$--tight} if it is
\begin{enumerate}
 \item $(1,C)$--quasi-geodesic, and
 \item for any $s\leq t\leq u$ in the domain of $\alpha$ any geodesic from $\alpha(s)$ to $\alpha(u)$ passes $C$--close to $\alpha(t)$.
\end{enumerate}
The path is \emph{$\lambda$--locally $C$--tight} if $\alpha|_I$ is $C$--tight for every interval $I$ of length at most $\lambda$.
\end{defn}
We remark that in this paper we consider a $(\lambda,\epsilon)$--quasi-geodesic to be a unit speed path $\sigma$ so that $d(\sigma(s),\sigma(t))\geq \lambda^{-1}|s-t|-\epsilon$, for all $s,t$.
The results we prove are also true if instead one considers quasi-isometric embeddings of an interval, but all the quasi-geodesics we need are continuous unit speed maps.

\begin{lemma}\label{concatatproj}
 Let $Y$ be a $\delta$--hyperbolic space, let $W\subseteq Y$ be $Q$--quasiconvex and let $x\in X$. If $w\in \pi_W(x)$, then for any $w'\in W$ the concatenation of geodesics $[x,w],[w,w']$ is $(4\delta+2Q)$--tight.  
\end{lemma}
\begin{proof}
 We prove that the concatenation satisfies the second condition in the definition of $C$--tight, with $C=2\delta + Q$.  We then note that such a concatenation must be $(1,2C)$--quasigeodesic.

 We can restrict to considering geodesics connecting some $z_1\in [x,w]$ to some $z_2\in [w,w']$.
 Choose $[z_1,w]\subseteq [x,w]$ and $[w,z_2]\subseteq[w,w']$, and let $p\in [z_1,z_2]$, $p_1\in [z_1,w]$, $p_2\in [w,z_2]$ be the internal points of the triangle $[z_1,z_2]\cup [z_1,w]\cup[w,z_2]$.  

  Since $p_2$ lies on $[w,w']$, we have $d(p_2,W)\leq Q$, and so $d(p_1,W)\leq \delta+Q$.  Since $p_1$ lies on a shortest path from $x$ to $W$, we have $d(p_1,w)\leq \delta+Q$.  Moreover $d(p_2,w)=d(p_1,w)\leq \delta + Q$.  It follows that any point on $[z_1,w]\cup [z_2,w]$ lies within $2\delta+Q$ of some point on $[z_1,z_2]$, as required.
\end{proof}

The proof of the following result is a  minor variation of the one from \cite[III.H.1.13]{BH}

\begin{lemma}\label{localtight}
 Let $Y$ be a $\delta$--hyperbolic space and let $\alpha$ be a $(6C+8\delta+1)$--local $C$--tight path. Then the Hausdorff distance between $\alpha$ and any geodesic with the same endpoints as $\alpha$ is at most $2C+4\delta$.
\end{lemma}

\begin{proof}
Let us first show $\alpha\subseteq N_{C+2\delta}(\gamma)$, where $\gamma$ is a geodesic with the same endpoints.
 
 Let $[0,a]$ be the domain of $\alpha$. Let $p=\alpha(t)$ be any point on $\alpha$ at maximal distance from $\gamma$, and let $d=d_Y(\alpha(t),\gamma)$.  Suppose by contradiction that $d > 2C+4\delta$.  Let $R = 3C+4\delta+1/2$, so that $\alpha$ is $2R$--locally $C$--tight.
Using local $C$--tightness, there are points $x=\alpha(t_1)$ and $y=\alpha(t_2)$ satisfying $t_1\in (t-R,t), t_2\in (t,t+R)$ and so that 
$\min\{ d_Y(x,p), d_Y(y,p)\} > 2C+4\delta$.  
Choose a geodesic $[x,y]$.  By local tightness, there is a $p'\in [x,y]$ within $C$ of $p$.  

Let $x', y'\in \gamma$ be chosen so that $d_Y(x,x')$ and $d_Y(y,y')$ are minimal, and let $[x',y']$ be the subsegment of $\gamma$ joining them.  Consider a geodesic quadrilateral $[x,y]\cup [x',y']\cup[x,x']\cup[y,y']$.  The point $p'$ is within $2\delta$ of some $p''$ in one of the other three sides.  It cannot be $[x',y']$, or we would have $d=d_Y(p,\gamma)\leq d_Y(p,p'')\leq C+2\delta$.  Suppose on the other hand that $p''\in [x,x']$ (the argument for $[y,y']$ is identical).  Then $d_Y(p,\gamma) \leq C+2\delta + d - d_Y(p'',x)$.  However $d_Y(p'',x)\geq d_Y(x,p) - C - 2\delta > C + 2\delta$, so $d_Y(p,\gamma)<d$, a contradiction.
 
 Suppose now that there exists $p\in \gamma\setminus N_{C+2\delta}(\alpha)$ (otherwise we are done), and let us show $d_Y(p,\alpha)\leq 2C+4\delta$. Any point on $\alpha$ is $(C+2\delta)$--close to a point on one of the two sides of $p$ in $\gamma$, and hence there exists some $q\in \alpha$ that is $(C+2\delta)$--close to points $p_1,p_2$ on opposite sides of $p$. Since $d_Y(p_1,p_2)\leq 2C+4\delta$ and $\gamma$ is a geodesic, we have $d_Y(p,p_i)\leq C+2\delta$ for some $i$, and hence $d_Y(p,q)\leq 2C+4\delta$, as required.
\end{proof}


\begin{thebibliography}{Bow99b}

\bibitem[AGM09]{agm}
I.~Agol, D.~Groves, and J.~F. Manning.
\newblock Residual finiteness, {QCERF} and fillings of hyperbolic groups.
\newblock {\em Geom. Topol.}, 13(2):1043--1073, 2009.

\bibitem[AGM16]{AgolGrovesManning-alternateMSQT}
I.~Agol, D.~Groves, and J.~F. Manning.
\newblock An alternate proof of {W}ise's {M}alnormal {S}pecial {Q}uotient
  {T}heorem.
\newblock {\em Forum of Mathematics, Pi}, 4, 2016.

\bibitem[Ago13]{VH}
I.~Agol.
\newblock The virtual {H}aken conjecture.
\newblock {\em Doc. Math.}, 18:1045--1087, 2013.
\newblock With an appendix by Agol, Daniel Groves, and Jason Manning.

\bibitem[BF95]{BF:stable}
M.~Bestvina and M.~Feighn.
\newblock Stable actions of groups on real trees.
\newblock {\em Invent. Math.}, 121(2):287--321, 1995.

\bibitem[BH96]{BleilerHodgson96}
S.~A. Bleiler and C.~D. Hodgson.
\newblock Spherical space forms and {D}ehn filling.
\newblock {\em Topology}, 35(3):809--833, 1996.

\bibitem[BH99]{BH}
M.~R. Bridson and A.~Haefliger.
\newblock {\em Metric spaces of non-positive curvature}, volume 319 of {\em
  Grundlehren der Mathematischen Wissenschaften [Fundamental Principles of
  Mathematical Sciences]}.
\newblock Springer-Verlag, Berlin, 1999.

\bibitem[BK05]{BonkKleiner05}
M.~Bonk and B.~Kleiner.
\newblock Quasi-hyperbolic planes in hyperbolic groups.
\newblock {\em Proc. Amer. Math. Soc.}, 133(9):2491--2494 (electronic), 2005.

\bibitem[Bow98]{Bowditch98}
B.~H. Bowditch.
\newblock A topological characterisation of hyperbolic groups.
\newblock {\em J. Amer. Math. Soc.}, 11(3):643--667, 1998.

\bibitem[Bow99a]{Bowditch:boundary_geom_finite}
B.~H. Bowditch.
\newblock Boundaries of geometrically finite groups.
\newblock {\em Math. Z.}, 230(3):509--527, 1999.

\bibitem[Bow99b]{Bowditch:connected_boundary}
B.~H. Bowditch.
\newblock Connectedness properties of limit sets.
\newblock {\em Trans. Amer. Math. Soc.}, 351(9):3673--3686, 1999.

\bibitem[Bow01]{Bowditch:peripheral_splittings}
B.~H. Bowditch.
\newblock Peripheral splittings of groups.
\newblock {\em Trans. Amer. Math. Soc.}, 353(10):4057--4082, 2001.

\bibitem[Bow12]{bowditch12}
B.~H. Bowditch.
\newblock Relatively hyperbolic groups.
\newblock {\em Internat. J. Algebra Comput.}, 22(3):1250016, 66, 2012.

\bibitem[BS07]{BuyaloSchroeder}
S.~Buyalo and V.~Schroeder.
\newblock {\em Elements of asymptotic geometry}.
\newblock EMS Monographs in Mathematics. European Mathematical Society (EMS),
  Z\"urich, 2007.

\bibitem[Can91]{Cannon91}
J.~W. Cannon.
\newblock The theory of negatively curved spaces and groups.
\newblock In {\em Ergodic theory, symbolic dynamics, and hyperbolic spaces
  ({T}rieste, 1989)}, Oxford Sci. Publ., pages 315--369. Oxford Univ. Press,
  New York, 1991.

\bibitem[CC92]{CannonCooper}
J.~W. Cannon and D.~Cooper.
\newblock A characterization of cocompact hyperbolic and finite-volume
  hyperbolic groups in dimension three.
\newblock {\em Trans. Amer. Math. Soc.}, 330(1):419--431, 1992.

\bibitem[Che99]{Chen99}
C.-H. Chen.
\newblock Warped products of metric spaces of curvature bounded from above.
\newblock {\em Trans. Amer. Math. Soc.}, 351(12):4727--4740, 1999.

\bibitem[Cla34]{Claytor34}
S.~Claytor.
\newblock Topological immersion of {P}eanian continua in a spherical surface.
\newblock {\em Ann. of Math. (2)}, 35(4):809--835, 1934.

\bibitem[Cou14]{Coulon14}
R.~Coulon.
\newblock On the geometry of {B}urnside quotients of torsion free hyperbolic
  groups.
\newblock {\em Internat. J. Algebra Comput.}, 24(3):251--345, 2014.

\bibitem[CS98]{CannonSwenson}
J.~W. Cannon and E.~L. Swenson.
\newblock Recognizing constant curvature discrete groups in dimension {$3$}.
\newblock {\em Trans. Amer. Math. Soc.}, 350(2):809--849, 1998.

\bibitem[DG08]{DelzantGromov08}
T.~Delzant and M.~Gromov.
\newblock Courbure m\'esoscopique et th\'eorie de la toute petite
  simplification.
\newblock {\em J. Topol.}, 1(4):804--836, 2008.

\bibitem[DG15]{DG15}
F.~Dahmani and V.~Guirardel.
\newblock Recognizing a relatively hyperbolic group by its {D}ehn fillings.
\newblock {\em arXiv preprint arXiv:1506.03233}, 2015.

\bibitem[DGO]{DGO}
F.~Dahmani, V.~Guirardel, and D.~Osin.
\newblock Hyperbolically embedded subgroups and rotating families in groups
  acting on hyperbolic spaces.
\newblock Preprint, \href{http://arxiv.org/abs/1111.7048}{arXiv:1111.7048v5}.

\bibitem[Dur14]{Durham-thesis}
M.~G. Durham.
\newblock {\em The coarse geometry of the Teichm\"uller metric: A quasiisometry
  model and the actions of finite groups}.
\newblock PhD thesis, University of Illinois at Chicago, 2014.

\bibitem[FM10]{FujMan10}
K.~Fujiwara and J.~F. Manning.
\newblock {${\rm CAT(0)}$} and {${\rm CAT}(-1)$} fillings of hyperbolic
  manifolds.
\newblock {\em J. Differential Geom.}, 85(2):229--269, 2010.

\bibitem[Fre31]{FreudenthalCompactification}
H.~Freudenthal.
\newblock \"{U}ber die {E}nden topologischer {R}\"aume und {G}ruppen.
\newblock {\em Math. Z.}, 33(1):692--713, 1931.

\bibitem[Ger12]{Ger:Floyd}
V.~Gerasimov.
\newblock Floyd maps for relatively hyperbolic groups.
\newblock {\em Geom. Funct. Anal.}, 22(5):1361--1399, 2012.

\bibitem[GM]{GM-splittings}
D.~Groves and J.~F. Manning.
\newblock Dehn fillings and elementary splittings.
\newblock {\em Trans. Amer. Math. Soc.}
\newblock to appear, preprint at
  \href{http://arxiv.org/abs/1506.03831}{arXiv:1506.03831}.

\bibitem[GM08]{rhds}
D.~Groves and J.~F. Manning.
\newblock {Dehn filling in relatively hyperbolic groups}.
\newblock {\em Israel Journal of Mathematics}, 168:317--429, 2008.

\bibitem[GP13]{GerPot:Floyd}
V.~Gerasimov and L.~Potyagailo.
\newblock Quasi-isometric maps and {F}loyd boundaries of relatively hyperbolic
  groups.
\newblock {\em J. Eur. Math. Soc. (JEMS)}, 15(6):2115--2137, 2013.

\bibitem[Gro13]{Groff}
B.~W. Groff.
\newblock Quasi-isometries, boundaries and {JSJ}-decompositions of relatively
  hyperbolic groups.
\newblock {\em J. Topol. Anal.}, 5(4):451--475, 2013.

\bibitem[Hru10]{Hru-relqconv}
G.~C. Hruska.
\newblock Relative hyperbolicity and relative quasiconvexity for countable
  groups.
\newblock {\em Algebr. Geom. Topol.}, 10(3):1807--1856, 2010.

\bibitem[Iva97]{Ivanov97}
S.~V. Ivanov.
\newblock Gromov-{H}ausdorff convergence and volumes of manifolds.
\newblock {\em Algebra i Analiz}, 9(5):65--83, 1997.

\bibitem[Kap07]{kapovichproblems}
M.~Kapovich.
\newblock Problems on boundaries of groups and {K}leinian groups.
\newblock \url{https://www.math.ucdavis.edu/~kapovich/EPR/problems.pdf}, 2007.

\bibitem[KK00]{KK}
M.~Kapovich and B.~Kleiner.
\newblock Hyperbolic groups with low-dimensional boundary.
\newblock {\em Ann. Sci. \'Ecole Norm. Sup. (4)}, 33(5):647--669, 2000.

\bibitem[Mac08]{MackayQuasiArcs}
J.~M. Mackay.
\newblock Existence of quasi-arcs.
\newblock {\em Proc. Amer. Math. Soc.}, 136(11):3975--3981, 2008.

\bibitem[Man15]{ManningBoundary}
J.~F. Manning.
\newblock The {B}owditch boundary of $({G},\mathcal{H})$ when ${G}$ is
  hyperbolic.
\newblock Preprint, \href{http://arxiv.org/abs/1504.03630}{arXiv:1504.03630},
  2015.

\bibitem[Moi77]{Moise77}
E.~E. Moise.
\newblock {\em Geometric topology in dimensions {$2$} and {$3$}}.
\newblock Springer-Verlag, New York-Heidelberg, 1977.
\newblock Graduate Texts in Mathematics, Vol. 47.

\bibitem[MOY12]{MOY:blowingupanddown}
Y.~{Matsuda}, S.-i. {Oguni}, and S.~{Yamagata}.
\newblock {Blowing up and down compacta with geometrically finite convergence
  actions of a group}.
\newblock 2012.
\newblock Preprint, \href{http://arxiv.org/abs/1201.6104}{arXiv:1201.6104}.

\bibitem[MS]{moshersageev}
L.~Mosher and M.~Sageev.
\newblock Nonmanifold hyperbolic groups of high cohomological dimension.
\newblock Preprint, available at \url{http://andromeda.rutgers.edu/~mosher/}.

\bibitem[MS89]{MartinSkora}
G.~J. Martin and R.~K. Skora.
\newblock Group actions of the {$2$}-sphere.
\newblock {\em Amer. J. Math.}, 111(3):387--402, 1989.

\bibitem[MS11]{MS1111.2499}
J.~Mackay and A.~Sisto.
\newblock Quasi-hyperbolic planes in relatively hyperbolic groups.
\newblock 2011.
\newblock Preprint, \href{http://arxiv.org/abs/1111.2499}{arXiv:1111.2499}.

\bibitem[Osi06]{osin:relhypbook}
D.~V. Osin.
\newblock Relatively hyperbolic groups: intrinsic geometry, algebraic
  properties, and algorithmic problems.
\newblock {\em Mem. Amer. Math. Soc.}, 179(843):vi+100, 2006.

\bibitem[Osi07]{osin:peripheral}
D.~V. Osin.
\newblock Peripheral fillings of relatively hyperbolic groups.
\newblock {\em Invent. Math.}, 167(2):295--326, 2007.

\bibitem[Ric63]{RichardsClassification}
I.~Richards.
\newblock On the classification of noncompact surfaces.
\newblock {\em Trans. Amer. Math. Soc.}, 106:259--269, 1963.

\bibitem[Swa96]{Sw:no_cut_points}
G.~A. Swarup.
\newblock On the cut point conjecture.
\newblock {\em Electron. Res. Announc. Amer. Math. Soc.}, 2(2):98--100
  (electronic), 1996.

\bibitem[Thu80]{thurston:notes}
W.~P. Thurston.
\newblock Geometry and topology of three-manifolds.
\newblock {P}rinceton lecture notes available at
  \url{http://www.msri.org/publications/books/gt3m/}, 1980.

\bibitem[Tra13]{Tran:comparison_boundaries}
H.~C. Tran.
\newblock Relations between various boundaries of relatively hyperbolic groups.
\newblock {\em Internat. J. Algebra Comput.}, 23(7):1551--1572, 2013.

\bibitem[V{\"a}i05]{Vai:hyperbolic}
J.~V{\"a}is{\"a}l{\"a}.
\newblock Gromov hyperbolic spaces.
\newblock {\em Expo. Math.}, 23(3):187--231, 2005.

\bibitem[Yam04]{Yaman}
A.~Yaman.
\newblock A topological characterisation of relatively hyperbolic groups.
\newblock {\em J. Reine Angew. Math.}, 566:41--89, 2004.

\end{thebibliography}
\end{document}